\documentclass{article}
\usepackage[a4paper,top=4cm,bottom=4cm,left=2.2cm,right=2.2cm,bindingoffset=5mm]{geometry}
\usepackage{xcolor}

\usepackage[hyphens]{url}
\usepackage[breaklinks=true,colorlinks,linkcolor={blue},citecolor={red},urlcolor={purple},]{hyperref}

\usepackage{todonotes}
\usepackage{comment}

\newcommand\Item[1][]{%
  \ifx\relax#1\relax  \item \else \item[#1] \fi
  \abovedisplayskip=0pt\abovedisplayshortskip=0pt~\vspace*{-\baselineskip}}

\usepackage[font=small,format=default,indention=2em,labelsep = period]{caption} 
\usepackage[subrefformat=parens,labelformat=parens]{subfig}
\usepackage{enumitem}
\usepackage{multirow}
\usepackage{pict2e,overpic}
\usepackage{pgfplots}

\usepackage{amssymb,bm,amsmath,amsthm,accents}
\usepackage{graphicx,mathtools}	
\usetikzlibrary{calc}
\usepackage{esvect}

\definecolor{bordeaux}{rgb}{0.89, 0.0, 0.13}
\definecolor{cadmiumgreen}{rgb}{0.0, 0.42, 0.24}

\newcommand{\cateB}[1]{\textcolor{blue}{#1}}     
\newcommand{\disp}{\displaystyle}     
\newcommand{\R}{\mathbb R}

\newcommand{\de}{\partial}        
\newcommand{\Sr}{\Sigma^{r}}
\newcommand{\Sl}{\Sigma^{\ell}}
\newcommand{\Sret}{\Sigma^{R; t_o, t_a}}

\newcommand{\ba}{\begin{align}}
\newcommand{\ea}{\end{align}}

\def\rhom{\rho^{\mathrm{max}}}
\def\vmax{V^{\mathrm{max}}}
\def\qmax{Q^{\mathrm{max}}}

\def\wl{w_{L}}
\def\wr{w_{R}}

\def\vr{V_\rho}

\def\umeno{U_{-}}
\def\upiu{U_{+}}
\def\luno{\lambda_{1}}
\def\ldue{\lambda_{2}}

\def\duno{d_{1}}
\def\ddue{d_{2}}
\def\stre{s_{3}}
\def\rettar{\mathrm{r}}
\def\rettas{\mathrm{z}}
\newcommand{\dem}[2]{d(#1,#2)}
\newcommand{\supp}[2]{s(#1,#2)}
\def\pos{\mathcal{P}}
\def\neg{\mathcal{N}}
\def\dex{\partial_{x}}
\def\det{\partial_{t}}

\def\edge{\mathcal{I}}
\def\vert{\mathcal{J}}

\def\road{s}

\def\quno{q_{1}}
\def\qdue{q_{2}}

\def\rhot{\widetilde \rho}

\def\umeno{U^{-}}

\def\rhomeno{\rho^{-}}
\def\wrhomeno{\widetilde\rho^{-}}

\def\wmeno{w^{-}}

\def\upiu{U^{+}}

\def\rhopiu{\rho^{+}}

\def\wpiu{w^{+}}
\def\vpiu{v^{+}}
\def\umorto{U^{\dagger}}

\def\rhomorto{\rho^{\dagger}}
\def\wrhomorto{\widetilde\rho^{\dagger}}

\def\wmorto{w^{\dagger}}
\def\hu{\mathring U}
\def\hw{\mathring w}
\def\hrho{\mathring\rho}
\def\hduno{\mathring d_{1}}
\def\hddue{\mathring d_{2}}
\def\hstre{\mathring s_{3}}
\def\hsq{\mathring s_{4}}
\def\rsolvG{APRSOM}

\def\ondarho{\tilde\rho}
\def\ondaq{\tilde q}
\def\ondaw{\tilde w}
\def\ondaduno{\tilde\duno}
\def\ondaddue{\tilde\ddue}

\def\uS{\hat U}
\def\rhoS{\hat\rho}
\def\qS{\hat q}
\def\wS{\hat w}

\def\uV{(U_{1},\dots,U_{n},U_{n+1},\dots,U_{n+m})}
\def\uVS{(\uS_{1},\dots,\uS_{n},\uS_{n+1},\dots,\uS_{n+m})}
\def\qv{\mathcal{Q}}
\def\uVh{(\mathring U_{1},\dots,\mathring U_{n},\mathring U_{n+1},\dots,\mathring U_{n+m})}
\def\uVd{(U_{1}, U_{2}, U_{3}, U_{4})}
\def\uVhd{(\mathring U_{1},\mathring U_{2},\mathring U_{3},\mathring U_{4})}

\def\umenoS{\hat U}
\def\rhomenoS{\hat\rho}
\def\qmenoS{\hat q}
\def\wmenoS{\hat w}

\def\upiuS{\hat U}
\def\rhopiuS{\hat\rho}
\def\qpiuS{\hat q}
\def\wpiuS{\hat w}

\def\cons{Y}
\def\conszero{Y_{0}}
\def\puno{p_{1}}
\def\pdue{p_{2}}

\def\huno{h_{1}}
\def\hdue{h_{2}}
\def\htre{h_{3}}

\def\pb{\widehat{p}}

\def\sq{s_{4}}
\newcommand{\aij}[2]{\alpha_{#1#2}}

\def\di{d_{i}}
\def\qn{q_{n}}

\newcommand{\hbv}[1]{\mathsf{h}_{\mathrm{step}_{#1}}}
\newcommand{\ind}[1]{\ell_{#1}}
\def\insInd{I}
\def\indice{S}

\def\entropy{\eta}
\def\entFlux{\Phi}

\def\lunoloc{L^{1}_{\mathrm{loc}}}
\def\rsolv{\mathcal{RS}}

\def\tv{\mathrm{TV}}
\def\tb{\bar t}
\def\hh{\bar h}

\def\qstar{Q^{*}}
\def\qstar{q^{*}}
\def\at{\aij{3}{1}}
\def\bt{\aij{3}{2}}
\def\aq{\aij{4}{1}}
\def\bq{\aij{4}{2}}
\def\omegaInc{\Omega_{\mathrm{inc}}}
\def\omegaOut{\Omega_{\mathrm{out}}}

\def\qv{\mathcal{Q}}

\let\oldparagraph=\paragraph
\renewcommand\paragraph[1]{\oldparagraph{#1.}}

\theoremstyle{plain}
\newtheorem{defn}{Definition}[section]
\newtheorem{theorem}[defn]{Theorem}
\newtheorem{prop}[defn]{Proposition}
\newtheorem{lemma}[defn]{Lemma}
\newtheorem{corollario}[defn]{Corollary}

\newtheorem{remark}[defn]{Remark}
\numberwithin{equation}{section}

\numberwithin{defn}{section}
\usepackage{upref}

\title{\Large\textbf{Adapting Priority Riemann Solver for GSOM on road networks}}

\author{\normalsize{Caterina Balzotti}\thanks{ELEM Biotech S.L., Barcelona, Spain (\href{mailto:cbalzotti@elem.bio}{cbalzotti@elem.bio})
}
\and {\setcounter{footnote}{3}\normalsize{Roberta Bianchini}\thanks{Istituto per le Applicazioni del Calcolo ``M.\ Picone'', Consiglio Nazionale delle Ricerche, Rome, Italy (\href{mailto:roberta.bianchini@cnr.it}{roberta.bianchini@cnr.it}). }}
\and {\setcounter{footnote}{3}\normalsize{Maya Briani}\thanks{Istituto per le Applicazioni del Calcolo ``M.\ Picone'', Consiglio Nazionale delle Ricerche, Rome, Italy (\href{mailto:maya.briani@cnr.it}{maya.briani@cnr.it}). }}
\and {\setcounter{footnote}{2}\normalsize{Benedetto Piccoli}\thanks{Department of Mathematical Sciences, Rutgers University, Camden, USA (\href{mailto:piccoli@camden.rutgers.edu}{piccoli@camden.rutgers.edu}).}}
}
\date{\today}

\begin{document}

\allowdisplaybreaks

\maketitle

\begin{abstract}
\noindent In this paper, we present an extension of the Generic Second Order Models (GSOM) for traffic flow on road networks. We define a Riemann solver at the junction based on a priority rule and provide an iterative algorithm to construct solutions at junctions with $n$ incoming and $m$ outgoing roads. 

\noindent The logic underlying our solver is as follows: the flow is maximized while respecting the priority rule, which can be adjusted if the supply of an outgoing road exceeds the demand of a higher-priority incoming road. Approximate solutions for Cauchy problems are constructed using wave-front tracking.

\noindent We establish bounds on the total variation of waves interacting with the junction and present explicit calculations for junctions with two incoming and two outgoing roads. A key novelty of this work is the detailed analysis of returning waves - waves generated at the junction that return to the junction after interacting along the roads - which, in contrast to first-order models such as LWR, can increase flux variation.
\end{abstract}

\begin{description}
\item[\textbf{Keywords.}] Second order traffic models; Priority rule; Networks;  Cauchy problem; Wave-front tracking; Returning wave.
\item[\textbf{Mathematics Subject Classification.}] 90B20; 35L65.
\end{description}

\section{Introduction}\label{sec:intro}
This paper focuses on macroscopic second-order traffic models on road networks. We consider the Generic Second Order Models (GSOM) \cite{FanNHM14,FanSunPiccoliSeiboldWork2017,LebacqueMammarHajSalem2007}, 
which are a family of traffic models described by a first-order scalar conservation law for the density of vehicles $\rho$ combined with an advection equation of a certain property of drivers $w$ linked to the density of vehicles by a speed function $v=V(\rho,w)$. Through the variable $w$ it is possible to take into account different driving behaviors. In fact, $w$ parametrizes the family of fundamental diagrams $Q(\rho,w)=\rho V(\rho,w)$, whose curves correspond to different driving aptitudes. 
The model equations are given by:
\begin{equation*}
	\begin{dcases}
		\det\rho+\dex(\rho v) = 0\\
		\det w+v\ \dex(w) = 0\\
	\end{dcases}.
\end{equation*}

Traffic models on networks have been widely studied in recent years, and authors have considered several traffic scenarios proposing a rich amount of alternative models at junctions. For instance, the first order Lighthill-Whitham-Richards (LWR) model \cite{LighthillWhitham1955,Richards1956} has been extended to road networks in several papers, see \cite{bressan2014flows,dellemonache2018CMS,GaravelloPiccoli2006,garavello2006AIHP,garavello2016models,HoldenRisebro1995,lebacque2005first}, as well as the second order Aw-Rascle-Zhang (ARZ) model \cite{AwRascle2000,Zhang2002}, see \cite{GaravelloPiccoli2006AwRascle, GoettlichHertyMoutariWeissen2021, herty2006NHM,herty2006SIAM}.
Most traffic models on networks rely on solving the Riemann problem at junctions, a Cauchy problem with constant initial data on each connected road. Unique solutions require coupling conditions that differentiate models. Most models share two key assumptions: that flow through the junction is conserved and that waves generated at junctions have a negative velocity on incoming roads and a positive velocity on outgoing roads. The second assumption is necessary to guarantee that boundary-value problems are correctly solved on each road, and that the conservation of cars through the junction is guaranteed. Other common assumptions include maximizing flow through the junction and allocating vehicles on outgoing roads based on a distribution matrix. 

The solution to the Riemann problem for the system of conservation laws arising in the GSOM framework is defined by three states: left, middle, and right, connected by a shock or rarefaction wave and a contact discontinuity, respectively. We say that a $\rho$-wave is generated between the left and the middle state, and a $w$-wave is generated between the middle and the right state. Along $\rho$-waves the variable $w$ is constant while the density changes, and in $w$-waves $\rho$ and $w$ both change, and the velocity $V$ is conserved. Moreover, $\rho$-waves can travel with positive or negative speed while the speed of $w$-waves is always non-negative. 
We then consider a Riemann solver designed for the GSOM family at a junction.
The approach involves solving a left-half Riemann problem (where waves have negative velocity) at the nodes for incoming roads, and a right-half Riemann problem (where waves have positive velocity) for outgoing roads. This defines the region of admissible states to ensure that waves do not propagate into the junction. Thus, in our case, only $\rho$ waves can leave the junction toward an incoming road, and both $\rho$ waves and $w$ waves coming from the junction can enter an outgoing road. The definition of admissible solutions always excludes the non-physical case that there can be a jump with zero speed at the junction. Together with the definition of admissible states, we assume maximization of the flow and conservation of $\rho$ and  $y=\rho w$ through the junction. The determination of a unique solution is achieved by introducing a priority rule on the incoming roads and a distribution of vehicles on the outgoing roads according to a proper distribution matrix. For this, we propose a new logic which is a generalization of the approach proposed in \cite{dellemonache2018CMS} for first-order models: the flow is maximized respecting the priority rule, but the latter can be modified if the outgoing road supply exceeds the demand of the road with higher priority. In Section \ref{sec:NinM} we define the Adapting Priority Riemann Solver for Second Order Models (APRSOM) that is an iterative algorithm able to construct the solution to generic junctions with $n$ incoming and $m$ outgoing roads, computing the incoming fluxes at the junction step by step.

Once the Riemann Solver is defined, we pass to tackle 
Cauchy problem on networks associated to GSOM. In conservative form, the equations read
\begin{equation*}
	\begin{dcases}
		\det\rho_{\road}+\dex(\rho_{\road} v_{\road}) = 0\\
		\det y_{\road}+\dex(y_{\road} v_{\road}) = 0\\
		(\rho_{\road}(x,0),y_{\road}(x,0)) = (\rho_{r,0},y_{r,0}) 
	\end{dcases}
\end{equation*}
for $r=1,\ldots,n+m$ roads connected at a junction and for the initial data $(\rho_{r,0},y_{r,0}) $ with bounded variation. To prove the existence of weak solutions to the Cauchy problem, a typical approach relies on the Wave-Front-Tracking (WFT) method. The latter requires the study of Riemann problems along the roads and at junctions. Following the strategy originally proposed in \cite{garavello2006AIHP} and extended in \cite{dellemonache2018CMS}, we introduce four properties that a Riemann Solver must satisfy in order to guarantee bounds on the total variation of the flux $Q$ and of the variable $w$ for waves that interact with the junction. More precisely, these properties allow to estimate the increase in the total variation of $Q$ and $w$ due to the interaction with the junction. 
A special study has been made for \textit{returning waves} which, unlike the first-order LWR model, can lead to an increase of the flux at the junction.
We give a precise definition of returning waves and provide a general estimate of the flux variation at the junction, without distinguish between incoming and outgoing roads. However, refined estimates show that on an incoming road, a returning wave interacting with a $w$-wave along the road can cause an increase in flux at the junction and requires a specific evaluation. On the other hand, for an outgoing road, it can be shown that a returning wave always causes a decrease in flow at the junction and therefore does not require a specific estimate.
After these fine estimates are achieved, we follow the general strategy of \cite{dellemonache2018CMS}
to prove existence of solutions.

\medskip 

The paper is organized as follows. In section \ref{sec:GSOM}, we introduce the basic definitions
of the theory of Second Order traffic Models (GSOM) on a single road, and extend them to a network in section \ref{sec:GSOMnetwork}. In section \ref{sec:NinM} we define our Adaptive Priority
Riemann solver for a generic junction with $n$ incoming roads and $m$ outgoing roads, and in section \ref{sec:2in1} we illustrate the algorithm for the particular case of a merge (a network composed of two incoming roads and one outgoing road). Section \ref{sec:TVbounds} is devoted to bounds on the total variation of the flux over the entire $n$ in $m$ network, and in particular, section \ref{sec:fluxRet} studies the variation of the flux due to returning waves at the junction.  
In section \ref{sec:theoEx} we prove the existence of solutions to Cauchy problems.
Appendix A collects the proof of the main theorem of the paper.

\section{Generic Second Order traffic Models}\label{sec:GSOM}
In this section we present the traffic model and we collect the main definitions used throughout the work. We deal with the Generic Second Order Models (GSOM) \cite{LebacqueMammarHajSalem2007}, a family of macroscopic traffic models which are described by a first order Lighthill-Whitham-Richards (LWR) model \cite{LighthillWhitham1955,Richards1956} with variable fundamental diagrams. Such models are defined by
\begin{align}
	\begin{split}
		&\begin{cases}
			\det\rho+\dex(\rho v) = 0\\
			\det w+v\dex w = 0\\
		\end{cases}\\
		&\disp \text{with } v=V(\rho,w),
	\end{split}
	\label{eq:GSOM1}
\end{align}
where $\rho(x,t)$, $v(x,t)$ and $w(x,t)$ represent the density, the speed and a property of vehicles advected by the flow, respectively, and $V$ is a specific velocity function. The first equation of \eqref{eq:GSOM1} is the conservation of vehicles, the second one is the advection of the attribute of drivers, which defines their driving aptitude by means of different fundamental diagrams. Indeed, the variable $w$ identifies the flux curve $Q(\rho,w)$ and thus the speed of vehicles $V(\rho,w)=Q(\rho,w)/\rho$ which characterizes the behavior of drivers. 
System \eqref{eq:GSOM1} is written in conservative form as
\begin{align}
	\begin{split}
		&\begin{cases}
			\det\rho+\dex(\rho v) = 0\\
			\det y+\dex(y v) = 0\\
		\end{cases}\\
		&\disp \text{with } v = V\Big(\rho,\frac{y}{\rho}\Big),
	\end{split}
	\label{eq:GSOM2}
\end{align}
where $y=\rho w$
denotes the total property of vehicles. 

The flux function $Q(\rho,w)$ and the velocity function $V(\rho,w)=Q(\rho,w)/\rho$ are assumed to satisfy the following properties.
\begin{enumerate}[label=(H\arabic*),ref=\textup{(H\arabic*)}]
	\item\label{q1} $Q(0,w) = 0$ and $Q(\rhom(w),w) = 0$ for each $w\in[\wl,\wr]$, where $\rhom(w)$ is the maximal density of vehicles for $Q(\cdot,w)$ and $[\wl,\wr]$ is the domain of $w$, for suitable $\wl$ and $\wr$.
	\item\label{q2} $Q(\rho,w)$ is strictly concave with respect to $\rho$, i.e. $\frac{\de^{2}Q}{\de\rho^{2}}<0$.
	\item\label{q3} $Q(\rho,w)$ is non-decreasing with respect to $w$, i.e. $Q_{w}\geq0$.
	\item\label{v1} $V(\rho,w)\geq0$ for each $\rho$ and $w$.
	\item\label{v2} $V(\rho,w)$ is strictly decreasing with respect to $\rho$, i.e. $V_{\rho}<0$ for each $w$.
	\item\label{v3} $V(\rho,w)$ is non-decreasing with respect to $w$, i.e. $V_{w}\geq0$.
\end{enumerate}
Note that property \ref{v2} is a consequence of \ref{q2}. Indeed, for $\rho\neq0$ 
\begin{equation}\label{def:f}
	\vr(\rho,w)=\frac{\rho Q_{\rho}(\rho,w)-Q(\rho,w)}{\rho^{2}}=: \frac{f(\rho, w)}{\rho^{2}}<0 \quad\text{for each $w$}
\end{equation}
since $f(0, w)=0$ for all $w$ and $\partial_\rho f(\rho, w)=-\frac{\de^{2} Q}{\de\rho^{2}}>0$ by \ref{q2}. Property \ref{v3} follows by \ref{q3} trivially. Properties \ref{q1} and \ref{q2} imply that the flux curve $Q(\cdot,w)$ has a unique point of maximum for any $w$. We denote by $\sigma(w)$ the critical density, i.e. the density value where the flux attains its maximum $\qmax(w)$. 
Moreover, for any $\rho$ there exists a unique $\rhot(w)$ such that $Q(\rho,w)=Q(\rhot(w),w)$.

Denoting $\cons=(\rho,y)^{T}$ and $F(\cons) = (\rho v, yv)^{T}$, the GSOM model \eqref{eq:GSOM2} can be rewritten as
\begin{equation}\label{eq:GSOMsist}
	\det \cons+\dex F(\cons) = 0.
\end{equation}
We are therefore interested in describing the solution to Cauchy problem:
\begin{equation}\label{eq:cauchySist}
	\begin{cases}
		\det \cons+\dex F(\cons) = 0,\\
		\cons(x,0) = \cons_{0}(x).
	\end{cases}
\end{equation}
From the standard theory of conservation laws \cite{Bressan2000,Dafermos2000}, it is natural to work with weak solutions, defined below.
\begin{defn}\label{def:weak}
Let $\conszero\in \lunoloc(\R,\R^{2})$ and $T>0$. A function $\cons:\R\times[0,T]\to\R^{2}$ is a \emph{weak solution} to \eqref{eq:cauchySist} if $\cons$ is continuous as a function from $[0,T]$ into $\lunoloc$ and if, for every test function $\varphi\in C^{1}$ with compact support in the set $\R\times(-\infty,T)$, it holds
\begin{equation*}
	\int_{0}^{T}\int_{\R}(\cons\det\varphi+F(\cons)\dex\varphi)dxdt+\int_{\R}\conszero(x)\varphi(x,0)dx=0.
\end{equation*} 
\end{defn}
Since there are infinitely many weak solutions to \eqref{eq:cauchySist}, we introduce the classical selection principle, which leads to the entropy solution \cite{Bressan2000}, the physically relevant one.
\begin{defn}
A $C^{1}$ function $\entropy:\R^{2}\to\R$ is an \emph{entropy} associated to \eqref{eq:GSOMsist} if it is convex and there exists a $C^{1}$ function $\entFlux:\R^{2}\to\R$ such that
\begin{equation}\label{eq:entropiaFlusso}
	D\entropy(\cons)DF(\cons) = D\entFlux(\cons)
\end{equation}
for every $\cons\in\R^{2}$. The function $\entFlux$ is called an \emph{entropy flux} for $\entropy$. The pair $(\entropy,\entFlux)$ is called \emph{entropy-entropy flux} pair.
\end{defn}
\begin{defn}\label{def:entropy}
A weak solution $\cons=\cons(x,t)$ to \eqref{eq:cauchySist} is called \emph{entropy admissible} if, for every $C^{1}$ test function $\varphi\geq0$ with compact support in $\R\times[0,T)$ and for every entropy-entropy flux pair $(\entropy,\entFlux)$, it holds
\begin{equation}\label{eq:entropyDis}
	\int_{0}^{T}\int_{\R}(\entropy(\cons)\varphi_{t}+\entFlux(\cons)\varphi_{x})dxdt\geq0.
\end{equation}
\end{defn}


\bigskip

Before introducing the problem on road networks, we discuss the Riemann problem on a single road for GSOM. We then consider system \eqref{eq:GSOMsist} together with piecewise constant initial data which has a single discontinuity in the domain of interest.
The solution to Riemann problems is given by a combination of elementary waves, i.e. shocks, rarefaction waves and contact discontinuities.
The study of the Jacobian of $F(\cons)$ shows that system \eqref{eq:GSOMsist} is strictly hyperbolic with two distinct eigenvalues for $\rho\neq0$
\begin{align}
	\luno(\cons) &= V(\cons)+\rho V_{\rho}(\cons)\label{eq:l1},\\
	\ldue(\cons) &= V(\cons)\label{eq:l2}
\end{align} 
which coincide if and only if $\rho=0$. The eigenvectors associated to the eigenvalues are $\gamma_{1}(\cons) = (\rho,y)^{T} \qquad\text{and}\qquad \gamma_{2}(\cons) = (-V_{y},V_{\rho})^{T}$,  and thus $\luno$ is genuinely nonlinear ($\nabla\luno\cdot\gamma_{1}\neq0$) and $\ldue$ is linearly degenerate ($\nabla\ldue\cdot\gamma_{2}=0$). The waves associated to the first eigenvalue
$\lambda_1$ are then shock or rarefaction waves, while those associated to $\lambda_2$ are contact discontinuities.
The Riemann invariants are 
\begin{equation}\label{eq:riemann-inv}
	z_{1}(\cons)=y/\rho \qquad\text{and}\qquad z_{2}(\cons)=V(\cons).
\end{equation}
From now on we will use the $(\rho,w)$ variables. Thus, we set
$$
 U=(\rho,w), \quad z_1(U)=w, \quad z_2(U)=V(U).
$$
The first eigenvalue in \eqref{eq:l1}, in $(\rho,w)$ variables is $\luno(\rho,w)=V(\rho,w)+\rho\vr(\rho,w)=Q_{\rho}(\rho,w)$ and by properties \ref{q1} and \ref{q2}  it is such that $\luno\geq0$ for $\rho\leq\sigma(w)$ and $\luno<0$ for $\rho>\sigma(w)$. Hence, for each $w$ the 1-shocks and 1-rarefaction waves have non-negative speed for $\rho\leq\sigma(w)$ and negative speed for $\rho>\sigma(w)$. The second eigenvalue in \eqref{eq:l2} verifies $\ldue(\rho,w)=V(\rho,w)\geq0$ by definition of $V$, thus the speed of the 2-contact discontinuities is always non-negative.

Given two generic left and right states $U^-=(\rho^-,w^-)$ and $U^+=(\rho^+,w^+)$, the solution to Riemann problem is composed of three states: left $U^-=(\rho^-,w^-)$, middle $U^*=(\rho^*,w^*)$ and right $U^+=(\rho^+,w^+)$.
\begin{defn}[$\rho$-waves and $w$-waves]
 We will refer to \emph{$\rho$-waves} as the $1$-wave between the left and middle state $(U^-,U^*)$ ($w$ is conserved and $\rho$ changes), and to \emph{$w$-waves} as the $2$-wave between the middle and right state $(U^*,U^+)$ ($\rho$ and $w$ change and the velocity $V$ is conserved).   
\end{defn}
We have the following:
\begin{itemize}
\item between $U^-$ e $U^*$ the $\rho$-wave is such that $ w^-=w^* $
and, if $\rho^-<\rho^*$ the wave is a shock with speed 
$$
s = \frac{Q(\rho^*,w^*)-Q(\rho^-,w^-)}{\rho^*-\rho^-},
$$
if $\rho^->\rho^*$ the wave is a rarefaction such that $\lambda_1(U^-)<\lambda_1(U^*)$; 
\item between $U^*$ and $U^+$ the $w$-wave travels with velocity $V$ such that $V(\rho^*,w^*)=V(\rho^+,w^+)$.

\end{itemize}
\begin{lemma}\label{lem:orderV}
Suppose that $V(\rho^*, w^*)=V(\rho^+, w^+)$. Then, the following holds
\begin{equation}
\rho^*-\rho^+= - \frac{\de_w V (\rho^+, \tilde w)}{\de_\rho V(\tilde \rho, w^*)} (w^*-w^+).
\end{equation}
In particular, 
\begin{itemize}
\item if $w^*< w^+$, then $\rho^*<\rho^+$; 
\item if $w^*\ge  w^+$, then $\rho^*\ge \rho^+$.
\end{itemize}
\end{lemma}
\begin{proof}
By means of the Mean Value Theorem, one has that
\begin{align*}
0=V(\rho^*, w^*)-V(\rho^+, w^+)&=V(\rho^*, w^*)-V(\rho^+, w^*)+V(\rho^+, w^*) -V(\rho^+, w^+)\\
& = \de_\rho V(\tilde \rho, w^*) (\rho^*-\rho^+) + \de_w V (\rho^+, \tilde w) (w^*-w^*).
\end{align*}
The proof follows since $ \de_\rho V(\tilde \rho, w^*)<0$ by \ref{v2} and $ \de_w V (\rho^+, \tilde w)\ge 0$ by \ref{v3}.
\end{proof}
\begin{lemma}\label{lem:monotone}
Suppose that $V(\rho^-, w^-)=V(\rho^*, w^*)$ and $V(\rho^+, w^+)=V(\hat \rho, \hat w)$. 
It holds that $\rho^-< \rho^+$ if and only if $\rho^*< \hat \rho$. 
\end{lemma}
\begin{proof}
Being $\rho^-< \rho^+$ and $V_\rho (\cdot, \cdot)<0$, then $V(\rho^-, \cdot)> V(\rho^+, \cdot),$ yielding $V(\rho^*, \cdot)> V(\hat \rho, \cdot)$. The proof ends using again the monotonicity of $V(\rho, \cdot)$.
\end{proof}

\begin{lemma}\label{lemma:VelOndeRho}
When both a $\rho$-wave and a $w$-wave travel with positive speed and the $\rho$-wave is behind the $w$-wave, they cannot interact ($\rho$-waves are slower than $w$-waves).
\end{lemma}
\begin{proof}
Suppose to have a $\rho$-wave and a $w$-wave traveling with positive speed one behind the other, and separating the three states $U_1=(\rho_1,w_1)$, $U_2=(\rho_2,w_2)$, $U_3=(\rho_3,w_3)$. By definition of $\rho$-wave $w_1=w_2$ and by definition of $w$-wave $V(\rho_2,w_2=w_1)=V(\rho_3,w_3)$. We want to prove that the speed of the $\rho$-wave is less that the speed of the $w$-wave $\lambda_2(\rho_3,w_3) = V(\rho_3,w_3)=V(\rho_2,w_2=w_1)$.
\begin{itemize}
\item If the $\rho$-wave is a shock with positive speed, then $\rho_1 <\rho_2$ and by contradiction
\begin{equation*}
\mbox{shock speed} = \frac{\rho_2V(\rho_2,w_1)-\rho_1V(\rho_1,w_1)}{\rho_2-\rho_1} > V(\rho_2,w_1)=V(\rho_3,w_3) 
\end{equation*}
if and only if $V(\rho_1,w_1) < V(\rho_2, w_1)$. Since $V$ is decreasing in $\rho$ by hypothesis \ref{v2}, it would imply $\rho_1 >\rho_2$ which contradicts the hypothesis.
\item If the $\rho$-wave is a rarefaction, then since $V$ is decreasing in $\rho$
$$\lambda_1(\rho_2,w_1) =  V(\rho_2,w_1)+\rho_2\partial_\rho V = V(\rho_3,w_3) +\rho_2\partial_\rho V < V(\rho_3,w_3), $$
whence the thesis.
\end{itemize}
However, note that a $w$-wave traveling behind a $\rho$-wave can interact with it.
\end{proof}

\section{Road network}\label{sec:GSOMnetwork}
We recall now the main definitions concerning traffic models on a road network, and we refer to \cite{dellemonache2018CMS,GaravelloPiccoli2006,garavello2006AIHP,HoldenRisebro1995} for further details. Consider a junction $J$ with $n$ incoming and $m$ outgoing roads $I_{\road}=[a_{\road},b_{\road}]\subset\R$, $\road = 1,\dots,n+m$, possibly with $a_{\road}=-\infty$ and $b_{\road}=+\infty$. We define a network as a couple $(\edge,\vert)$ where $\edge$ is a finite collection of roads $I_{\road}$, and $\vert$ is a finite collection of junctions $J$.

\begin{defn}\label{def:sol-net}
A collection of functions $(\rho_{\road},y_{\road})\in C([0,+\infty);\lunoloc(I_{\road})^{2})$, $\road=1,\dots,n+m$, is a weak solution at $J$ if
\begin{itemize}
	\item For every $\road\in\{1,\dots,n+m\}$ the couple $(\rho_{\road},y_{\road})$ is an entropy admissible solution to \eqref{eq:GSOM2} on the road $I_{\road}$ in the sense of Definition \ref{def:entropy}.
	\item For every $\road\in\{1,\dots,n+m\}$ and for a.e. $t>0$ the function $x\mapsto (\rho_{\road}(x,t),y_{\road}(x,t))$ has a version with bounded total variation.
	\item For a.e. $t>0$, it holds 
	\begin{equation*}
		\sum_{i=1}^{n}Q(\rho_{i}(b_{i}-,t),w_{i}(b_{i}-,t))=\sum_{j=n+1}^{n+m}Q(\rho_{j}(a_{j}+,t),w_{j}(a_{j}+,t))
	\end{equation*}
	where $w_{\road}=y_{\road}/\rho_{\road}$ and $(\rho_{\road},y_{\road})$ is the version with bounded total variation of the previous point.
\end{itemize}
\end{defn}

We now focus on the Riemann problem at the junction: on each road $I_{\road}$, $\road=1,\dots,n+m$, we solve
\begin{equation}\label{eq:GSOMrete}
	\begin{dcases}
		\det\rho_{\road}+\dex(\rho_{\road} v_{\road}) = 0\\
		\det y_{\road}+\dex(y_{\road} v_{\road}) = 0\\
		(\rho_{\road}(x,0),y_{\road}(x,0)) = 
		\begin{cases}
			(\rho^{-},y^{-}) &\quad\text{for $x<x_{0}$}\\
			(\rho^{+},y^{+}) &\quad\text{for $x>x_{0}$,}
		\end{cases}
	\end{dcases}
\end{equation}
with $v_{\road} = V\left(\rho_{\road},y_{\road}/\rho_{\road}\right)$ and where either the left or right state is known. Depending on whether the road is incoming or outgoing, we have the following possibilities:
\begin{itemize}
	\item If $I_{i}$ is an incoming road at the junction then $x_{0}=b_{i}$ and only the left state $(\rho^{-},y^{-})$ is known. In this case we look for weak solutions of \eqref{eq:GSOMrete} such that the waves have non-positive speed.
	\item If $I_{j}$ is an outgoing road at the junction then $x_{0}=a_{j}$ and only the right state $(\rho^{+},y^{+})$ is known. In this case we look for weak solutions of \eqref{eq:GSOMrete} such that the waves have non-negative speed.
\end{itemize}

As mentioned above, we work with the couple of variables $(\rho,w)$. Occasionally, we will adopt the shortened notation
\begin{align}\label{eq:q-small}
q:=Q(\rho, w) \quad \text{for any} \quad U=(\rho, w).
\end{align}
\begin{defn}\label{def:rsolv}
A Riemann solver $\rsolv$ is a function
\begin{align*}
	\rsolv:([0,\rhom]\times[\wl,\wr])^{n+m}&\longrightarrow([0,\rhom]\times[\wl,\wr])^{n+m}\\
	\uV&\longmapsto\uVS
\end{align*}
such that

\begin{enumerate}
	\item $\disp\sum_{i=1}^{n}\qS_{i}=\sum_{j=n+1}^{n+m}\qS_{j}$, with $\qS_{i}=Q(\rhoS_{i},\wS_{i})$ and $\qS_{j}=Q(\rhoS_{j},\wS_{j})$.
	\item For every $i=1,\dots,n$ the Riemann problem \eqref{eq:GSOMrete} has initial datum
	\begin{equation*}
	(\rho_{i}(x,0),y_{i}(x,0)) = \begin{cases}
		(\rho_{i},\rho_{i}w_{i}) &\quad\text{for $x<b_{i}$}\\
		(\rhoS_{i},\rhoS_{i}\wS_{i}) &\quad\text{for $x>b_{i}$,}
		\end{cases}
	\end{equation*}
	and is solved with waves with negative speed. 
	\item For every $j=n+1,\dots,n+m$ the Riemann problem \eqref{eq:GSOMrete} has initial datum
	\begin{equation*}
	(\rho_{j}(x,0),y_{j}(x,0)) = \begin{cases}
		(\rhoS_{j},\rhoS_{j}\wS_{j}) &\quad\text{for $x<a_{j}$}\\
		(\rho_{j},\rho_{j}w_{j}) &\quad\text{for $x>a_{j}$,}
	\end{cases}
	\end{equation*}
	and is solved with waves with positive speed. 
	\item It satisfies the consistency condition 
	\begin{equation*}
		\rsolv(\rsolv\uV)=\rsolv\uV
	\end{equation*}
	for every $\uV\in([0,\rhom]\times[\wl,\wr])^{n+m}$.
\end{enumerate}
\end{defn}
We introduce the supply and demand functions to maximize flow at the junction. The \emph{supply} function $\supp{\rho}{w}$ is defined as
\begin{equation}\label{eq:supply}
\supp{\rho}{w} = \begin{cases}
\qmax(w) &\quad\text{if $\rho\leq\sigma(w)$}\\		
Q(\rho,w) &\quad\text{if $\rho>\sigma(w)$,}
\end{cases}
\end{equation}
while we define the \emph{demand} function $\dem{\rho}{w}$ as
\begin{equation}\label{eq:demand}
\dem{\rho}{w} = \begin{cases}
Q(\rho,w) &\quad\text{if $\rho\leq\sigma(w)$}\\
\qmax(w) &\quad\text{if $\rho>\sigma(w)$.}	
\end{cases}
\end{equation}

\subsection{Incoming roads}
Let us consider an incoming road at a junction. 
Only waves with negative speed are admissible. Since $\ldue\geq0$, we can only have a $\rho$-wave which can be a shock or a rarefaction. 
We fix a left state $\umeno=(\rhomeno,\wmeno)$ and look for the set of all admissible right states $\upiuS=(\rhopiuS,\wpiuS)$ that can be connected to $\umeno$ with waves with negative speed.
Along the $\rho$-waves the variable $w$ is conserved, therefore only the density $\rho$ changes. This case is analogous to the definition of admissible solutions on incoming roads for first order traffic models, see for instance \cite{GaravelloPiccoli2006}. 

\begin{prop}\label{prop:incoming}
Let $V$ be a velocity function that verifies properties  \ref{v1}-\ref{v3} and 
let $\umeno=(\rhomeno,\wmeno)$ be a left state on an incoming road.\\
If $\rhomeno=0$, then the only admissible right state is $\upiuS=\umeno$.\\ 
If $\rhomeno\neq0$, then the set of admissible right states $\upiuS=(\rhopiuS,\wpiuS)$ 
verifies $\wpiuS=\wmeno$ and
\begin{enumerate}
\item If $\rhomeno\leq\sigma(\wmeno)$, then $\rhopiuS\in\neg(\umeno)=\{\rhomeno\}\cup \big(\wrhomeno(\wmeno),\rhom(\wmeno)]$, where $\wrhomeno(\wmeno)$ is the density such that $Q(\wrhomeno(\wmeno),\wmeno)=Q(\rhomeno,\wmeno)$.
\item If $\rhomeno>\sigma(\wmeno)$, then $\rhopiuS\in\neg(\umeno)=[\sigma(\wmeno),\rhom(\wmeno)]$.
\end{enumerate}
Moreover, denoting by $d$ the demand function defined in \eqref{eq:demand},  it holds
\begin{equation}\label{eq:dem}
Q(\rhopiuS,\wpiuS)\leq \dem{\rhomeno}{\wmeno}.
\end{equation}
\end{prop}

\begin{proof}
First assume 
$\rhomeno\neq0$.
If $\rhomeno\leq\sigma(\wmeno)$ 
to have $\lambda_1\leq 0$ there are two possibilities: either $\upiuS=\umeno$, or moving above the density value $\wrhomeno(\wmeno)>\sigma(\wmeno)$ by a jump with zero speed. Indeed, since $Q(\wrhomeno(\wmeno),\wmeno)=Q(\rhomeno,\wmeno)$, the Rankine-Hugoniot condition $s =\big(Q(\wrhomeno(\wmeno),\wmeno)-Q(\rhomeno,\wmeno)\big)/(\wrhomeno-\rhomeno)$ implies that the speed of the discontinuity $s$ is zero. In this case, excluding zero speed jumps we can move with a 1-shock with negative speed towards any right state $\upiuS$ with $\wpiuS=\wmeno$ and $\wrhomeno(\wmeno)<\rhopiuS\leq\rhom(\wmeno)$. If $\rhomeno=0$ then $\wrhomeno(\wmeno)=\rhom(\wmeno)$, therefore  the solution is $\upiuS=\umeno$.

If $\rhomeno>\sigma(\wmeno)$, every state $\upiuS$ with $\wpiuS=\wmeno$ and $\rhopiuS\in[\sigma(\wmeno),\rhom(\wmeno)]$ is connected to $\umeno$ with waves with negative speed. In particular, we have a 1-rarefaction wave if  $\rhopiuS\leq\rhomeno$ and a 1-shock if $\rhopiuS>\rhomeno$.
\end{proof}
\begin{remark} 
We allow to remain stationary in $U^-$ (no wave is generated at the junction), while we exclude non-physical vertical shocks with zero velocity, i.e. the solution $\hat\rho =\widetilde\rho^-$.
\end{remark}
\begin{defn}[good and bad datum]\label{def:InBadDatum}
    For every incoming road we say that a datum $(\rho^-,w^-)$ is a good datum if $\rho^-\in [\sigma(w^-),\rho^\text{max} (w^-)]$ and a bad datum otherwise.
\end{defn}
\begin{figure}[h!]
\centering
\normalsize
\includegraphics[]{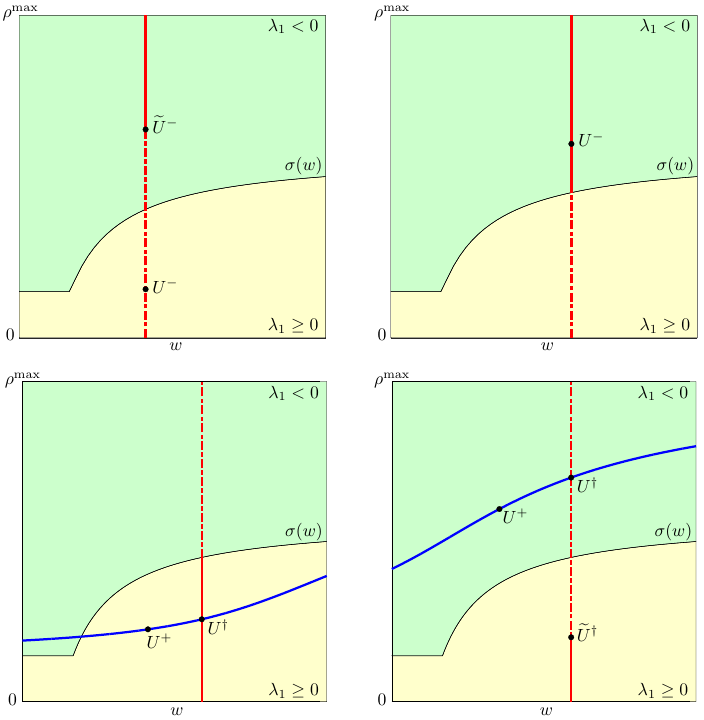}
\caption{The function graphs refer to the CGARZ model with the family of flux functions defined in \cite{FanSunPiccoliSeiboldWork2017}. Top: two possible configurations of incoming road states. The red solid line identifies the set of all possible right states $\upiuS$ reachable from the left state $\umeno$.
Bottom: two possible configurations of the state on an outgoing road. The red solid line identifies the set of possible left states $\umenoS$ reachable from the right state $\upiu$. 
The blue lines are the level lines of the second Riemann invariant $V(\rho,w)=V(\rho^+,w^+)$.}
\label{fig:inEout}
\end{figure}

\subsection{Outgoing roads}
Let us consider an outgoing road at a junction. We are interested in the waves with positive speed, thus we can have a 1-shock or 1-rarefaction wave and a 2-contact discontinuity.

We fix a right state $\upiu=(\rhopiu,\wpiu)$ and look for the set of all admissible left states $\umenoS=(\rhomenoS,\wmenoS)$  that can be connected to $\upiu$ with waves with positive speed.
We emphasize that along the 1-waves the $w$ is conserved and only the density $\rho$ changes. We therefore assume that it is given the value $\overline w$, which depends on the states of the incoming roads. 
On the other hand, along the 2-wave the velocity $V(\rho,w)$ is conserved. Then, the definition of the admissible states $\umenoS$ depends on the existence of an intermediate point 
$\umorto=(\rhomorto,\wmorto)$ such that $\wmorto=\overline w$ and $V(\rhomorto,\wmorto)=V(\rhopiu,\wpiu)$.

\begin{prop}\label{prop:Umorto}
Let $V$ be a velocity function that verifies properties  \ref{v1}-\ref{v3}. For a given value $\overline w$ and a given right state $\upiu=(\rhopiu,\wpiu)$ with associated velocity $\vpiu=V(\rhopiu,\wpiu)$, if $\vpiu\leq\vmax(\overline w)$ then there exists a unique point $\umorto=(\rhomorto,\wmorto)$ such that $\wmorto=\overline w$ and $V(\rhomorto,\wmorto)=\vpiu$.
\end{prop}
\begin{proof}
If $\vpiu\leq\vmax(\overline w)$ then the equation $V(\rho,\overline w)=\vpiu$ admits a solution. 
By \ref{v2}, $\partial_\rho V<0$ and, by the implicit function theorem,
there exists $\rho(w;\vpiu)$ such that
$V(\rho(w;\vpiu),w)=\vpiu$. Moreover, \ref{v2}-\ref{v3} imply 
$$
\frac{d\rho}{dw}\,(w;\vpiu)=-\partial_w V/\partial_\rho V \geq 0.
$$
We then have $\wmorto=\overline w$ and $\rhomorto = \rho(\overline w;\vpiu)$.
\end{proof}

\begin{prop}\label{prop:outgoing}
Let $V$ be a velocity function that verifies properties  \ref{v1}-\ref{v3},
$\upiu=(\rhopiu,\wpiu)$ a right state on an outgoing road, and $\vpiu=V(\rhopiu,\wpiu)$ the associated velocity.
A left state $\umenoS=(\rhomenoS,\wmenoS)$, which can be connected to $\upiu$ with positive speed,
satisfies $\wmenoS = \overline w$ and the following
\begin{itemize}
\item[(i)] If $v^+ \leq \vmax(\overline w)$, let 
$\umorto=(\rhomorto,\wmorto)$ be the intersection point between 
the two level curves of the first and second Riemann invariant given by $\{w \, :\, w=\overline w\}$ and $\{(\rho,w)\,:\,V(\rho,w)=V(\rho^+,w^+)\}$ respectively, 
then $\wmorto=\overline w$ and
\begin{enumerate}
\item if $\rhomorto\leq\sigma(\wmorto)$, then $\rhomenoS\in\pos(\upiu)=[0,\sigma(\wmorto)]$;
\item if $\rhomorto>\sigma(\wmorto)$, then $\rhomenoS\in\pos(\upiu)=[0,\wrhomorto(\wmorto)\big)$, where $\wrhomorto(\wmorto)$ is the density such that $Q(\wrhomorto,\wmorto)=Q(\rhomorto,\wmorto)$. Note that we do not allow jumps with zero speed to occur at the junction, i.e. $\rhomenoS < \wrhomorto(\wmorto)$.
\end{enumerate}
\item[(ii)] If $v^+ > \vmax(\overline w)$ then $\rhomenoS\in\pos(\upiu)=[0,\sigma(\overline w)]$.
\end{itemize}
Moreover, denoting by $s$ the supply function defined in \eqref{eq:supply}, it holds 
\begin{equation}\label{eq:sup}
Q(\rhomenoS,\wmenoS)\leq \supp{\rhomorto}{\wmorto}.
\end{equation}
\end{prop}
\begin{proof}
If $\vpiu\leq\vmax(\overline w)$, by Proposition \ref{prop:Umorto} there exists a unique point $\umorto$ such that $\wmorto = \overline w$ and $V(\rho,\wmorto)=\vpiu$. Thus,
if $\rhomorto\leq\sigma(\wmorto)$, then every state $\umenoS$ with $\wmenoS=\overline w$ and $\rhomenoS\in[0,\sigma(\wmorto)]$ 
can be connected to $\umorto$ by waves with positive speed (Figure \ref{fig:inEout} bottom-left). In particular we have a 1-rarefaction wave if  $\rhomorto\leq\rhomenoS$ and a 1-shock if $\rhomorto>\rhomenoS$. Then, $\umorto$ is connected to $\upiu$ by a 2-contact discontinuity which has positive speed. 

If $\rhomorto>\sigma(\wmorto)$, we move by a jump with positive speed to the density $\rhomenoS<\wrhomorto(\wmorto)$. In this case, a 1-rarefaction connects to an intermediate state $\umenoS$ with $\wmenoS=\wmorto$ and $0\leq\rhomenoS<\wrhomorto(\wmorto)$, then a 2-contact discontinuity connects to $\upiu$.

Otherwise, if $\vpiu>\vmax(\overline w)$ then the equality $V(\rho,\overline w)=\vpiu$ can not hold. It holds $\rhomorto = 0$ and the admissible left state $\rhomenoS$ has to be in $[0,\sigma(\overline w)]$. 
\end{proof}

To summarize, we denote 
\begin{equation}\label{eq:rhomorto}
	\rhomorto(\overline w;\vpiu) = \begin{cases}
		\rho(\overline w;\vpiu) &\text{if $\vpiu\leq\vmax(\overline w)$}\\
		0 &\text{if $\vpiu>\vmax(\overline w)$}
	\end{cases}
\end{equation}
where $\rho(\cdot;\vpiu)$ is the implicit function given by the equation $V(\rho,w)=\vpiu$, which is well defined as stated in Proposition \ref{prop:Umorto}.

We conclude this section showing that the situation where a $w$-wave with zero speed is emanated from the junction J cannot happen on outgoing roads: as the following result points out, in that case the $\rho$-wave emanated from the junction has non-positive speed (not admissible on outgoing roads).
\begin{lemma}\label{lem:casea}
Let $U^-=(\rho^-, w^-)$ and $\umorto=(\rhomorto,\wmorto)$ be respectively the left and the right state of a $\rho$-wave and let $\umorto=(\rhomorto,\wmorto)$ and $U^+=(\rho^+, w^+)$ be the left and the right state of a $w$-wave, both emanated from the junction $J$ at time $\bar t>0$. Suppose that the $w$-wave has zero speed $\lambda_2(\umorto)=\lambda_2(U^+)=0$. Then, the $\rho$-wave is a shock with non-positive speed $\lambda_1\le 0$. 
\end{lemma}
\begin{proof}
Since $\umorto=(\rhomorto,\wmorto)$ and $U^+=(\rho^+, w^+)$ are the left and the right states of a $w$-wave with vanishing speed $\lambda_2(\rho^+, w^+)= V (\rho^+, w^+)=V(\rhomorto, \wmorto)=0$, recalling by \ref{v2} that $V_\rho < 0$ for every $w$, it follows that $\rhomorto=\rho^\text{max}(\wmorto)$ and $\rho^+= \rho^\text{max}(w^+)$. Moreover, being $U^-=(\rho^-, w^-)$ and $\umorto=(\rhomorto,\wmorto)$ the left and right state of a $\rho$-wave, it follows that $w^-=\wmorto$. This yields that $\rhomorto=\rho^\text{max}(w^-)$, so that the right state of the $\rho$-wave (of left state  $U^-=(\rho^-, w^-)$) is given by $(\rho^\text{max}(w^-), w^-)$. By property \ref{q1}, it follows that $Q(\rho^\text{max}(w^-), w^-)=0$. Moreover, since obviously $\rho^\text{max}(w^-)> \rho^-$, then the $\rho$-wave is a shock and its speed is given by $s=\frac{Q(\rho^\text{max}(w^-), w^-)-Q(\rho^-, w^-)}{\rho^\text{max}(w^-)-\rho^-}=-\frac{Q(\rho^-, w^-)}{\rho^\text{max}(w^-)-\rho^-}=-\frac{\rho^- V (\rho^-, w^-)}{\rho^\text{max}(w^-)-\rho^-} \le 0$ as $V(\cdot, \cdot) \ge 0$ by \ref{v1}.
\end{proof}

\begin{remark} By point (i)-2 in Proposition \ref{prop:outgoing} and Lemma \ref{lem:casea} we have that on an outgoing road a $w$-wave generated at the junction is always followed by a $\rho$-wave. In fact, as for incoming roads, no vertical shocks can occur at the junction.
\end{remark}

\begin{defn}[good and bad datum]\label{def:OutBadDatum}
For every outgoing road we say that a datum $(\rho^+,w^+)$ is a good datum if $\rhomorto(\overline w;v^+)\in[0,\sigma(\overline w)]$ for a given value $\overline w$, and a bad datum otherwise.
\end{defn}
 



\section{The Adapting Priority Riemann Solver for junctions with n incoming and m outgoing roads}\label{sec:NinM}
In this section, we introduce the Riemann Solver to define the solution in the general case of a junction with $n$ incoming roads and $m$ outgoing roads. We have $n$ left states $U_{i}$ and $m$ right states $U_{j}$, and our aim is to find $\uS_{i}$ and $\uS_{j}$ for $i=1,\dots,n$ and $j=n+1,\dots,n+m$. In order to determine which incoming road has the priority of sending vehicles with respect to the others, we introduce a priority vector 
$$(\puno,\dots,p_{n}) \quad\mbox{such that}\quad p_i\geq 0 \quad\mbox{and}\quad \sum_{i=1}^{n}p_{i}=1.$$ 
If we have $p_{i_{1}}=\dots=p_{i_{\kappa}}=0$ for $\kappa$ distinct indexes $i_{1},\dots,{i_{\kappa}}$, then no vehicles from these roads cross the junction, and thus we reduce to the $(n-\kappa)\to m$ case.
Then, we define the matrix of distribution
\begin{equation}\label{eq:Amatrix}
	A=\begin{pmatrix}
	\aij{n+1,}{1} &\dots & \aij{n+m,}{1}\\
	\vdots & \ddots & \vdots\\
	\aij{n+1,}{n} & \dots & \aij{n+m,}{n}
	\end{pmatrix}
\end{equation}
whose elements $\aij{j,}{i}$ define the percentage of distribution of vehicles from road $i$ to road $j$ and are such that $\sum_{j=n+1}^{n+m}\aij{j,}{i}=1$, $i=1,\dots,n$. 
If we have $\kappa$ columns with zero entries, then no vehicle enters the corresponding outgoing roads, reducing the problem to case $n\to(m-\kappa)$. 
Therefore we assume that for each $j$ there exists at least a value $\aij{j,}{i}\neq0$ for $i=1,\dots,n$.

\medskip

We now introduce the \textit{Adapting Priority Riemann Solver for Second-Order Models} (APRSOM), which we propose for computing the unknowns values $\uS_{i}$ and $\uS_{j}$ at a junction with $i=1,\dots,n$ incoming and $j=n+1,\dots,n+m$ outgoing roads.
This approach can be summarized as follows:
\begin{itemize}
    \item[(a)] We define the set $\omegaInc\subset\R^{n}$ as the collection of all admissible solutions determined by the incoming roads and the $m$ hyperplanes where the outgoing flow is maximized.
    \item[(b)] We determine whether the priority rule line first intersects one of the maximizing hyperplanes or a boundary of $\omegaInc\subset\R^{n}$. If it intersects a maximizing hyperplane first, we immediately identify the solution that both maximizes the flow and respects the priority rule. Otherwise, we fix the component $\qS_{i}$ corresponding to the boundary of $\omegaInc$ crossed by the priority rule and proceed iteratively along that boundary. At each step, we reduce the problem's dimensionality and continue searching for the flow maximization solution.
\end{itemize}

Let us start first by assuming the conservation of $\rho$ and $y$ at the junction, i.e. for each $j=n+1,\dots,n+m$ we set
\begin{align}
	\sum_{i=1}^{n} \aij{j,}{i}\qS_{i} &= \qS_{j}\label{eq:NinMq}\\
	\sum_{i=1}^{n} \aij{j,}{i}\qS_{i}\wS_{i} &= \qS_{j}\wS_{j}.\label{eq:NinMw}
\end{align}
By \eqref{eq:NinMq} in \eqref{eq:NinMw}, for each outgoing road we have
\begin{equation}\label{eq:wGeneric}
	\wS_{j} = \frac{\sum_{i=1}^{n} \aij{j,}{i}\qS_{i}\wS_{i} }{\sum_{i=1}^{n} \aij{j}{i}\qS_{i}},
\end{equation}
where by Proposition \ref{prop:incoming} for incoming roads we have $\wS_{i}=w_{i}$, $i=1,\ldots,n$. 

\medskip

We now move to the $(\quno,\dots,\qn)$-hyperplane and we follow the idea given in points (a) and (b) above, looking for the maximization of the flow. 
\paragraph{Step 1}
For each incoming road we consider the demand function $\di=\dem{\rho_{i}}{w_{i}}$ defined in \eqref{eq:demand}, in order to define the set of all admissible solutions on incoming roads
\begin{align}
	\Omega_{\mathrm{inc}}&=[0,d_{1}]\times\dots\times[0,d_{n}].\label{eq:omegaN}
\end{align}
We assume that $d_{i}\neq0$ for each $i=1,\dots,n$. Indeed, the trivial case of $d_{i}=0$ for all $i$ means that no vehicles cross the intersection, while the case of $d_{i_{1}}=\dots=d_{i_{\kappa}}=0$ for $\kappa$ distinct indexes $i_{1},\dots,{i_{\kappa}}$ reduces the junction to the $(n-\kappa)\to m$ case.
We then introduce the $(n-1)$-dimensional manifold (hyperplane) of priority rule in parametric form introducing the \textit{flux variable} $h$ such that
\begin{equation}\label{eq:rettarN}
	\rettar:\begin{cases}
		\quno=\puno h\\
		\vdots\\
		\qn=p_{n} h,
	\end{cases}
\end{equation}
flux quantities
\begin{equation}\label{eq:hi}
		h_{i} = \max\{h\,:\,p_{i}h\leq\di\}=\disp\frac{\di}{p_{i}}.
\end{equation}
Note that $(\puno h_{i},\dots,p_{n} h_{i})$ is the intersection point between the line $\rettar$ and the hyperplane $d_{i}$. 
Next, we set $(\qS_{1},\dots, \qS_{n})=(\puno h,\dots, p_{n} h)$ in \eqref{eq:wGeneric} and we obtain
\begin{equation}\label{eq:wGeneric2}
	\wS_{j} =  \frac{\sum_{i=1}^{n} \aij{j}{i}p_{i}w_{i} }{\sum_{i=1}^{n} \aij{j}{i}p_{i}}.
\end{equation}
By \eqref{eq:wGeneric2} for $j=n+1,\dots,n+m$ we define for the supply function given in \eqref{eq:sup},
\begin{align}
	\nonumber s_{j}(\wS_{j})&=\supp{\rhomorto(\wS_{j};\vpiu_{j})}{\wS_{j}},\\
	\Omega_{\mathrm{out}}&=[0,s_{n+1}]\times\dots\times[0,s_{n+m}],\label{eq:omegaM}
\end{align}
where $\rhomorto(w;v)$ is specify in \eqref{eq:rhomorto}.
We then introduce 
\begin{equation}\label{eq:psi1}
	\psi_{j}(h)=h\sum_{i=1}^{n}\aij{j}{i}p_{i}
\end{equation}
and we define 
\begin{equation}\label{eq:min1}
	h_{j} = \min\{h>0\,:\, \psi_{j}(h)= s_{j}(\wS_{j}), \text{for $\wS_{j}$ in \eqref{eq:wGeneric2}}\}=\frac{s_{j}(\wS_{j})}{\sum_{i=1}^{n}\aij{j}{i}p_{i}}
\end{equation}
which identifies the intersection points $(\puno h_{j},\dots,p_{n} h_{j})$ between $\rettar$ in \eqref{eq:rettarN} and the hyperplanes 
\begin{equation*}
	\rettas_{j}: \sum_{i=1}^{n}\aij{j}{i}q_{i}=s_{j}(\wS_{j})
\end{equation*}
where the outgoing flux is maximized.
We define 
\begin{equation}\label{eq:hbv1}
	\hbv{1}=\min_{i,j}\{h_{i},h_{j}\}.
\end{equation}
We have the following possibilities:
\begin{enumerate}[label=(\arabic*),ref=\textup{(\arabic*)}]
	\item\label{alg1:casoA} If there exists an index $j\geq n+1$ such that $\hbv{1}=h_{j}$ then the line $\rettar$ first intersects a hyperplane $\rettas_{j}$ which maximizes the outgoing flux of road $j$ and satisfies the priority rule, thus we define the fluxes $(\qS_{1},\dots,\qS_{n})=(\puno\hbv{1},\dots, p_{n}\hbv{1})$ and the procedure stops. 
	\item\label{alg1:casoB} There is no index $j\geq n+1$ such that $\hbv{1}=h_{j}$. In this case we proceed as follows.
	\begin{enumerate}[label=(\alph*)]
		\item If we need to respect the priority rule then we define the fluxes $(\qS_{1},\dots,\qS_{n})=(\hbv{1}\puno,\dots,\hbv{1} p_{n})$, with $\hbv{1}=h_{i}$ for some $i\leq n$, and we stop.
		\item If we can adapt the priority rule, let $\ind{1}\leq n$ be the index of the incoming road such that $h_{\ind{1}}=\hbv{1}$. We set $\qS_{\ind{1}}=d_{\ind{1}}$, we introduce $\insInd=\{\ind{1}\}$ and we proceed by iteration.
	\end{enumerate}
\end{enumerate}

\paragraph{Step $\indice+1$}
Assume to have already defined $\indice$ components of the vector $(\qS_{1},\dots,\qS_{n})$, i.e. for each $\ind{k}\in \insInd=\{\ind{1},\dots,\ind{\indice}\}$ we have $\qS_{\ind{k}}=d_{\ind{k}}$ and we have to determine the remaining $\qS_{i}=hp_{i}$ for $i\leq n$ and $i\notin\insInd$. We now introduce the function
\begin{align*}
	\varphi_{j}(h) &= h\sum_{i\notin\insInd} \aij{j}{i}p_{i}w_{i} +\sum_{k\in\insInd}\aij{j}{\ind{k}}d_{\ind{k}}w_{\ind{k}}
\end{align*}
and modify $\psi_{j}(h)$ in \eqref{eq:psi1} as
\begin{align*}	
	\psi_{j}(h) &= h\sum_{i\notin\insInd} \aij{j}{i}p_{i} +\sum_{k\in\insInd}\aij{j}{\ind{k}}d_{\ind{k}}.
\end{align*}
We rewrite \eqref{eq:wGeneric} as
\begin{equation}\label{eq:wGenericK}
	\wS_{j}(h) =  \frac{\varphi_{j}(h)}{\psi_{j}(h)},
\end{equation}
and we exploit it to define $s_{j}(\wS_{j}(h)):=\supp{\rhomorto(\wS_{j}(h);\vpiu_{j})}{\wS_{j}(h)}$, $j=n+1,\dots,n+m$, and
\begin{equation}\label{eq:nmMax}
	h_{j} = \min\{h\in[\hbv{\indice},+\infty)\,:\, \psi_{j}(h)= s_{j}(\wS_{j}(h)), \text{with $\wS_{j}(h)$ in \eqref{eq:wGenericK}}\}.
\end{equation}
To conclude the iterative step we define 
\begin{equation}\label{eq:hbvk}
	\hbv{\indice+1}=\min_{i\notin\insInd,j}\{h_{i},h_{j}\}.
\end{equation}
with $h_{i}$ in \eqref{eq:hi}. Again we have two possibilities:
\begin{enumerate}[label=(\arabic*), ref=\textup{(\arabic*)}]
	\item\label{algS:casoA} If there exists an index $j\geq n+1$ such that $h_{j}=\hbv{\indice+1}$ then the straight line $\rettar$ first intersects a  hyperplane $\rettas_{j}$ which maximizes the outgoing flux of road $j$ and thus we define the remaining fluxes $\qS_{i}=\hbv{\indice+1} p_{i}$ for $i\notin\insInd$.
	\item\label{algS:casoB} Otherwise $\hbv{\indice+1}=h_{\ind{\indice+1}}$ for some $\ind{\indice+1}\leq n$, $\ind{\indice+1}\notin\insInd$. We add the new index in $\insInd$, i.e. $\insInd=\{\ind{1},\dots,\ind{\indice+1}\}$, and we continue iteratively until we have defined all the elements of the vector $(\qS_{1},\dots,\qS_{n})$. 
\end{enumerate}

\begin{remark}
We observe that the set $\mathcal{H}=\{h\in[\hbv{\indice},+\infty)\,:\, \psi_{j}(h)= s_{j}(\wS_{j}(h)), \text{with $\wS_{j}(h)$ in \eqref{eq:wGenericK}}\}$ is not empty for each step $\indice$ of the algorithm. Indeed, we have
\begin{align*}
	\lim_{h\to\infty}\psi_{j}(h)&=+\infty\\
	\lim_{h\to\infty} s_{j}(\wS_{j}(h))&= s_{j}(\widetilde w_{j})<+\infty,
\end{align*}
where $\widetilde w_{j} = \lim_{h\to\infty} \wS_{j}(h)<+\infty$, and $\psi_{j}(\hbv{\indice})\leq s_{j}(\wS_{j}(\hbv{\indice}))$ by construction. Therefore, by continuity, for each $j$ there exists a certain $h>\hbv{\indice}$ such that the equality $\psi_{j}(h)= s_{j}(\wS_{j}(h))$ holds.
\end{remark}

We can now define the APRSOM solver for GSOM on road networks.
\begin{defn}\label{def:RSgsom}
Let $\qv=(\qS_{1},\dots,\qS_{n})$ be the vector of incoming fluxes at the junction defined by the previous procedure applied to the initial state $\uV$, and $A\cdot \qv^{T}=(\qS_{n+1},\dots,\qS_{n+m})$ the vector of outgoing fluxes, where $A$ is the matrix of distribution \eqref{eq:Amatrix}. For every $i=1,\dots,n$ set 
\begin{itemize}
	\item $\wS_{i}=w_{i}$,
	\item $\rhoS_{i}\in\neg(U_{i})$ such that $Q(\rhoS_{i},\wS_{i})=\qS_{i}$, where $\neg(U_{i})$ is the set of possible right states for incoming roads defined in Proposition \ref{prop:incoming},
\end{itemize}
and $\uS_{i}=(\rhoS_{i},\wS_{i})$.
For every $j=n+1,\dots,n+m$ set 
\begin{itemize}
	\item $\wS_{j}$ as in \eqref{eq:wGeneric} if $\qS_{i}\neq0$ for at least an index $i$, or equal to $w_{j}$ otherwise,
	\item $\rhoS_{j}\in\pos(U_{j})$ such that $Q(\rhoS_{j},\wS_{j})=\qS_{j}$, where $\pos(U_{j})$ is the set of possible left states for outgoing roads defined in Proposition \ref{prop:outgoing},
\end{itemize}
and $\uS_{j} = (\rhoS_{j},\wS_{j})$.
The Adapting Priority Riemann Solver for Second Order Models ($\rsolvG$) on road networks is such that
\begin{equation*}
	\rsolvG\uV=\uVS.
\end{equation*}
\end{defn}

\subsection{The case of a merge}\label{sec:2in1}
For the sake of clarity, here is an illustration of the APRSOM algorithm for the case where there are two incoming roads and one outgoing road at an intersection (a merge). Let then be given two left states $\umeno_{1}$ and $\umeno_{2}$ for the incoming roads and a right state $\upiu_{3}$ for the outgoing road, then our aim is to determine $\upiuS_{1}$, $\upiuS_{2}$ and $\umenoS_{3}$. The approach is based on the priority rule defined by a vector $(\puno,\pdue)$, with $\puno+\pdue=1$ with  $\puno \neq 0$ and $\pdue \neq 0$.
In this case we have no distribution parameters and the conservation of $\rho$ and $y$ at the junction as in \eqref{eq:NinMq} and \eqref{eq:NinMw} respectively reads
\begin{align}
	\qpiuS_{1}+\qpiuS_{2}&=\qmenoS_{3} \label{eq:q2in1}\\
	\qpiuS_{1}\wpiuS_{1}+\qpiuS_{2}\wpiuS_{2}&=\qmenoS_{3}\wmenoS_{3}.\label{eq:w2in1}
\end{align}
By Proposition \ref{prop:incoming}, we have $\wpiuS_{1}=\wmeno_{1}$ and $\wpiuS_{2}=\wmeno_{2}$. Substituting \eqref{eq:q2in1} in \eqref{eq:w2in1} implies
\begin{align}\label{eq:w3}
	\wmenoS_{3} = \frac{\qpiuS_{1}\wmeno_{1}+\qpiuS_{2}\wmeno_{2}}{\qpiuS_{1}+\qpiuS_{2}}.
\end{align}
We now move to the $(\quno,\qdue)$-plane and we look for the maximization of the flow.
\paragraph{Step 1}
We introduce $\duno=\dem{\rhomeno_{1}}{\wmeno_{1}}$, $\ddue=\dem{\rhomeno_{2}}{\wmeno_{2}}$ and the rectangle of possible solutions $\Omega=[0,d_{1}]\times[0,d_{2}]$. As previously explained, we assume that both $\duno$ and $\ddue$ are positive to exclude both the trivial cases where no vehicle crosses the junction and the $1\to1$ case.
We then introduce in the $(\quno,\qdue)$-plane the straight line $\rettar$ of priority rule by means of the flux variable $h$, such that 
\begin{equation*}\label{eq:parametrica}
	\rettar : \begin{cases}
		\quno=\puno h\\ 
		\qdue = \pdue h,
	\end{cases}
\end{equation*}
and compute the flux quantities $\huno$ and $\hdue$ such that
\begin{align}
	\huno&=\max\{h\,:\,\puno h\leq\duno\}=
	\disp\frac{\duno}{\puno}
	\label{eq:h1}\\
	\hdue&=\max\{h\,:\,\pdue h\leq\ddue\}=
	\disp\frac{\ddue}{\pdue}
	\label{eq:h2}.
\end{align} 
Then, $(\duno,\pdue \huno)$ is the intersection point between the straight line $\rettar$ and the vertical line $\duno$ while $(\puno\hdue,\ddue)$ is the intersection point between the straight line $\rettar$ and the horizontal line $\ddue$.
By setting $\qpiuS_{1}=h\puno$ and $\qpiuS_{2}=h\pdue$ in \eqref{eq:w3} we have 
\begin{align}
	\wmenoS_{3}&=\puno\wmeno_{1}+\pdue\wmeno_{2}\label{eq:2in1w3step1}\\
	\nonumber\stre&=\supp{\rhomorto(\wmenoS_{3};\vpiu_{3})}{\wmenoS_{3}},\label{eq:2in1s3}
\end{align} 
where $\vpiu_{3}=V(\rhopiu_{3},\wpiu_{3})$ and $\rhomorto_{3}(\wmenoS_{3};\vpiu_{3})$ is given in \eqref{eq:rhomorto}. In order to maximize the flux on the outgoing road, we set
\begin{equation}\label{eq:2in1h3}
	\htre = \min\{h\,:\,h\puno+h\pdue=\stre\}=\stre,
\end{equation}
which identifies the intersection point $(\htre\puno,\htre\pdue)$ between the straight line $\rettar$ and the straight line
\begin{equation}\label{eq:rettas}
	\rettas: \quno+\qdue=\stre,
\end{equation}
where the outgoing flux is maximized.
We then define 
\begin{equation}
	\hbv{1}=\min\{\huno,\hdue,\htre\}.
\end{equation}
Figure \ref{fig:linea} shows an example of the three points identified by $\huno$, $\hdue$ and $\htre$, i.e. $P_{1} = (\duno,\huno\pdue)$, $P_{2}=(\hdue\puno,\ddue)$ and $P_{3}=(\htre\puno,\htre\pdue)$, where $\hbv{1}=\huno$. 
\begin{figure}[h!]
\centering
\includegraphics[]{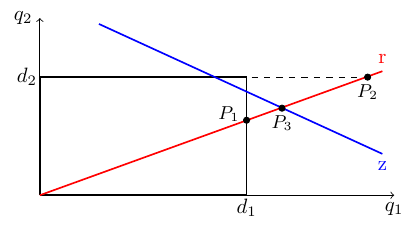}
\caption{Example of the three points $P_{1}=(\huno\puno,\huno\pdue)$, $P_{2}=(\hdue\puno,\hdue\pdue)$ and $P_{3}=(\htre\puno,\htre\pdue)$.}
\label{fig:linea}
\end{figure}

We now have two possibilities: $\hbv{1}=\htre$ or $\hbv{1}\neq\htre$. In the first case we immediately find a couple of incoming fluxes $(\qpiuS_{1},\qpiuS_{2})$ which satisfy the priority rule and that maximise the flux. In the second case the priority first intersects the boundary of the set of possible solutions $\Omega$ and does not maximize the outgoing flux. 
The approach we propose is divided in two cases:
\begin{enumerate}[label=(\alph*)]
	\item We need to strictly satisfy the priority rule. This case is necessary to simulate traffic scenarios such as traffic lights, where the priority rule must be satisfied.
	\item We are free to adapt, i.e. to change, the priority rule. This case is useful to maximise the flux when the intersection between the straight lines $\rettar$ and $\rettas$ is outside the set of possible solutions $\Omega$. The idea is to change the priority $\rettar$, and consequently the parameter $\wmenoS_{3}$ and the maximization straight line $\rettas$, looking for the intersection between the modified $\rettar$ and $\rettas$ which maximizes the flux at the junction.
\end{enumerate}

\begin{description}
\item[Case $\hbv{1}=\htre$.]
In this case, the priority rule first intersects the straight line $\rettas$, see Figure \ref{fig:2in1dentro}. This means that the intersection point identifies two incoming fluxes satisfying the priority and which maximise the outgoing flux. Therefore, we have $\wmenoS_{3}$ in \eqref{eq:2in1w3step1} and
\begin{equation*}
	\qpiuS_{1} = \puno\htre, \qquad \qpiuS_{2} = \pdue\htre,\qquad \qmenoS_{3}=\htre.
\end{equation*}  
\begin{figure}[h!]
\centering
\includegraphics[]{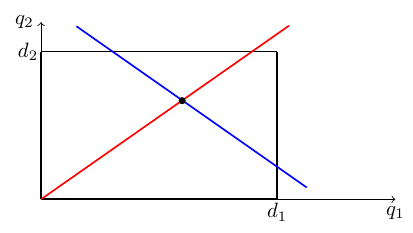}
\caption{Merge junction with $\hbv{1}=\htre$.}
\label{fig:2in1dentro}
\end{figure}
\item[Case $\hbv{1} = \huno$.]
In this case the straight line $\rettar$ first intersects the vertical line $\quno=\duno$. 
\begin{enumerate}[label=(\alph*)]
	\item If we need to respect the priority rule, then the solution is given by 
\begin{equation*}
	\qpiuS_{1} = \puno\huno, \qquad \qpiuS_{2} = \pdue\huno,\qquad \qmenoS_{3}=\huno,
\end{equation*}  
	with $\wmenoS_{3}$ defined in \eqref{eq:2in1w3step1}. This case is represented in Figure \subref*{fig:2in1A}.
	\item If we can adapt the priority rule, we fix $\qpiuS_{1}=\hbv{1}\puno=\duno$ and we move along the vertical side of $\Omega$, looking for $\qpiuS_{2}=\hbv{2}\pdue$ for a proper $\hbv{2}$. The idea of our approach is to modify both $\rettar$ and $\rettas$ in order to find the intersection between the two straight lines along the vertical line $\quno=\duno$.  
By equation \eqref{eq:w3} with $\qpiuS_{1}=\duno$ and $\qpiuS_{2}=h\pdue$ we have
\begin{equation}\label{eq:w3h1}
	\wmenoS_{3}(h)= \disp\frac{\duno\wmeno_{1}+ h\pdue\wmeno_{2}}{\duno+h\pdue}
\end{equation}
which is such that 
\begin{equation*}
	\lim_{h\to0}\wmenoS_{3}(h) = \wmeno_{1},\qquad \lim_{h\to\infty}\wmenoS_{3}(h) = \wmeno_{2}.
\end{equation*}
Let $\stre(\wmenoS_{3}):=\supp{\rhomorto(\wmenoS_{3};\vpiu_{3})}{\wmenoS_{3}}$, we define $\htre$ as
\begin{equation}\label{eq:2in1maxh3}
	\htre=\min\{h\in[\huno,+\infty) \,:\, \duno+h\pdue=\stre(\wmenoS_{3}(h)), \text{ for $\wmenoS_{3}(h)$ in \eqref{eq:w3h1}}\}.
\end{equation}
Note that there exists at least a value of $h$ satisfying $\duno+h\pdue=\stre(\wmenoS_{3}(h))$. Indeed, $\duno+\hbv{1}\pdue<\stre(\wmenoS_{3}(\hbv{1}))$ by hypothesis of $\hbv{1}\neq\htre$, and
\begin{equation*}
	\lim_{h\to+\infty}\duno+h\pdue=+\infty, \qquad \lim_{h\to+\infty} \stre(\wmenoS_{3}(h))=\stre(\wmeno_{2})<+\infty,
\end{equation*}
hence, by continuity, $\duno+h\pdue$ must intersect $\stre(\wmeno_{3}(h))$ for some $h$.

Once computed the new $\htre$, we define $\hbv{2}=\min\{\hdue,\htre\}$, with $\hdue$ in \eqref{eq:h2}, and the fluxes crossing the junction as 
\begin{equation*}
	\qpiuS_{1} = \duno,\quad \qpiuS_{2}= \hbv{2}\pdue,\quad \qmenoS_{3}=\duno+\hbv{2}\pdue.
\end{equation*}
In Figures \subref*{fig:2in1B} and \subref*{fig:2in1C} we show the solution with $\hbv{2}=\htre$ and $\hbv{2}=\hdue$, respectively. Moreover, the new vector of priority rule is $(\pb_{1},\pb_{2})=(\qpiuS_{1}/\qmenoS_{3},\qpiuS_{2}/\qmenoS_{3})$.
%

\end{enumerate}

\begin{figure}[h!]
\centering
\includegraphics[]{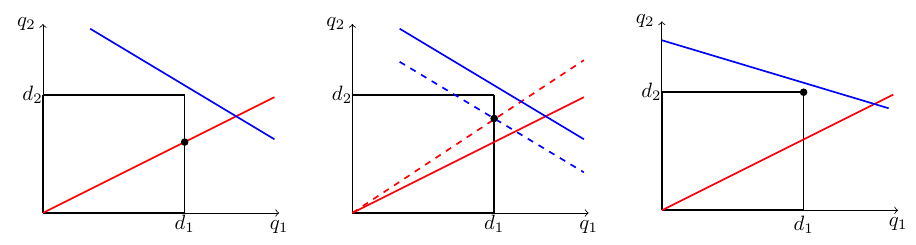}\caption{Merge junction with $\hbv{1}=\huno$.}
\label{fig:2in1destra}
\end{figure}

\item[Case $\hbv{1}=\hdue$.] This case is completely analogous to the previous one, but the straight line $\rettar$ first intersects the horizontal line $\qdue=\ddue$.
\begin{enumerate}[label=(\alph*)]
	\item If we need to respect the priority rule then the solution is given by
\begin{equation*}
	\qpiuS_{1} = \puno\hdue, \qquad \qpiuS_{2} = \pdue\hdue,\qquad \qmenoS_{3}=\hdue,
\end{equation*}  
with $\wmenoS_{3}$ in \eqref{eq:2in1w3step1}. This case is shown in Figure \subref*{fig:2in1D}.
	\item If we can adapt the priority rule then we define 
\begin{align*}
	\wmenoS_{3}(h) &= \disp\frac{h\puno\wmeno_{1}+\ddue\wmeno_{2}}{h\puno+\ddue}\\
	\htre&=\min\{h\in[\hdue,+\infty) \,:\, h\puno+\ddue=\stre(\wmenoS_{3}(h))\}\\
	\hbv{2}&=\min\{\huno,\htre\},
\end{align*}
with $\huno$ in \eqref{eq:h1}, from which we recover
\begin{equation*}
	\qpiuS_{1} = \hbv{2}\puno,\quad \qpiuS_{2}= \ddue,\quad \qmenoS_{3}=\hbv{2}\puno+\ddue.
\end{equation*}
In Figures \subref*{fig:2in1E} and \subref*{fig:2in1F} we show the solution with $\hbv{2}=\htre$ and $\hbv{2}=\huno$, respectively.
\end{enumerate}
\begin{figure}[h!]
\centering
\includegraphics[]{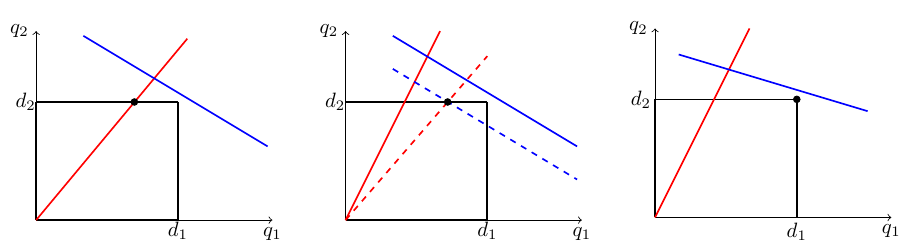}\caption{Merge junction with $\hbv{1}=\hdue$.}
\label{fig:2in1su}
\end{figure}
\end{description}


\section{Bounds on the total variation of the flux for wave-front tracking solutions}\label{sec:TVbounds}

The aim of this section is to give a bound to the total variation of the flux for the approximate solution on the networks obtained via wave-front tracking and the algorithm $\rsolvG$.
Such solutions are constructed solving recursively Riemann problems inside the roads and at the junctions. We refer the reader to \cite{Bressan2000} for a general introduction to wave-front tracking, and to \cite{garavello2016models} for the network case. As shown in \cite{garavello2006AIHP}, due to the finite speed of propagation of waves, it is sufficient to consider the case of a networks composed of a single junction. Therefore we consider the Cauchy problem 
\begin{equation}\label{eq:cauchyWFT}
	\begin{dcases}
		\det\rho_{\road}+\dex(\rho_{\road} v_{\road}) = 0\\
		\det y_{\road}+\dex(y_{\road} v_{\road}) = 0 &\quad\text{for $\road=1,\dots,n+m$}\\
		(\rho_{\road}(x,0),y_{\road}(x,0)) = (\rho_{\road,0}(x),y_{\road,0}(x))
	\end{dcases}
\end{equation}
with initial data $(\rho_{\road,0}(x),y_{\road,0}(x))$ of bounded variation.
The network is formed by a single junction with $n$ incoming and $m$ outgoing roads and at the junction the traffic dynamic is described by \eqref{eq:GSOM2}. We recall that $I_{\road}=(a_{\road},b_{\road})\subset\R$, $\road=1,\dots,n+m$, are the roads of the network.
For a collection of functions $(\rho_{\road},y_{\road})\in C([0,+\infty);\lunoloc(I_{\road})^{2})$ such that, for every $\road\in\{1,...,n+m\}$ and a.e. $t>0$, the map $x\mapsto\rho_{\road}(x,t)$ has a version with bounded total variation, we define the functionals
\begin{equation}\label{eq:funzionali}
	\begin{split}
	\Gamma(t) &= \sum_{i=1}^{n}Q(\rho_{i}(b_{i}-,t),w_{i}(b_{i}-,t))\\
	\tv_{Q}(t) &= \sum_{\road=1}^{n+m}\tv(Q(\rho_{\road}(\cdot,t),w_{\road}(\cdot,t)))\\
	\tv_{w}(t) &=\sum_{\road=1}^{n+m}\tv(w_{\road}(\cdot,t))
	\end{split}
\end{equation}
where $\tv$ is the total variation and $w_{\road}=y_{\road}/\rho_{\road}$. Note that $\Gamma(t)$ represents the flux crossing the junction at time $t$ and involves only the incoming roads.
\begin{defn}
We say that the state $\uV$ is an equilibrium for a Riemann Solver $\rsolv$ if
\begin{equation*}
	\rsolv\uV=\uV.
\end{equation*}
\end{defn}

We now focus on the algorithm $\rsolvG$ introduced in Section \ref{sec:NinM}. We recall that $A$ is the matrix of distribution defined in \eqref{eq:Amatrix} and that $(\puno,\dots,p_{n})$ is the vector defining the priority rule $\rettar$.
Let $\qv=(\qS_{1},\dots,\qS_{n})$ be the set of incoming fluxes obtained with $\rsolvG$ and $A\cdot \qv^{T}=(\qS_{n+1},\dots,\qS_{n+m})$ the resulting set of outgoing fluxes. We introduce 
\begin{equation}\label{eq:Theta}
	\Theta = \{\qv\in\omegaInc\,:\, A\cdot \qv^{T}\in\omegaOut, \text{ with $\omegaInc$ in \eqref{eq:omegaN} and $\omegaOut$ in \eqref{eq:omegaM}}\} 
\end{equation}
so that we define
\begin{equation}\label{eq:hh}
	\hh=\sup_{h \in \R} \{ (h\puno,\dots,h p_{n})\in\Theta\}.
\end{equation}
This value $\hh$ identifies the intersection point $\hh\,(\puno,\dots,p_{n})$ between the line $\rettar$ in \eqref{eq:rettarN} and the set $\Theta$ in \eqref{eq:Theta}.

We now introduce four properties of the Riemann solver to estimate the total variation of $Q$ and $w$ for waves interacting with the junction. The first property says that equilibria depends only on bad data (see Definitions \ref{def:InBadDatum} and \ref{def:OutBadDatum}).
\begin{enumerate}[label=(P\arabic*),ref=(P\arabic*)]
	\item\label{rsp1}
We say that a Riemann solver $\rsolv$ has the property $\mathrm{(P1)}$ if, given $\uV$ and $\uVh$ such that $w_{i}=\hw_{i}$ for $i=1,\dots,n$, $w_{j}=\hw_{j}$ for  $j=n+1,\dots,n+m$ and $\rho_{i}=\hrho_{i}$ $(\rho_{j}=\hrho_{j})$ whenever either $U_{i}$ or $\hu_{i}$ $(U_{j}$ or $\hu_{j})$ is a bad datum, then
\begin{align*}
	\rsolv\uV=\rsolv\uVh.
\end{align*}
\end{enumerate}

The second property refers to interacting waves which involve only the density $\rho$. This means that, starting from an equilibrium of $\rsolv$, we perturb the density of one of the roads keeping its $w$ value unchanged. The following property tells us that the increase in the variation of the flux and of $w$ at the junction is bounded by the strength of the interacting wave as well as by the sum of the variations in the incoming fluxes and in $\hh$ defined in \eqref{eq:hh}. Note that, even when the wave does not directly perturb the property $w$, the latter varies by interacting with the junction.

\begin{enumerate}[resume*]
	\item\label{rsp2}
We say that a Riemann solver $\rsolv$ has the property $(\mathrm{P2})$ if there exists a constant $C\geq1$ such that for every equilibrium $\uV$ of $\rsolv$ and for every wave $\ondarho_{i}$ perturbing $\rho_{i}$ for $i=1,\dots,n$ ($\ondarho_{j}$ perturbing $\rho_{j}$ for $j=n+1,\dots,n+m$, respectively$)$ interacting with $J$ at time $\tb>0$ and producing waves in the arcs according to $\rsolv$, for $\road=i$ ($r=j$) we have
\begin{align*}
	\tv_{Q}(\tb+)-\tv_{Q}(\tb-)&\leq C \min\{|\ondaq_{\road}-q_{\road}|,|\Gamma(\tb+)-\Gamma(\tb-)|+|\hh(\tb+)-\hh(\tb-)|\}\\
	\hh(\tb+)-\hh(\tb-)&\leq C|\ondaq_{\road}-q_{\road}|\\
	\tv_{w}(\tb+)-\tv_{w}(\tb-)&\leq C \min\{|\ondaq_{\road}-q_{\road}|,|\Gamma(\tb+)-\Gamma(\tb-)|+|\hh(\tb+)-\hh(\tb-)|\},
\end{align*}
with $\ondaq_{\road}=Q(\ondarho_{i},w_{i})$ and $q_{\road}=Q(\rho_{i},w_{i})$ $(\ondaq_{\road}=Q(\ondarho_{j},w_{j})$  and $q_{\road}=Q(\rho_{j},w_{j})$, respectively$)$.
\end{enumerate}

The third property also refers to interacting waves which involve only the density $\rho$. It tells us that, when the interacting wave with the junction determines a decrease in the flux, then also $\hh$ decreases and the variation of $\Gamma$ is bounded by the variation of $\hh$.

\begin{enumerate}[resume*]
\item\label{rsp3}
We say that a Riemann solver $\rsolv$ has the property $\mathrm{(P3)}$  if there exists a constant $C\geq1$ such that for every equilibrium $\uV$ of $\rsolv$ and for every wave $\ondarho_{i}$ perturbing $\rho_{i}$ with $\ondaq_{i}=Q(\ondarho_{i},w_{i})< q_{i}$ for $i=1,\dots,n$ ($\ondarho_{j}$ perturbing $\rho_{j}$ with $\ondaq_{j}=Q(\ondarho_{j},w_{j})<q_{j}$ for $j=n+1,\dots,n+m$, respectively$)$ interacting with $J$ at time $\tb>0$ and producing waves in the arcs according to $\rsolv$, we have
\begin{align*}
	\Gamma(\tb+)-\Gamma(\tb-)&\leq C|\hh(\tb+)-\hh(\tb-)|\\
	\hh(\tb+)&\leq\hh(\tb-).
\end{align*}
\end{enumerate}

Finally, we consider an interacting wave with the junction which perturbs both $\rho$ and $w$ on one of the incoming roads. The fourth property says that the increase in the variation of $w$ is bounded by the variation of the interacting wave in $w$ and the strength of the interacting wave as well as by the sum of the variations in the incoming fluxes and in $\hh$.

\begin{enumerate}[resume*]
\item\label{rsp4}
We say that a Riemann solver $\rsolv$ has the property $\mathrm{(P4)}$  if there exist two constants $C_{1}\geq1$ and $C_{2}\geq1$ such that for every equilibrium $\uV$ of $\rsolv$ and for every wave $(\ondarho_{i},\ondaw_{i})$ perturbing $(\rho_{i},w_{i})$, $i=1,\dots,n$, interacting with $J$ at time $\tb>0$ and producing waves in the arcs according to $\rsolv$, 
the estimates on $\tv_{Q}$, $\hh$ and $\Gamma$ hold and we have
\begin{equation*}
	\tv_{w}(\tb+)-\tv_{w}(\tb-)\leq C_{1}|\ondaw_{i}-w_{i}|+C_{2} \min\{|\ondaq_{i}-q_{i}|,|\Gamma(\tb+)-\Gamma(\tb-)|+|\hh(\tb+)-\hh(\tb-)|\},
\end{equation*}
with $\ondaq_{i}=Q(\ondarho_{i},\ondaw_{i})$ and $q_{i}=Q(\rho_{i},w_{i})$.
\end{enumerate}

\begin{remark}
Property \ref{rsp4} only refers to the incoming roads. Indeed, for any outgoing road, if we perturb the equilibrium $(\rho_{j},w_{j})$ with a wave $(\ondarho_{j},\ondaw_{j})$, the solution which arrives at the junction is only characterize by $\rho$-waves with constant $w$, thus on outgoing roads the junction is never affected by the variation in $w$. 
\end{remark}

\begin{theorem}\label{thm:propP}
The Riemann Solver $\rsolvG$ defined in Section \ref{sec:NinM} satisfies properties \ref{rsp1} -- \ref{rsp4} for junctions with $n=2$ incoming and $m=2$ outgoing roads.
\end{theorem}
\noindent The proof of this theorem is given in Appendix \ref{appendice}.

\subsection{Flux variation due to returning waves}\label{sec:fluxRet}
We fix a road $I_s$ (incoming or outgoing) and introduce the following definitions.

\begin{defn}[Backward wave tree]
For a fixed road $I_s$ and a wave located at a point $(x,t)$ of the domain ($I_s\times [0,+\infty[$), the \textbf{backward wave tree} is obtained by tracing the wave fronts of the solution - constructed via the WFT (Wave Front-Tracking) Algorithm - backward in time from the chosen point $(x,t)$ to the boundary of the domain. Wave fronts of both families are considered, thus, repeating this process recursively at each interaction point, it generates a tree-like structure that represents the backward propagation of information from the point $(x,t)$.
\end{defn}
\begin{defn}[Backward wave branch]
A \textbf{backward wave branch} of a backward wave tree consists of a piece-wise linear branch (or branches, in the case of interactions between waves of the same family) that includes only fronts of the same family. 
\end{defn}

Hereafter, a wave with right state $U^+=(\rho^+, w^+)$ and left state $U^-=(\rho^-, w^-)$ will be denoted by
\begin{equation}\label{def:Sigma}
\Sigma:=(U^-, U^+),
\end{equation}
and the \emph{forward} (resp. backward) \emph{flux variation} across the wave $\Sigma$ by
\begin{equation}\label{def:FVsigma}
\delta_+ Q_\Sigma = Q(\rho^{+}, w^{+})-Q(\rho^{-}, w^{-}) \quad\mbox{and}\quad \delta_- Q_\Sigma = -\delta_+ Q_\Sigma.
\end{equation}
%

\begin{defn}[Returning wave] \label{def:returnwave} 
A \textbf{returning wave} $\Sret = (U^{R, -}, U^{R, +})$, where $U^{R, \pm} = (\rho^{R, \pm}, w^{R, \pm})$, is a wave front generated at the junction J and interacting with the junction J at a later time  $t_a$. Thus the \emph{backward wave branch} includes at least one wave (of the same family) originating  from the junction J at a previous time.  We indicate by $t_o < t_a$ the greatest time at which a wave of the backward branch of $\Sigma^{R; t_o, t_a}$ originated from the junction J.
We shall refer to $t_o$ and $t_a$ as the \emph{original time} and the \emph{absorption time} of the returning wave.
Moreover, using the notation of \eqref{def:FVsigma}, we define the \emph{flux variation} at the absorption time $t_a$ as follows :
\begin{itemize}
\item if $\Sret$ is traveling on an incoming road
\begin{align}\label{def:fluxJin}
\delta_- Q_{\Sret}=Q(\rho^{R, -}, w^{R, -})-Q(\rho^{R, +}, w^{R, +}),
\end{align}
\item if $\Sret$ is traveling on an outgoing road
\begin{align}\label{def:fluxJout}
\delta_+ Q_{\Sret} =Q(\rho^{R, +}, w^{R, +})-Q(\rho^{R, -}, w^{R, -}).
\end{align}
\end{itemize}
\end{defn}
In Figure \ref{fig:inRet} and Figure \ref{fig:outRet} we show some graphical examples of backward wave branch of returning waves on incoming and outgoing roads respectively. $\rho$-waves are represented by solid blue lines and $w$-waves by dashed green lines.

\begin{figure}[h!]
\centering
\includegraphics[]{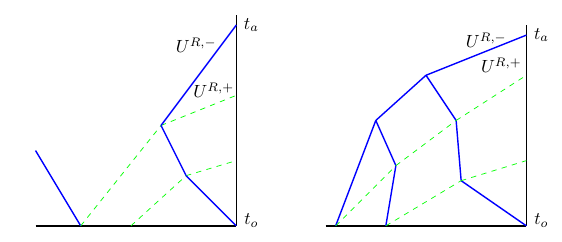}
\caption{Examples of \textit{backward wave branch} (the solid blue lines) of a returning wave $\Sret_\rho$ on an incoming road. Green dashed lines indicate $w$-waves.}
\label{fig:inRet}
\end{figure}

\begin{figure}[h!]
\centering
\includegraphics[]{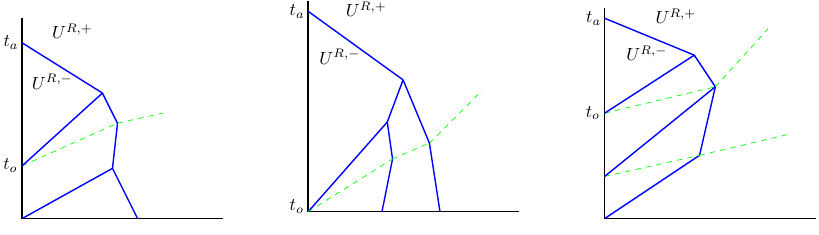}\caption{Examples of \textit{backward wave branch} (the solid blue lines) of a returning wave $\Sret_\rho$ on an outgoing road. Green dashed lines indicate $w$-waves.}
\label{fig:outRet}
\end{figure}

\begin{remark}\label{remark:ret-onde-rho} A returning wave $\Sret$ is always a $\rho$-wave (therefore we occasionally use the notation  $\Sret_\rho$).
In fact, only $\rho$-waves traveling on outgoing roads have the potential to change their speed sign and return from the right to the left (towards the junction). Conversely, $w$-waves can interact with the junction on incoming roads, but these waves cannot be considered ``returning waves'' because they consistently move with a positive speed and are never emitted from the junction.
\end{remark}

\begin{figure}[h!]
\centering
\includegraphics[]{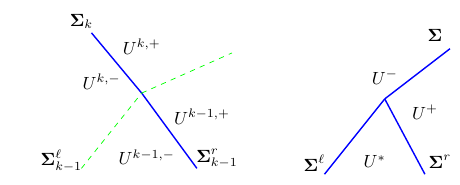}
\caption{Left: A possible configuration of an interaction between a $w$-wave (left, dashed green line) and a $\rho$-wave (right, solid blue line), $U^{k, \pm}=(\rho^{k,\pm},w^{k})$, $U^{k-1, \pm}=(\rho^{k-1,\pm},w^{k-1})$.
Right: A possible configuration of an interaction between two $\rho$-waves (blue lines), $U^\pm=(\rho^\pm,\cdot)$.}
\label{fig:conf1-conf2}
\end{figure}
%
Let us introduce the concept of level of waves.
\begin{defn}[Level of waves in the backward wave tree of a returning wave $\Sigma^{R; t_0, t_a}$]\label{def:order-waves}
\label{def:order}
Given a returning wave $\Sigma^{R; t_o, t_a}$ with original time $t_o>0$ and absorption time $t_a>0$, let us introduce a notation for the waves belonging to its backward wave tree. More precisely, we enumerate such waves starting from the bottom of the tree, namely from the waves at time $t_o$.
\begin{itemize}
\item All the waves traveling at time $t_o$ are called waves of level $1$ and they are denoted by $\Sigma_1$.
If a wave $\Sigma_1$ of level 1 interacts with a $w$-wave, then the interaction generates waves of level 2, which are denoted by $\Sigma_2$. 
\item If a wave $\Sigma_k$ of level $k \ge 1$ interacts with a $w$-wave, then the interaction generates waves $\Sigma_{k+1}$ of order $k+1$. 
\item In turn, if a wave of order $k \ge 1 $ interacts with a $\rho$-wave, then the level of the daughters does not change, namely the interaction generates waves $\Sigma_k$ of order $k$.
\end{itemize}
\noindent In summary, the level of the new waves changes only if one of the interacting waves is a $w$-wave. 
\begin{itemize}
\item Finally, let $K\ge 1$ be the maximum among the $k$'s of the waves in the backward wave tree of $\Sret$. We say that $\Sret$ is of order $K$.
\end{itemize}
\end{defn}
\noindent \textbf{Notation}: regardless of the precise level, the mothers of a wave $\Sigma$ are denoted by $\Sigma^\ell$ (left) and $\Sigma^r$ (right). 
\begin{lemma}\label{lemma:deltaQ} Consider a $\rho$-wave $\Sigma=(U^-,U^+)$, with left and right states $U^{\pm}=(\rho^\pm, w^\pm)$ such that $w^-=w^+$ (recall that $w$ is conserved through a $\rho$-wave). 
The flux variations $\delta_+Q_{\Sigma}$ (and $\delta_-Q_{\Sigma}=-\delta_+Q_{\Sigma}$) can be expressed recursively in terms of:
\begin{itemize}
\item[(a)] only the flux variations $\delta_+Q_{\Sigma^\ell}, \delta_+Q_{\Sigma^r}$ (and $\delta_-Q_{\Sigma^\ell}, \delta_-Q_{\Sigma^r}$) of the mothers if there has been no change of level (no interaction with $w$-waves); 
\item[(b)] the sum of the flux variations $\delta_+Q_{\Sigma^\ell}, \delta_+Q_{\Sigma^r}$ (and $\delta_-Q_{\Sigma^\ell}, \delta_-Q_{\Sigma^r}$) of the mothers and a term which is proportional to the jump of $w$, namely $(w^{k}-w^{k-1})$, if there has been an interaction with a $w$-wave (which is responsible of the change of level from $\Sigma=\Sigma_k$ for some $k\ge 1$ to $\Sigma_{k-1}$).
 \end{itemize}
\end{lemma}
\begin{proof}
To trace the \textit{backward wave tree}, we distinguish the two following situations: 
\begin{itemize}
\item[(i)] (Figure \ref{fig:conf1-conf2}, Left): $\Sigma_{k}=(U^{k, -},U^{k, +})$ has been generated by the interaction of a $w$-wave  $\Sl_{k-1}=(U^{k, -},U^{k-1, -})$ and a $\rho$-wave $\Sr_{k-1}=(U^{k-1, -},U^{k-1, +})$, traveling one behind the other (as showed in Lemma \ref{lemma:VelOndeRho}, $\rho$-waves are slower than $w$-waves). It then holds $V(\rho^{k,-},w^{k,-})=V(\rho^{k-1,-},w^{k-1, -})$ and $w^{k-1, -}=w^{k-1, +}=w^{k-1}$. The interaction of these two waves generates not only $\Sigma_{k}$ but also a $w$-wave $(U^{k, +},U^{k-1, +})$ with $w^{k,-}=w^{k,+}=w^{k}$ and $V(\rho^{k,+},w^{k})=V(\rho^{k-1,+},w^{k-1})$.
Then
\begin{align*}
Q(\rho^{k,+},w^{k}) - Q(\rho^{k,-},w^{k}) &= (\rho^{k,+}-\rho^{k-1,+})V(\rho^{k,+},w^{k}) + Q(\rho^{k-1,+},w^{k-1})-Q(\rho^{k-1,-},w^{k-1})
\\
& \quad + (\rho^{k-1,-}-\rho^{k,-})V(\rho^{k,-},w^{k}),
\end{align*}
where, by Lemma \ref{lem:monotone}, it holds that $\rho^{k,-} < \rho^{k,+}$ if and only if $\rho^{k-1,-}<\rho^{k-1,+}$. That is: the daughter $\rho$-wave 
$(U^{k, -},U^{k, +})$ is a shock (rarefaction) if and only if the mother $\rho$-wave $(U^{k-1, -},U^{k-1, +})$ is a shock (rarefaction) as well. 
Moreover, by Lemma \ref{lem:orderV}, we can write, for some $\tilde w_1, \tilde w_2$ between $w^{k-1}$ and $w^{k}$, some $\tilde \rho_1$ between $\rho^{k-1, +}$ and $\rho^{k, +}$ and some $\tilde \rho_2$ between $\rho^{k-1, -}$ and $\rho^{k, -}$ that
\begin{align*}
\delta_+Q_{\Sigma_{k}} =& 
- \frac{\de_w V (\rho^{k-1,+}, \tilde w_1)}{\de_\rho V(\tilde \rho_1, w^{k})}V(\rho^{k,+},w^{k}) (w^{k}-w^{k-1})
+ Q(\rho^{k-1,+},w^{k-1})-Q(\rho^{k-1,-},w^{k-1})
\\
& - \frac{\de_w V (\rho^-_{k}, \tilde w_2)}{\de_\rho V(\tilde \rho_1, w^{k-1})}V(\rho^{k,-},w^{k-1}) (w^{k-1}-w^{k})
\\
=&  (w^{k} - w^{k-1})\left(- \frac{\de_w V (\rho^{k-1,+}, \tilde w_1)}{\de_\rho V(\tilde \rho_1, w^{k})}V(\rho^{k,+},w^{k}) + \frac{\de_w V (\rho^-_{k}, \tilde w_2)}{\de_\rho V(\tilde \rho_1, w^{k-1})}V(\rho^{k,-},w^{k-1})\right)
\\
& + \delta_+Q_{\Sigma_{{k-1}}}.
\end{align*}

\item[(ii)] (Figure \ref{fig:conf1-conf2}, Right): $\Sigma$ (no need of subscript $k$ in this case) has been generated by the interaction of two consecutive $\rho$-waves $\Sl:=(U^-,U^*)$ and $\Sr:=(U^*,U^+)$.
In this case, we can directly write  
\begin{align*}
 \delta_+Q_{\Sigma} &= Q(\rho^+,\cdot)-Q(\rho^*,\cdot)
+Q(\rho^*,\cdot)-Q(\rho^-,\cdot) = \delta_+Q_{\Sl} + \delta_+Q_{\Sr}. 
\end{align*}
\end{itemize}
The process works recursively by applying the same reasoning: in Case (i), to the right wave $\Sr_{k-1}$; in Case (ii), applying it to both waves $\Sl$ and $\Sr$.
\end{proof}


The following proposition provides the most general estimate of the flux variation due to returning waves, without distinguishing between incoming and outgoing roads. Refined estimates specific to incoming and outgoing roads will be presented later.

\begin{prop}\label{lemma:stima} Let $\Sret$ be a returning wave of   original and absorption times $t_o$ and $t_a$ respectively and of order $K$, according to Definition \ref{def:order-waves}.
The flux variation $\delta_+ Q_{\Sret}$ ($\delta_- Q_{\Sret} $) at the absorption time $t_a$ of $\Sret$ on an incoming road (outgoing road) is estimated as follows:
\begin{align}\label{eq:deltaQret<}
\delta_+ Q_{\Sret}   \leq &  
\,C_* TV_{tree}^K + \sum_{r=1}^{N_\rho^{t_o}} [\delta_+Q_{\Sigma_{r}}]_+,
\end{align}
where
\begin{align}\label{def:tvtree}
TV_{tree}^K:=\sum_{k=2}^K  \left|w^{k} - w^{k-1}\right|, \qquad N_\rho^{t_o}=\mathrm{card}(\{\Sigma : \Sigma \; \text{is of level 1 at the original time} \; t_o\}),
\end{align}
and
\begin{equation}\label{eq:Cstar}
C_*:= 2\,\sup_{\rho,\rho', \rho'' \in [0, \rho_\text{max}]}\sup_{w, w' \in [w_L, w_R]}\left(- \frac{\de_w V(\rho, w)  }{\de_\rho V(\rho', w') }\right) \cdot V(\rho'', w')>0
\end{equation}
and $[\cdot]_+$ denotes the positive part.
\end{prop}
\begin{remark}
    Notice that $TV_{tree}^K$ takes into account the $w$-variation of only a fixed number $K$ of waves, where $K$ is the number of levels (or, equivalently, the number of $w$-waves) in the backward wave tree of $\Sret$ (see Definition \ref{def:order-waves}). Moreover, $N_\rho^{t_o}$ is the number of waves at the root of the backward wave tree of $\Sret$.
\end{remark}
\begin{proof}
Applying Lemma \ref{lemma:deltaQ} to the returning wave $\Sret$ yields 
\begin{align}\label{eq:deltaQret=}
\delta_+ Q_{\Sret}  = &  
\sum_{k=2}^K  C_{(k-1,k)}\, \left(w^{k} - w^{k-1}\right)
+ \sum_{r=1}^{N_\rho^{t_o}} \delta_+Q_{\Sigma_{r}}.
\end{align}
In fact, when tracing the backward wave tree of $\Sret$ from $t_a$ to $t_o$, both types of iterations described in Lemma \ref{lemma:deltaQ} occur multiple times, leading to the equality \eqref{eq:deltaQret=}. Notably, one of the $\rho$-waves emanates from the junction. The constants $C_{(k-1, k)}$ are explicitly specified in Lemma \ref{lemma:deltaQ}, and straightforward calculations allow us to estimate them by $C_*$. Consequently, we obtain \eqref{eq:deltaQret<}.
\end{proof}

\subsubsection{Refined estimates of flux variation due to returning waves on incoming roads}
Below, we present estimates of the flux variation at the junction due to a returning wave from an incoming road. As mentioned in Remark \ref{remark:ret-onde-rho}, these returning waves are exclusively $\rho$-waves. For possible configurations, see Figure \ref{fig:inRet}.


\begin{prop}\label{prop:ret-wave-in}
Let $\Sret_{i}=(U^{R,-}_i, U^{R, +}_i)$, with $U^{R, \pm}_i=(\rho^{R, \pm}_i, w^{R, \pm}_i)$ and $w^{R,-}_i=w^{R,+}_i$, be a returning wave traveling on an incoming road $i \in \{1, \cdots, n\}$ and interacting with the junction $J$ at time $t_a$. 
\begin{itemize}
\item If one of the following configurations occurs: 
\begin{itemize}
\item[(a)] $\Sret_{i}$ is a shock wave;
\item[(b)] $\Sret_{i}$ is a rarefaction wave with $\widetilde\rho^{R,-}_i < \rho^{R, +}_i < \sigma (w^{R, +}_i)=\sigma (w^{R, -}_i)<\rho^{R, -}_i$, where $\widetilde\rho^{R,-}_i$ is determined by the identity $Q(\widetilde\rho^{R,-}_i,w^{R,-}_i) = Q(\rho^{R,-}_i,w^{R,-}_i)$;
\item[(c)] the backward characteristic tree of the returning wave $\Sret_{i}$  includes only $\rho$-waves,
\end{itemize}
 then it holds
\begin{align}\label{eq:ret-in-neg}
\delta_- Q_{\Sret_{i}}=Q(\rho^{R, -}_i, w^{R, -}_i)- Q(\rho^{R, +}_i, w^{R, +}_i)<0.
\end{align}
\item If $\Sret_{i}$ is a rarefaction wave with $\rho^{R, +}_i < \widetilde\rho^{R,-}_i < \sigma (w^{R,-}_i)<\rho^{R, -}_i$, where $\widetilde\rho^{R,-}_i$ is determined by the identity $Q(\widetilde\rho^{R,-}_i,w^{R,-}_i) = Q(\rho^{R,-}_i,w^{R,-}_i)$,and 
it is generated by the interaction of a $w$-wave $\Sl=(U^{0,-}_i, U^{0,+}_i)$ (traveling along the arc) with $U^{0,-}_i=U^{R,-}_i$ ($w^{0,-}_i = w^{R, -}_i$) and the $\rho$-wave $\Sr=(U^{0,+}_i, \hat U_i)$ originated from the junction at time $t_o$, then 
   \begin{equation}
       0<\delta_- Q_{\Sret_{i}}=Q(\rho^{R, -}_i, w^{R, -}_i)- Q(\rho^{R, +}_i, w^{R, +}_i)< C_\star |w^{0, +}_i-w^{0,-}_i|. \label{eq:ret-inc-case2}
   \end{equation}
\item In the most general case where $\Sret_{i}$ is a rarefaction wave such that $\delta_- Q_{\Sret_{i}}>0$,
then \eqref{eq:deltaQret<} holds, i.e.
\begin{align}\label{eq:prop_in_stima}
\delta_- Q_{\Sret_i}   \leq &  
\,C_* TV_{tree}^K + \sum_{r=1}^{N_\rho^{t_o}} [\delta_-Q_{\Sigma_{r}}]_+,
\end{align}
 where $TV_{tree}^K$, $N_\rho^{t_o}$ are given in \eqref{def:tvtree} and $C_*$ is \eqref{eq:Cstar}.
\end{itemize}
\end{prop}
\begin{proof}
We will prove the statements of the propositions point by point.
\begin{itemize}
\item First, if the returning wave $\Sret$ defined in Definition \ref{def:returnwave}, is a shock with positive speed then $\rho^{R, -} < \rho^{R, +}$ and $Q(\rho^{R, -}, w^{R, -}) < Q(\rho^{R, +},w^{R, +})$, from which \eqref{eq:ret-in-neg} follows. Likewise if $\Sret$ is a rarefaction with the ordering  $\widetilde\rho^{R,-} < \rho^{R, +} < \sigma (w^{R, +})=\sigma (w^{R, -})<\rho^{R, -}$.
\noindent Next, if the backward characteristic tree of $\Sret_i$ consists solely of $\rho$-waves coming from the left (interacting with the waves forming the backward characteristic branch of $\Sret_i$), then we are in the case of a standard LWR (Lighthill-Whitham-Richards) model, and we can rely on \cite[Proposition 4.1]{dellemonache2018CMS}. In fact, 
by Definition \ref{def:returnwave}, in the backward characteristic branch of the returning wave $\Sret_i$, there exists a wave which was emanated from the junction at the original time $t_o$. By the definition of original time $t_o$ (maximum time at which a wave from the backward branch of $\Sret_i$ originated from $J$), the value $U^{R, +}$ coincides with the right value of 
such a $\rho$-wave generated at $J$ at $t_o$, and then $\rho^{R, +} > \sigma (w^{R, +})$. Since $\Sret_i$ has a positive speed, we deduce that $\rho^{R, -} < \sigma (w^{R, +})$, from which it follows that $\rho^{R, -}-\rho^{R, +}<0.$ This implies that $\Sret_i$ is a shock, and therefore \eqref{eq:ret-in-neg} holds. 
\item This is a special case of Lemma \ref{lemma:deltaQ}-(a), where the only $\rho$-waves involved in the backward wave tree of the returning wave $\Sret_{i}$ are the ones originated from the junction. Then, it holds
    \begin{align*}
Q(\rho^{R,-},w^{R,-}) - Q(\rho^{R,+},w^{R,+}) &= (\rho^{R,-}-\rho^{0,+})V(\rho^{R,-},w^{0,-}) + Q(\rho^{0,+},w^{0,+})-Q(\hat\rho_i,w^{0,+})
\\
& \quad + (\hat\rho_i-\rho^{R,+})V(\rho^{R,+},w^{0,+}).
\end{align*}
    Since $\Sr$ originated from the junction, then $\hat\rho_i>\sigma(w^{R,+}_i)$. Moreover,
    the condition $\rho^{R,+}_i < \rho^{R, -}_i$ implies $\hat\rho_i<\rho^{0,+}_i$ by Lemma \ref{lem:monotone} and $Q(\rho^{0,+},w^{0,+})-Q(\hat\rho_i,w^{0,+}) < 0$. Hence \eqref{eq:ret-inc-case2}.
\item Finally, in the most general case, if the returning wave is a rarefaction with $Q(\rho^{R, -}, w^{R, -}) - Q(\rho^{R, +},w^{R, +})>0$, by applying Proposition \ref{lemma:stima} we get \eqref{eq:prop_in_stima}.
\end{itemize}
\end{proof}
%
%

\subsubsection{Refined estimates of flux variation due to returning waves on outgoing roads}
Below, we discuss the flux variation at the junction caused by a returning wave on an outgoing road. For possible configurations, see Figure \ref{fig:outRet}.
\begin{prop}\label{prop:ret-wave-out}
Let $\Sret_j=(U^{R,-}_j, U^{R, +}_j)$, with $U^{R, \pm}_j=(\rho^{R, \pm}_j, w^{R, \pm}_j)$ and $w^{R,-}_j=w^{R,+}_j=w^{R}_j$, be a returning wave traveling on an outgoing road $j \in \{n+1, \cdots, n+m\}$ and interacting with the junction $J$ at time $t_a$. Then it is a shock wave such that 
\begin{align}\label{eq:ret-out-neg}
\delta_+ Q_{\Sret_j}=Q(\rho^{R, +}_j, w^{R, +}_j)- Q(\rho^{R, -}_j, w^{R, -}_j)<0.
\end{align}
\end{prop}
\begin{proof}
First recall that by Proposition \ref{prop:outgoing}, on outgoing roads we do not allow vertical shocks to occur at the junction. Then, the waves always come out of the junction in pairs, first the $w$-wave and then the $\rho$-wave, or only a $\rho$-wave is generated. 
Therefore, by the definition of original time $t_o$ (maximum time at which a wave from the backward branch of $\Sret_i$ originated from J), the value $U^{R, -}_j$ coincides with the left value of 
the $\rho$-wave generated at $J$ at $t_o$ (traveling alone or preceded by a $w$-wave), and then $\rho^{R, -}_j < \sigma (w^{R}_j)$. Since $\Sret_j$ has a negative speed, we deduce that $\rho^{R, +}_j > \sigma (w^{R}_j)$, from which it follows that $\rho^{R, -}_j<\rho^{R, +}_j$. This implies that $\Sret_j$ is a shock, and therefore \eqref{eq:ret-out-neg} holds. 
\end{proof}

\bigskip

\section{Existence of solution to the Cauchy problem}\label{sec:theoEx}
Given initial data of bounded variation, one can solve Cauchy problems by constructing approximate solutions via Wave Front Tracking (WFT). To prove the convergence of WFT approximations, it is necessary to estimate the number of waves, the number of wave interactions, and to provide estimates on the total variation of the approximate solutions. 
We provide the following existence result.

\begin{theorem}\label{thm:esistenza}
Let us consider a junction $J$ with $n$ incoming and $m$ outgoing roads $I_{\road}=[a_{\road},b_{\road}]\subset\R$, $\road = 1,\dots,n+m$, possibly with $a_{\road}=-\infty$ and $b_{\road}=+\infty$. Consider the network identified by the couple $(\edge,\vert)$ where $\edge$ is a finite collection of roads $I_{\road}$, $s=1, \cdots, n+m$ - specifically, $n$ incoming and $m$ outgoing roads - and $\vert$ is a finite collection of junctions $J$.
If a Riemann Solver $\rsolv$ in Definition \ref{def:rsolv} satisfies properties \ref{rsp1} -- \ref{rsp4}, then the collection of $n+m$ systems of equations for each road indexed by $s$ \eqref{eq:cauchyWFT}, endowed with initial data $(\rho_s(x, 0), y_s(x,0))$ belonging to the space of functions with bounded variation of each road $BV(I_\road)$ - where $y_s=\rho_s w_s$ - admits an entropy weak solution on the network $(\edge,\vert)$ in the sense of Definition \ref{def:sol-net}.
\end{theorem}

The following is a direct consequence of the above result and Theorem \ref{thm:propP}.

\begin{corollario}
Consider a network $(\edge,\vert)$ as in Theorem \ref{thm:esistenza}, composed of $n=2$ incoming and $m=2$ outgoing roads. Then the associated Cauchy problem admits an entropy weak solution that can be constructed by a Wave Front-Tracking (WFT) approximation based on the Riemann Solver $\rsolvG$ defined in Section \ref{sec:NinM}. 
\end{corollario}

\begin{proof}[Proof of Theorem \ref{thm:esistenza}]
We adapt the proof of \cite[Theorem 4.1]{dellemonache2018CMS} for scalar equations to the case of systems of equations. It is based on first estimating the total variation in time of $\bar h$ and then that of $\Gamma$. Let $\text{PV}$ and ${NV}$ denote the positive and negative variations of a function, respectively. We have the following relations:
\begin{align}
\text{TV}(\bar h) &= \text{PV}(\bar h) + \text{NV}(\bar h), \\
\text{PV}(\bar h) &= \text{PV}^O(\bar h) + \text{PV}^R(\bar h),
\end{align}
where $\text{PV}^O$ is the variation due to interactions of the original waves with the junction $J$, and $\text{PV}^R$ is the variation due to returning waves as in Definition \ref{def:returnwave}. Observe from Proposition \ref{prop:ret-wave-out} that returning waves on outgoing roads always generate a negative variation of the flux $Q$ at the junction. By property \ref{rsp3}, the function $\bar h$ on outgoing roads due to returning waves decreases, and therefore its variation is only negative, i.e., 
\[
\text{TV}_\text{out}^R(\bar h) = \text{NV}_\text{out}^R(\bar h).
\]
As a consequence, we estimate the positive variation due to returning waves only for incoming roads, $\text{PV}_\text{in}^R(\bar h)$. We can thus rely on Proposition \ref{prop:ret-wave-in}, yielding, for some constant $C_1$:
\begin{align*}
\text{PV}_\text{in}^R(\bar h) &\le \left(V^\text{max} + \sup_{(\hat  \rho, \hat w) \in (0, \rho_\text{max}] \times [w_L, w_R]}  \frac{|f(\hat \rho, \hat w)|}{\hat \rho^2}\right) \text{TV}_{\rho}(t) + C_* \text{TV}_w(t) \\
&\le C_1 \left(V^\text{max} + \sup_{(\hat  \rho, \hat w) \in (0, \rho_\text{max}] \times [w_L, w_R]}  \frac{|f(\hat \rho, \hat w)|}{\hat \rho^2}\right) \text{TV}_{\rho}(0) + C_* \text{TV}_w(0),
\end{align*}
where the second inequality holds since $\text{TV}_\rho(t)$ and $\text{TV}_w(t)$ only include original waves, in accordance with the Temple structure of the system \eqref{eq:GSOM1}. Moreover, using again the Temple structure, we obtain
\begin{align*}
\text{PV}^O(\bar h) &\le C_2 \left( \text{TV}(\rho(0)) + \text{TV}_w(0) \right)
\end{align*}
for some constant $C_2$. Therefore, $\text{PV}_\text{in}(\bar h)$ is bounded, as is $\text{TV}(\bar h)$. Altogether, it follows from property \ref{rsp2} that, denoting by $\tau \in \text{Int}$ an interaction time, 
\begin{align*}
\text{TV}_Q(t) &\le \text{TV}_Q(0) + \sum_{\tau \in \text{Int}, \, \tau \le t} \Delta \text{TV}_Q(\tau) \\
&\le \text{TV}_Q(0) + \max_{(\rho, w) \in [0, \rho_\text{max}] \times [w_L, w_R]} |Q_\rho (\rho, w )| \text{TV}_\rho(0)  + \max_{(\rho, w) \in [0, \rho_\text{max}] \times [w_L, w_R]} |Q_w (\rho, w)| \text{TV}_w(0) \\
&\quad + C (\text{TV}(\Gamma) + C \text{TV}(\bar h)) \\
&\le \text{TV}_Q(0) + \max_{(\rho, w) \in [0, \rho_\text{max}] \times [w_L, w_R]} |Q_\rho (\rho, w )| \text{TV}_\rho(0) + \max_{(\rho, w) \in [0, \rho_\text{max}] \times [w_L, w_R]} |Q_w (\rho, w)| \text{TV}_w(0) \\
&\quad + C T\text{V}(\bar h),
\end{align*}
where the final inequality follows from \ref{rsp3}. Using the previous estimate for $\text{TV}(\bar h)$ completes the proof by relying on a WFT approximation.

\end{proof}

\paragraph{Acknowledgment} R.B. and M.B. acknowledge financial support by the Italian Ministry of University and Research, PRIN PNRR P2022XJ9SX ``Heterogeneity on the Road - Modeling, Analysis, Control'', PNRR Italia Domani, funded
by the European Union under NextGenerationEU, CUP B53D23027920001.
The endowment of the Lopez Chair supported B.P.'s research and he thanks the Institute for Advanced Study of Princeton
for the hospitality.



\appendix
\section{Appendix: proof of Theorem \ref{thm:propP}}\label{appendice}
The aim of this appendix is to prove Theorem \ref{thm:propP}, therefore we show that $\rsolvG$ satisfies properties \ref{rsp1} -- \ref{rsp4} in the case of two incoming and two outgoing roads at the junction. Let us begin fixing the notation. The priority rule $\rettar$ is defined by the vector $(\puno,\pdue)$ with $\puno+\pdue=1$, while the matrix of distribution is
\begin{equation*}
	A=\begin{pmatrix}
	\at & \aq\\
	\bt & \bq
	\end{pmatrix}
\end{equation*}
with $\at+\aq=1$ and $\bt+\bq=1$.
The conservation of $\rho$ in \eqref{eq:NinMq} implies
\begin{equation}\label{eq:q3q4}
\begin{split}
	\at\qS_{1}+\bt\qS_{2}&=\qS_{3}\\
	\aq\qS_{1}+\bq\qS_{2}&=\qS_{4}.	
\end{split}
\end{equation}
We denote by $\duno=\dem{\rho_{1}}{w_{1}}$, $\ddue=\dem{\rho_{2}}{w_{2}}$, $\stre=\supp{\rhomorto_{3}(\vpiu_{3},\wS_{3})}{\wS_{3}}$ and $\sq=\supp{\rhomorto_{4}(\vpiu_{4},\wS_{4})}{\wS_{4}}$, where $\wS_{3}$ and $\wS_{4}$ are determined by \eqref{eq:wGeneric}, and $\rhomorto_{j}(\vpiu_{j},\wS_{j})$, $j=3,4$, is given in Definition \ref{eq:rhomorto}. The quantities $\duno$, $\ddue$, $\stre$ and $\sq$ define the sets $\omegaInc$ and $\omegaOut$, see \eqref{eq:omegaN} and \eqref{eq:omegaM}.
Finally, we denote by
\begin{equation}\label{eq:rettar22}
	\rettar: \begin{cases}
		\quno=h\puno\\ 
		\qdue = h\pdue.
	\end{cases}
\end{equation}
the priority rule straight line and by
\begin{equation}\label{eq:z3z4}
\begin{split}
	&\rettas_{3}: \at\quno+\bt\qdue=\stre\\
	&\rettas_{4}: \aq\quno+\bq\qdue=\sq,
\end{split}
\end{equation}
the straight lines that maximize the outgoing flux. 

\begin{prop}
$\rsolvG$ satisfies property \ref{rsp1}.
\end{prop}
\begin{proof}
Let us consider two states $\uVd$ and $\uVhd$ such that $w_{i}=\hw_{i}$ for $i=1,2$, $w_{j}=\hw_{j}$ for $j=3,4$ and $\rho_{i}=\hrho_{i}$ $(\rho_{j}=\hrho_{j})$ whenever either $U_{i}$ or $\hu_{i}$ $(U_{j}$ or $\hu_{j})$ is a bad datum. This implies that for every bad datum we have 
\begin{align*}
d_{i}&=\mathring d_{i}=Q(\rho_{i},w_{i})\\
s_{j}&=\mathring s_{j}=Q(\rho_{j},w_{j}).
\end{align*}
On the other hand, for any good datum, since $w_{i}=\hw_{i}$ for $i=1,2$ and $w_{j}=\hw_{j}$ for $j=3,4$, we have 
\begin{align*}
d_{i}&=\mathring d_{i}=\qmax(w_{i})\\
s_{j}&=\mathring s_{j}=\qmax(w_{j}).
\end{align*}
Therefore, $\omegaInc=[0,\duno]\times[0,\ddue]=[0,\hduno]\times[0,\hddue]$ and $\omegaOut=[0,\stre]\times[0,\sq]=[0,\hstre]\times[0,\hsq]$. Since the Riemann Solver $\rsolvG$ only depends on the priority rule, the matrix $A$ and the sets $\omegaInc$ and $\omegaOut$, then it holds
\begin{align*}
	\rsolvG\uV=\rsolvG\uVh.
\end{align*}
%
\end{proof}

We now consider properties \ref{rsp2}, \ref{rsp3} and \ref{rsp4}.
For convenience, we work in the $(\quno,\qdue)$-plane. Starting from an equilibrium for $\rsolvG$, we estimate the variation of the flux and of $w$ sending a wave on each one of the roads. Our aim is to show that we can control the variation of $Q$ and $w$. Let us begin with \ref{rsp2} and \ref{rsp3}; starting from a certain equilibrium $(U_{1}, U_{2},U_{3},U_{4})$, we send a wave $\ondarho_{i}$ (or $\ondarho_{j}$), with corresponding flux $\ondaq_{i}$ (or $\ondaq_{j}$), and we compute the solution of $\rsolvG$, $(\uS_{1}, \uS_{2},\uS_{3},\uS_{4})$, with corresponding fluxes $(\qS_{1},\qS_{2},\qS_{3},\qS_{4})$. We are interested in computing (see \eqref{eq:funzionali})
\begin{align}
	\nonumber\Delta\Gamma(\tb)&= (\qS_{1}- q_{1})+(\qS_{2}- q_{2})\\
	\label{eq:tvw}\Delta\tv_{w} & = |\wS_{3}-w_{3}|+|\wS_{4}-w_{4}|.
\end{align}
The variation of the flux is
\begin{equation*}
	\Delta\tv_{Q}(\tb)  = |\qS_{i}-\ondaq_{i}|+|\qS_{\ell}- q_{\ell}|+|\qS_{3}-q_{3}|+|\qS_{4}-q_{4}|-|\ondaq_{i}- q_{i}|
\end{equation*}
if the interacting wave is in the incoming road $I_{i}$, with $\ell=3-i$, and 
\begin{equation*}
	\Delta\tv_{Q}(\tb)  = |\qS_{1}- q_{1}|+|\qS_{2}- q_{2}|+|\qS_{j}-\ondaq_{j}|+|\qS_{k}-q_{k}|-|\ondaq_{j}-q_{j}|
\end{equation*}
if the interacting wave is in the outgoing road $I_{j}$, with $k=7-j$. Note that in \eqref{eq:tvw} we only have variations of $w$ in the outgoing roads since $w$ is a Riemann invariant and thus $\wS_{i}=w_{i}$, $i=1,2$.
For property \ref{rsp4}, in the case of a wave $(\ondarho_{i},\ondaw_{i})$ along an incoming road $I_{i}$, $i\in\{1,2\}$, the variation in $w$ becomes
\begin{equation*}
	\Delta\tv_{w}  = |\wS_{3}-w_{3}|+|\wS_{4}-w_{4}|-|\ondaw_{i}-w_{i}|.
\end{equation*}

The computations related to $\tv_{w}$ showed two possible configurations to obtain the desired estimates. First of all, we observe that by \eqref{eq:wGeneric} we have
\begin{align*}
	\wS_{3} & =w_{2}+\frac{\qS_{1}}{\qS_{3}}\at(w_{1}-w_{2})=w_{1}+\frac{\qS_{2}}{\qS_{3}}\bt(w_{2}-w_{1})\\
	w_{3} & =w_{2}+\frac{ q_{1}}{q_{3}}\at(w_{1}-w_{2})=w_{1}+\frac{ q_{2}}{q_{3}}\at(w_{2}-w_{1}),
\end{align*}
where we choose one of the two formulations for $\wS_{3}$ and $w_{3}$, depending on which one is the more convenient from data. 
Analogously for $\wS_{4}$ and $w_{4}$.
More generally we have 
\begin{equation}\label{eq:wTecn}
	\begin{split}
		\wS_{3} & =w_{\ell}+\frac{\qS_{i}}{\qS_{3}}\aij{3}{i}(w_{i}-w_{\ell}), \qquad 
		w_{3}  =w_{\ell}+\frac{ q_{i}}{q_{3}}\aij{3}{i}(w_{i}-w_{\ell})\\
		\wS_{4} & =w_{\ell}+\frac{\qS_{i}}{\qS_{4}}\aij{4}{i}(w_{i}-w_{\ell}), \qquad 
		w_{4}  =w_{\ell}+\frac{ q_{i}}{q_{4}}\aij{4}{i}(w_{i}-w_{\ell})
	\end{split}
\end{equation}
with $i\in\{1,2\}$ and $\ell=3-i$. In particular, when we send a wave $(\ondarho_{i},\ondaw_{i})$ on road $i$, we have
\begin{align}\label{eq:wTecnT}
	\wS_{3} & =w_{\ell}+\frac{\qS_{i}}{\qS_{3}}\aij{3}{i}(\ondaw_{i}-w_{\ell}), \qquad \wS_{4}  =w_{\ell}+\frac{\qS_{i}}{\qS_{4}}\aij{4}{i}(\ondaw_{i}-w_{\ell}).
\end{align}

\paragraph{Configuration 1.}
The following configuration is obtained when $\qS_{1}= q_{1}$ or $\qS_{2}= q_{2}$. Let $\qS_{i}= q_{i}$, $i=1,2$, and set $\ell=3-i$, then $|\qS_{3}-q_{3}|=\aij{3}{\ell}|\qS_{\ell}- q_{\ell}|$. Sending a wave $\ondarho_{r}$, $r=1,\dots,4$, by \eqref{eq:wTecn} we have
\begin{align}
	|\wS_{3}-w_{3}|&=\frac{\aij{3}{i} q_{i}|w_{1}-w_{2}||q_{3}-\qS_{3}|}{\qS_{3}q_{3}}=\frac{\at\bt q_{i}|w_{1}-w_{2}|}{\qS_{3}q_{3}}|\qS_{\ell}- q_{\ell}|\label{eq:w3T1}\\
	|\wS_{4}-w_{4}|&=\frac{\aij{4}{\ell} q_{i}|w_{1}-w_{2}||q_{4}-\qS_{4}|}{\qS_{4}q_{4}}=\frac{\aq\bq q_{i}|w_{1}-w_{2}|}{\qS_{4}q_{4}}|\qS_{\ell}- q_{\ell}|.\label{eq:w4T1}
\end{align}
Sending a wave $(\ondarho_{i},\ondaw_{i})$, $i\in\{1,2\}$ and $\ell=3-i$, by \eqref{eq:wTecn} and \eqref{eq:wTecnT} we have
\begin{align}
	\nonumber|\wS_{3}-w_{3}|&=\frac{\aij{3}{i}|\qS_{i}q_{3}(\ondaw_{i}-w_{\ell})- q_{i}\qS_{3}(w_{i}-w_{\ell})|}{\qS_{3}q_{3}}\\
	\nonumber&=\frac{\aij{3}{i}|\qS_{i}q_{3}(\ondaw_{i}-w_{i}+w_{i}-w_{\ell})- q_{i}\qS_{3}(w_{i}-w_{\ell})|}{\qS_{3}q_{3}}\\
	\nonumber&\leq \frac{\aij{3}{i}\qS_{i}|\ondaw_{i}-w_{i}|}{\qS_{3}}+\frac{\aij{3}{i}|w_{1}-w_{2}||\qS_{i}q_{3}- q_{i}\qS_{3}|}{\qS_{3}q_{3}}\\
	&=\frac{\aij{3}{i}\qS_{i}|\ondaw_{i}-w_{i}|}{\qS_{3}}+\frac{\at\bt q_{\ell}|w_{1}-w_{2}|}{\qS_{3}q_{3}}|\qS_{i}- q_{i}|\label{eq:wt3T1}\\
	|\wS_{4}-w_{4}|&\leq\frac{\aij{4}{i}\qS_{i}|\ondaw_{i}-w_{i}|}{\qS_{4}}+\frac{\aq\bq q_{\ell}|w_{1}-w_{2}|}{\qS_{4}q_{4}}|\qS_{i}- q_{i}|.\label{eq:wt4T1}
\end{align}

\paragraph{Configuration 2.}
The following configuration is obtained when $\qS_{1}\neq q_{1}$ and $\qS_{2}\neq q_{2}$. Sending a wave $\ondarho_{r}$, $r=1,\dots,4$, by \eqref{eq:wTecn} we have
\begin{align}
	\nonumber|\wS_{3}-w_{3}|&=\frac{\aij{3}{i}|w_{1}-w_{2}||\qS_{i}q_{3}- q_{i}\qS_{3}|}{\qS_{3}q_{3}}\\
	\nonumber&=\frac{\aij{3}{i}|w_{1}-w_{2}||\qS_{i}(\aij{3}{i} q_{i}+\aij{3}{\ell} q_{\ell})- q_{i}(\aij{3}{i}\qS_{i}+\aij{3}{\ell}\qS_{\ell})|}{\qS_{3}q_{3}}\\
	\nonumber&=\frac{\at\bt|w_{1}-w_{2}|}{\qS_{3}q_{3}}|\qS_{1} q_{2}- q_{1}\qS_{2}|\\
	\nonumber&=\frac{\at\bt|w_{1}-w_{2}|}{\qS_{3}q_{3}}|\qS_{1} q_{2}-\qS_{1}\qS_{2}+\qS_{1}\qS_{2}- q_{1}\qS_{2}|\\
	&=\frac{\at\bt|w_{1}-w_{2}|}{\qS_{3}q_{3}}|\qS_{2}(\qS_{1}- q_{1})-\qS_{1}(\qS_{2}- q_{2})|\label{eq:w3T2}\\
	|\wS_{4}-w_{4}|&=\frac{\aq\bq|w_{1}-w_{2}|}{\qS_{4}q_{4}}|\qS_{2}(\qS_{1}- q_{1})-\qS_{1}(\qS_{2}- q_{2})|.\label{eq:w4T2}
\end{align}
Note that, in the case of $\Delta\hh(\tb)\propto(\qS_{1}- q_{1})$, since $\Delta\Gamma(\tb)=(\qS_{1}- q_{1})+(\qS_{2}- q_{2})$, we rewrite \eqref{eq:w3T2} and \eqref{eq:w4T2} as
\begin{align}
	\nonumber|\wS_{3}-w_{3}|&=\frac{\at\bt|w_{1}-w_{2}|}{\qS_{3}q_{3}}|(\qS_{1}+\qS_{2})(\qS_{1}- q_{1})-\qS_{1}\Delta\Gamma(\tb)|\\
	&\leq\frac{\at\bt|w_{1}-w_{2}|}{\qS_{3}q_{3}}|\qS_{1}+\qS_{2}|(|\Delta\Gamma|+|\Delta\hh(\tb)|)\label{eq:w3T3}\\
	|\wS_{4}-w_{4}|&\leq\frac{\aq\bq|w_{1}-w_{2}|}{\qS_{4}q_{4}}|\qS_{1}+\qS_{2}|(|\Delta\Gamma|+|\Delta\hh(\tb)|).\label{eq:w4T3}
\end{align}
In the case of $\Delta\hh(\tb)\propto(\qS_{2}- q_{2})$ we follow similar computations and obtain the same result.

Finally, sending a wave $(\ondarho_{i},\ondaw_{i})$, $i\in\{1,2\}$ and $\ell=3-i$, we have
\begin{align}
	\nonumber|\wS_{3}-w_{3}|&=\frac{\aij{3}{i}|\qS_{i}q_{3}(\ondaw_{i}-w_{\ell})- q_{i}\qS_{3}(w_{i}-w_{\ell})|}{\qS_{3}q_{3}}\\
	\nonumber&=\frac{\aij{3}{i}|\qS_{i}q_{3}(\ondaw_{i}-w_{i}+w_{i}-w_{\ell})- q_{i}\qS_{3}(w_{i}-w_{\ell})|}{\qS_{3}q_{3}}\\
	\nonumber&\leq \frac{\aij{3}{i}\qS_{i}|\ondaw_{i}-w_{i}|}{\qS_{3}}+\frac{\aij{3}{i}|w_{1}-w_{2}||\qS_{i}q_{3}- q_{i}\qS_{3}|}{\qS_{3}q_{3}}\\
	&=\frac{\aij{3}{i}\qS_{i}|\ondaw_{i}-w_{i}|}{\qS_{3}}+\frac{\at\bt|w_{1}-w_{2}|}{\qS_{3}q_{3}}|\qS_{2}(\qS_{1}- q_{1})-\qS_{1}(\qS_{2}- q_{2})|\label{eq:wt3T2}\\
	|\wS_{4}-w_{4}|&\leq\frac{\aij{4}{i}\qS_{i}|\ondaw_{i}-w_{i}|}{\qS_{4}}+\frac{\aq\bq|w_{1}-w_{2}|}{\qS_{4}q_{4}}|\qS_{2}(\qS_{1}- q_{1})-\qS_{1}(\qS_{2}- q_{2})|.\label{eq:wt4T2}
\end{align}
Note that, in the case of $\Delta\hh(\tb)\propto(\qS_{1}- q_{1})$ or $\Delta\hh(\tb)\propto(\qS_{2}- q_{2})$ we rewrite \eqref{eq:wt3T2} and \eqref{eq:wt4T2} as
\begin{align}
	|\wS_{3}-w_{3}|&=\frac{\aij{3}{i}\qS_{i}|\ondaw_{i}-w_{i}|}{\qS_{3}}+\frac{\at\bt|w_{1}-w_{2}|}{\qS_{3}q_{3}}|\qS_{1}+\qS_{2}|(|\Delta\Gamma|+|\Delta\hh(\tb)|)\label{eq:wt3T3}\\
	|\wS_{4}-w_{4}|&\leq\frac{\aij{4}{i}\qS_{i}|\ondaw_{i}-w_{i}|}{\qS_{4}}+\frac{\aq\bq|w_{1}-w_{2}|}{\qS_{4}q_{4}}|\qS_{1}+\qS_{2}|(|\Delta\Gamma|+|\Delta\hh(\tb)|).\label{eq:wt4T3}
\end{align}

\medskip
We divide the proof of \ref{rsp2} -- \ref{rsp4} in three cases, depending on the initial position of the equilibrium.
Since we work in the $(\quno,\qdue)$-plane, we identify the equilibrium  $(U_{1},U_{2},U_{3},U_{4})$ with the corresponding fluxes $( q_{1}, q_{2},q_{3},q_{4})$. 
Therefore, with a slight abuse of notation we will write the equilibrium condition as $\rsolvG( q_{1}, q_{2},q_{3},q_{4})= ( q_{1}, q_{2},q_{3},q_{4})$. Note that this implies that $q_{3}$ and $q_{4}$ satisfy \eqref{eq:q3q4}.
\begin{description}
	\item[Case A:] We start from the equilibrium $( q_{1}, q_{2},q_{3},q_{4})=(\duno,\ddue,\at\duno+\bt\ddue,\aq\duno+\bq\ddue)$.
	\item[Case B:] We start from the equilibrium along one of the straight lines $\quno=\duno$ or $\qdue=\ddue$.
	\item[Case C:] We start from the equilibrium defined by the intersection between the priority rule $\rettar$ in \eqref{eq:rettar22} and one of the straight lines $\rettas_{3}$ or $\rettas_{4}$ in \eqref{eq:z3z4}.
\end{description}

\subsection{Case A} 
This case is verified when the equilibrium is $( q_{1}, q_{2},q_{3},q_{4})=(\duno,\ddue,\at\duno+\bt\ddue,\aq\duno+\bq\ddue)$. Without loss of generality, we assume that the priority rule $\rettar$ first intersects the straight line $\qdue=\ddue$. We study the effects produced by a single wave sent on each road.
\subsubsection{Case A1: Wave on road 1}\label{sec:A1}
Let us start with a wave on road 1.
\begin{enumerate}[label=\roman*)]
\item We assume $\ondaq_{1}> q_{1}$. 
First of all we analyze the effects of a wave related only to the density $\rho$, i.e. we send a certain $\ondarho_{1}$ on road 1 keeping $w_{1}$ fixed. In Figure \subref*{fig:A1magg} we show a possible solution given by the algorithm. Specifically we have
\begin{equation*}
	 q_{1}\leq\qS_{1}\leq\ondaq_{1},\qquad  q_{2}=\qS_{2}, \qquad \qS_{3}=\at\qS_{1}+\bt q_{2}\geqq_{3}, \qquad \qS_{4}=\aq\qS_{1}+\bq q_{2}\geqq_{4}.
\end{equation*}
We refer to the Appendix of \cite{dellemonache2018CMS} for the estimates of $\Delta\Gamma$, $\Delta\hh$ and $\Delta\tv_{Q}$ of \ref{rsp2}. By \eqref{eq:w3T1} and \eqref{eq:w4T1} we have
	\begin{flalign*}
	|\wS_{3}-w_{3}|&\leq\frac{\at\bt q_{2}|w_{2}-w_{1}|}{(q_{3})^{2}}|\ondaq_{1}- q_{1}|,\qquad
	|\wS_{4}-w_{4}|\leq\frac{\aq\bq q_{2}|w_{2}-w_{1}|}{(q_{4})^{2}}|\ondaq_{1}- q_{1}|&&\\
	\Rightarrow\Delta\tv_{w}(\tb)&
	\leq\left(\frac{\at\bt}{(q_{3})^{2}}+\frac{\aq\bq}{(q_{4})^{2}}\right) q_{2}|w_{2}-w_{1}||\ondaq_{1}- q_{1}|.&&
	\end{flalign*}

Next, we analyze the effects of a wave in $\rho$ and $w$, i.e. we send a couple $(\ondarho_{1},\ondaw_{1})$ on road 1 such that we still have $\ondaq_{1}> q_{1}$. The estimates on $\Gamma$, $\hh$ and $\tv_{Q}$ do not change, while for $\tv_{w}$ by \eqref{eq:wt3T1} and \eqref{eq:wt4T1} we have
\begin{flalign*}	
	|\wS_{3}-w_{3}|&\leq \frac{\at\ondaq_{1}}{q_{3}}|\ondaw_{1}-w_{1}|+\frac{\at\bt q_{2}|w_{2}-w_{1}|}{(q_{3})^{2}}|\ondaq_{1}- q_{1}|&&\\
	|\wS_{4}-w_{4}|&\leq \frac{\at\ondaq_{1}}{q_{4}}|\ondaw_{1}-w_{1}|+\frac{\aq\bq q_{2}|w_{2}-w_{1}|}{(q_{4})^{2}}|\ondaq_{1}- q_{1}|&&\\
	\tv_{w}(\tb+)&\leq \left(\frac{\at}{q_{3}}+\frac{\aq}{q_{4}}\right)\ondaq_{1}|\ondaw_{1}-w_{1}|+\left(\frac{\at\bt}{(q_{3})^{2}}+\frac{\aq\bq}{(q_{4})^{2}}\right) q_{2}|w_{2}-w_{1}||\ondaq_{1}- q_{1}|&&\\
	\tv_{w}(\tb-)&=|\ondaw_{1}-w_{1}|&&\\
	\Rightarrow \Delta\tv_{w}(\tb)&\leq\left[\left(\frac{\at}{q_{3}}+\frac{\aq}{q_{4}}\right)\ondaq_{1}-1\right]|\ondaw_{1}-w_{1}|&&\\
	&+\left(\frac{\at\bt}{(q_{3})^{2}}+\frac{\aq\bq}{(q_{4})^{2}}\right) q_{2}|w_{2}-w_{1}||\ondaq_{1}- q_{1}|.&&
\end{flalign*}
Therefore \ref{rsp2} and \ref{rsp4} hold.

\item We assume $\ondaq_{1}< q_{1}$. 
First of all we analyze the effects of a wave related only to the density $\rho$, i.e. we send a certain $\ondarho_{1}$ on road 1 keeping $w_{1}$ fixed. In Figure \subref*{fig:A1min} we show a possible solution given by the algorithm. Specifically we have
\begin{equation*}
	 q_{1}>\qS_{1}=\ondaq_{1},\qquad  q_{2}=\qS_{2}, \qquad \qS_{3}=\at\ondaq_{1}+\bt q_{2}\leqq_{3}, \qquad \qS_{4}=\aq\ondaq_{1}+\bq q_{2}\leqq_{4}.
\end{equation*}
We refer to the Appendix of \cite{dellemonache2018CMS} for the estimates of $\Delta\Gamma$, $\Delta\hh$ and $\Delta\tv_{Q}$ of \ref{rsp2} and \ref{rsp3}. Note that $\qS_{3}\geq\bt q_{2}$ and $\qS_{4}\geq\bq q_{2}$. By \eqref{eq:w3T1} and \eqref{eq:w4T1} we have
	\begin{flalign*}
	|\wS_{3}-w_{3}|&\leq\frac{\at|w_{2}-w_{1}|}{q_{3}}|\ondaq_{1}- q_{1}|,\qquad
	|\wS_{4}-w_{4}|\leq\frac{\aq|w_{2}-w_{1}|}{q_{4}}|\ondaq_{1}- q_{1}|&&\\
	\Rightarrow\Delta\tv_{w}(\tb)&
	\leq\left(\frac{\at}{q_{3}}+\frac{\aq}{q_{4}}\right)|w_{2}-w_{1}||\ondaq_{1}- q_{1}|.&&
	\end{flalign*}

Next, we analyze the effects of a wave in $\rho$ and $w$, i.e. we send a couple $(\ondarho_{1},\ondaw_{1})$ on road 1 such that we still have $\ondaq_{1}< q_{1}$. The estimates on $\Gamma$, $\hh$ and $\tv_{Q}$ do not change, while for $\tv_{w}$ by \eqref{eq:wt3T1} and \eqref{eq:wt4T1} we have
\begin{flalign*}	
	|\wS_{3}-w_{3}|&\leq \frac{\at\ondaq_{1}}{q_{3}}|\ondaw_{1}-w_{1}|+\frac{\at|w_{2}-w_{1}|}{q_{3}}|\ondaq_{1}- q_{1}|&&\\
	|\wS_{4}-w_{4}|&\leq \frac{\at\ondaq_{1}}{q_{4}}|\ondaw_{1}-w_{1}|+\frac{\aq|w_{2}-w_{1}|}{q_{4}}|\ondaq_{1}- q_{1}|&&\\
	\tv_{w}(\tb+)&\leq \left(\frac{\at}{q_{3}}+\frac{\aq}{q_{4}}\right)\ondaq_{1}|\ondaw_{1}-w_{1}|+\left(\frac{\at}{q_{3}}+\frac{\aq}{q_{4}}\right)|w_{2}-w_{1}||\ondaq_{1}- q_{1}|&&\\
	\tv_{w}(\tb-)&=|\ondaw_{1}-w_{1}|&&\\
	\Rightarrow \Delta\tv_{w}(\tb)&\leq\left[\left(\frac{\at}{\bt}+\frac{\aq}{\bq}\right)\frac{\ondaq_{1}}{ q_{2}}-1\right]|\ondaw_{1}-w_{1}|+\left(\frac{\at}{q_{3}}+\frac{\aq}{q_{4}}\right)|w_{2}-w_{1}||\ondaq_{1}- q_{1}|.&&
\end{flalign*}
Therefore \ref{rsp2}, \ref{rsp3} and \ref{rsp4} hold.

\end{enumerate}

\begin{figure}[h!]
\centering
\includegraphics[]{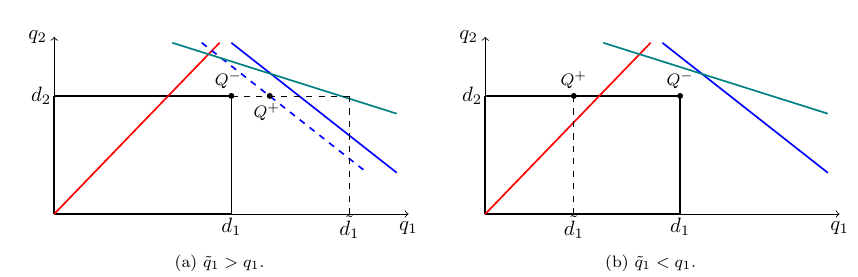}
\caption{Case A1: Wave on road 1.}
\label{fig:A1}
\end{figure}

\subsubsection{Case A2: Wave on road 2}\label{sec:A2}
We now consider a wave on road 2.
\begin{enumerate}[label=\roman*)]
\item We assume $\ondaq_{2}> q_{2}$. 
First of all we analyze the effects of a wave related only to the density $\rho$, i.e. we send a certain $\ondarho_{2}$ on road 2 keeping $w_{2}$ fixed. We have two possibilities: $\ondaddue$ is big enough to find an intersection between a straight line which maximizes the outgoing flux  $\rettas_{3}$ or $\rettas_{4}$ and the priority rule $\rettar$ (see Figure \subref*{fig:A2magg1}), or the solution is along $\qdue=\ondaddue$ (see Figure \subref*{fig:A2magg2}). 

In the first case we have
\begin{equation*}
	 q_{1}>\qS_{1},\qquad  q_{2}\leq\qS_{2}\leq\ondaq_{2}, \qquad \qS_{3}=\at\qS_{1}+\bt\qS_{2}, \qquad \qS_{4}=\aq\qS_{1}+\bq\qS_{2}.
\end{equation*}
Note that, since $\qS_{1}\geq\puno q_{2}/\pdue$, then $\qS_{3}\geq\at\puno q_{2}/\pdue$ and $\qS_{4}\geq\aq\puno q_{2}/\pdue$. 
We refer to the Appendix of \cite{dellemonache2018CMS} for the estimates of $\Delta\Gamma$, $\Delta\hh$ and $\Delta\tv_{Q}$ of \ref{rsp2}. We only observe that
\begin{align*}
	\Delta\hh(\tb) &=\frac{\qS_{2}- q_{2}}{\pdue}\leq\frac{|\ondaq_{2}- q_{2}|}{\pdue}.
\end{align*}
By \eqref{eq:w3T3} and \eqref{eq:w4T3} we have
	\begin{flalign*}
	|\wS_{3}-w_{3}|&\leq\frac{\bt\pdue|w_{1}-w_{2}|(\ondaq_{2}+ q_{1})}{\puno q_{2}q_{3}}(|\Delta\Gamma(\tb)|+|\Delta\hh(\tb)|)&&\\
	|\wS_{4}-w_{4}|&\leq\frac{\bq\pdue|w_{1}-w_{2}|(\ondaq_{2}+ q_{1})}{\puno q_{2}q_{4}}(|\Delta\Gamma(\tb)|+|\Delta\hh(\tb)|)&&\\
	\Rightarrow\Delta\tv_{w}(\tb)&\leq\left(\frac{\bt}{q_{3}}+\frac{\bq}{q_{4}}\right)\frac{\pdue|w_{1}-w_{2}|(\ondaq_{2}+ q_{1})}{\puno q_{2}}(|\Delta\Gamma(\tb)|+|\Delta\hh(\tb)|)
	.&&
\end{flalign*}

Next, we analyze the effects of a wave in $\rho$ and $w$, i.e. we send a couple $(\ondarho_{2},\ondaw_{2})$ on road 2 such that we still have $\ondaq_{2}> q_{2}$. The estimates on $\Gamma$, $\hh$ and $\tv_{Q}$ do not change, while for $\tv_{w}$ by \eqref{eq:wt3T3} and \eqref{eq:wt4T3} we have 
\begin{flalign*}
	|\wS_{3}-w_{3}|&\leq\frac{\bt\pdue\ondaq_{2}|\ondaw_{2}-w_{2}|}{\at\puno q_{1}}
	+\frac{\bt\pdue|w_{1}-w_{2}|(\ondaq_{2}+ q_{1})}{\puno q_{2}q_{3}}(|\Delta\Gamma(\tb)|+|\Delta\hh(\tb)|)&&\\
	|\wS_{4}-w_{3}|&\leq\frac{\bq\pdue\ondaq_{2}|\ondaw_{2}-w_{2}|}{\aq\puno q_{1}}
	+\frac{\bq\pdue|w_{1}-w_{2}|(\ondaq_{2}+ q_{1})}{\puno q_{2}q_{4}}(|\Delta\Gamma(\tb)|+|\Delta\hh(\tb)|)&&\\
	\Rightarrow \Delta\tv_{w}(\tb)&\leq\left(\frac{\bt}{q_{3}}+\frac{\bq}{q_{4}}\right)\frac{\pdue\ondaq_{2}|\ondaw_{2}-w_{2}|}{\puno q_{2}}|\ondaw_{2}-w_{2}|&&\\
	&+\left(\frac{\bt}{q_{3}}+\frac{\bq}{q_{4}}\right)\frac{\pdue|w_{1}-w_{2}|(\ondaq_{2}+ q_{1})}{\puno q_{2}}(|\Delta\Gamma(\tb)|+|\Delta\hh(\tb)|).&&
\end{flalign*}

In the second case we have 
\begin{equation*}
	 q_{1}>\qS_{1},\qquad  q_{2}<\ondaq_{2}=\qS_{2}, \qquad \qS_{3}=\at\qS_{1}+\bt\qS_{2}, \qquad \qS_{4}=\aq\qS_{1}+\bq\qS_{2}.
\end{equation*}
Note that $\qS_{3}\geq\bt\ondaq_{2}$ and $\qS_{4}\geq\bq\ondaq_{2}$. 
We are interested in property \ref{rsp2} thus we compute
\begin{itemize}[label={*},leftmargin=*]
	\Item
	\begin{flalign*}
	\Gamma(\tb+) & = \qS_{1}+\qS_{2}=\qS_{1}+\ondaq_{2}, \qquad \Gamma(\tb-) =  q_{1}+ q_{2}&&\\
	\Rightarrow \Delta\Gamma(\tb) & = (\qS_{1}- q_{1})+(\ondaq_{2}- q_{2}).&&
	\end{flalign*}
	\Item
	\begin{flalign*}
	\hh(\tb+)&=\frac{\ondaq_{2}}{\pdue},\qquad \hh(\tb-)=\frac{ q_{2}}{\pdue}&&\\
	\Rightarrow \Delta \hh(\tb)&=\frac{\ondaq_{2}- q_{2}}{\pdue}\leq\frac{|\ondaq_{2}- q_{2}|}{\pdue}.&&
	\end{flalign*}
	\Item
	\begin{flalign*}
	\tv_{Q}(\tb+) &= |\qS_{1}- q_{1}|+|\at\qS_{1}+\bt\ondaq_{2}-\at q_{1}-\bt q_{2}|&&\\
	&+|\aq\qS_{1}+\bq\ondaq_{2}-\aq q_{1}-\bq q_{2}| 
	\leq 2|\qS_{1}- q_{1}|+|\ondaq_{2}- q_{2}|&&\\
	\tv_{Q}(\tb-) &= |\ondaq_{2}- q_{2}|&&\\
	\Rightarrow\Delta\tv_{Q}(\tb) &\leq 2|\qS_{1}- q_{1}| \leq 2(|\Delta\Gamma(\tb)|+|\Delta\hh(\tb)|).&&
	\end{flalign*}
	\item By \eqref{eq:w3T3} and \eqref{eq:w4T3} we have
	\begin{flalign*}
	|\wS_{3}-w_{3}|&\leq\frac{\at|w_{1}-w_{2}|(\ondaq_{2}+ q_{1})}{\ondaq_{2}q_{3}}(|\Delta\Gamma(\tb)|+|\Delta\hh(\tb)|)&&\\
	|\wS_{4}-w_{4}|&\leq\frac{\aq|w_{1}-w_{2}|(\ondaq_{2}+ q_{1})}{\ondaq_{2}q_{4}}(|\Delta\Gamma(\tb)|+|\Delta\hh(\tb)|)&&\\
	\Rightarrow\Delta\tv_{w}(\tb)&\leq\left(\frac{\at}{q_{3}}+\frac{\aq}{q_{4}}\right)\frac{|w_{1}-w_{2}|(\ondaq_{2}+ q_{1})}{\ondaq_{2}}(|\Delta\Gamma(\tb)|+|\Delta\hh(\tb)|)
	.&&
	\end{flalign*}
\end{itemize}

Next, we analyze the effects of a wave in $\rho$ and $w$, i.e. we send a couple $(\ondarho_{2},\ondaw_{2})$ on road 2 such that we still have $\ondaq_{2}> q_{2}$. The estimates on $\Gamma$, $\hh$ and $\tv_{Q}$ do not change, while for $\tv_{w}$ by \eqref{eq:wt3T3} and \eqref{eq:wt4T3} we have 
\begin{flalign*}
	|\wS_{3}-w_{3}|&\leq|\ondaw_{2}-w_{2}|+\frac{\at|w_{1}-w_{2}|(\ondaq_{2}+ q_{1})}{\ondaq_{2}q_{3}}(|\Delta\Gamma(\tb)|+|\Delta\hh(\tb)|)&&\\
	|\wS_{4}-w_{3}|&\leq|\ondaw_{2}-w_{2}|+\frac{\aq|w_{1}-w_{2}|(\ondaq_{2}+ q_{1})}{\ondaq_{2}q_{4}}(|\Delta\Gamma(\tb)|+|\Delta\hh(\tb)|)&&\\
	\tv_{w}(\tb-)&=|\ondaw_{2}-w_{2}|&&\\
	\Rightarrow \Delta\tv_{w}(\tb)&\leq|\ondaw_{2}-w_{2}|+\left(\frac{\at}{q_{3}}+\frac{\aq}{q_{4}}\right)\frac{|w_{1}-w_{2}|(\ondaq_{2}+ q_{1})}{\ondaq_{2}}(|\Delta\Gamma(\tb)|+|\Delta\hh(\tb)|).&&
\end{flalign*}
Therefore \ref{rsp2} and \ref{rsp4} hold.

\begin{figure}[h!]
\centering
\includegraphics[]{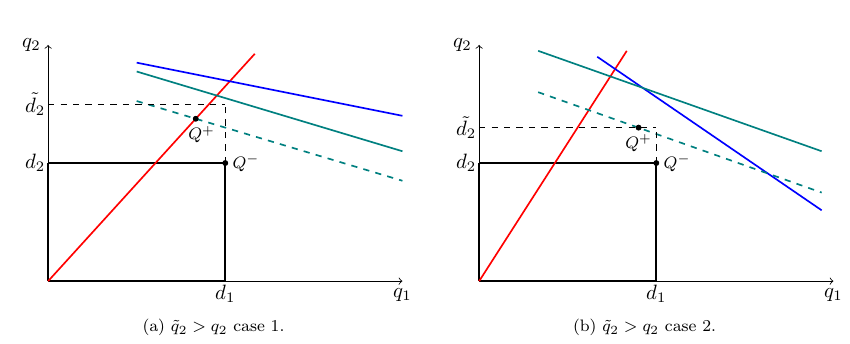}
\caption{Case A2: Wave on road 2 with $\ondaq_{2}> q_{2}$.}
\label{fig:A2}
\end{figure}

\item We assume $\ondaq_{2}< q_{2}$. 
First of all we analyze the effects of a wave related only to the density $\rho$, i.e. we send a certain $\ondarho_{2}$ on road 2 keeping $w_{2}$ fixed. In Figure \ref{fig:A2min} we show a possible solution given by the algorithm. Specifically we have
\begin{equation*}
	 q_{1}=\qS_{1},\qquad  q_{2}>\qS_{2}=\ondaq_{2}, \qquad \qS_{3}=\at q_{1}+\bt\ondaq_{2}\leqq_{3}, \qquad \qS_{4}=\aq q_{1}+\bq\ondaq_{2}\leqq_{4}.
\end{equation*}
We refer to the Appendix of \cite{dellemonache2018CMS} for the estimates of $\Delta\Gamma$, $\Delta\hh$ and $\Delta\tv_{Q}$ of \ref{rsp2} and \ref{rsp3}. Note that $\qS_{3}\geq\at q_{1}$ and $\qS_{4}\geq\aq q_{1}$. By \eqref{eq:w3T1} and \eqref{eq:w4T1} we have
	\begin{flalign*}
	|\wS_{3}-w_{3}|&\leq\frac{\bt|w_{1}-w_{2}|}{q_{3}}|\ondaq_{2}- q_{2}|,\qquad
	|\wS_{4}-w_{4}|\leq\frac{\bq|w_{1}-w_{2}|}{q_{4}}|\ondaq_{2}- q_{2}|&&\\
	\Rightarrow\Delta\tv_{w}(\tb)&=\left(\frac{\bt}{q_{3}}+\frac{\bq}{q_{4}}\right)|w_{1}-w_{2}||\ondaq_{2}- q_{1}|.
	&&
	\end{flalign*}

Next, we analyze the effects of a wave in $\rho$ and $w$, i.e. we send a couple $(\ondarho_{2},\ondaw_{2})$ on road 2 such that we still have $\ondaq_{2}< q_{2}$. The estimates on $\Gamma$, $\hh$ and $\tv_{Q}$ do not change, while for $\tv_{w}$ by \eqref{eq:wt3T1} and \eqref{eq:wt4T1} we have
\begin{flalign*}
	|\wS_{3}-w_{3}|&\leq \frac{\bt\ondaq_{2}}{\at q_{1}}|\ondaw_{2}-w_{2}|+\frac{\bt|w_{2}-w_{1}|}{q_{3}}|\ondaq_{2}- q_{2}|&&\\
	|\wS_{4}-w_{4}|&\leq \frac{\bq\ondaq_{2}}{\at q_{1}}|\ondaw_{2}-w_{2}|+\frac{\bt|w_{2}-w_{1}|}{q_{4}}|\ondaq_{2}- q_{2}|&&\\
	\tv_{w}(\tb+)&\leq \left(\frac{\bt}{\at}+\frac{\bq}{\aq}\right)\frac{\ondaq_{2}}{ q_{1}}|\ondaw_{2}-w_{2}|+\left(\frac{\bt}{q_{3}}+\frac{\bq}{q_{4}}\right)|w_{2}-w_{1}||\ondaq_{2}- q_{2}|
	&&\\
	\tv_{w}(\tb-)&=|\ondaw_{2}-w_{2}|&&\\
	\Rightarrow \Delta\tv_{w}(\tb)&\leq\left[\left(\frac{\bt}{\at}+\frac{\bq}{\aq}\right)\frac{\ondaq_{2}}{ q_{1}}-1\right]|\ondaw_{2}-w_{2}|+\left(\frac{\bt}{q_{3}}+\frac{\bq}{q_{4}}\right)|w_{2}-w_{1}||\ondaq_{2}- q_{2}|.&&
\end{flalign*}
Therefore \ref{rsp2}, \ref{rsp3} and \ref{rsp4} hold.

\end{enumerate}

\begin{figure}[h!]
\centering
\includegraphics[]{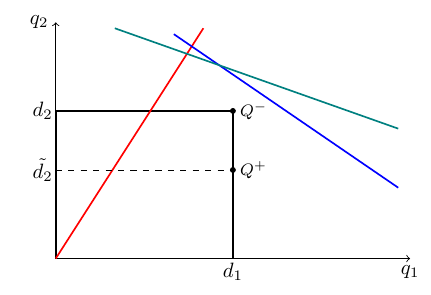}
\caption{Case A2: Wave on road 2 with $\ondaq_{2}< q_{2}$.}
\label{fig:A2min}
\end{figure}

\subsubsection{Case A3: Wave on road 3}\label{sec:A3}
We now consider a wave on road 3. The case of a wave on road 4 is analogous. Note that in this case we are not interested in what happens sending a wave in $\rho$ and $w$, since the changes in $w$ on the outgoing roads do not affect the Riemann Solver. Hence, we only study waves which involve the density $\rho$ and keep $w_{3}$ fixed.

\begin{enumerate}[label=\roman*)]
\item We assume $\ondaq_{3}>q_{3}$, which implies that the straight line which maximizes the outgoing flux $\rettas_{3}$ move to the top. Therefore, in this case the solution of the algorithm is again the equilibrium $U$, thus nothing happens.

\item We assume $\ondaq_{3}< q_{3}$. 
We send a certain $\ondarho_{3}$ on road 3 keeping $w_{3}$ fixed. In this case the straight line which maximizes the outgoing flux $\rettas_{3}$ moves to the bottom, thus we have two possibilities: the new maximisation straight line $\rettas_{3}$ intersects the priority rule $\rettar$ (see Figure \subref*{fig:A3min1}) or the solution is such that $\qS_{2}= q_{2}$ and $ q_{2}\puno/\pdue\leq\qS_{1}\leq q_{1}$ (see Figure \subref*{fig:A3min2}).

In the first case we have
\begin{equation*}
	 q_{1}\geq\qS_{1},\qquad  q_{2}\geq\qS_{2}, \qquad \qS_{3}=\ondaq_{3}=\at\qS_{1}+\bt\qS_{2}\leq q_{3}, \qquad \qS_{4}=\aq\qS_{1}+\bq\qS_{2}\leq q_{4}.
\end{equation*}
We refer to the Appendix of \cite{dellemonache2018CMS} for the estimates of $\Delta\Gamma$, $\Delta\hh$ and $\Delta\tv_{Q}$ of \ref{rsp2} and \ref{rsp3}. Note that
\begin{align*}
	\qS_{1}&=\frac{\puno\ondaq_{3}}{\at\puno+\bt\pdue},\qquad
	\qS_{2} = \frac{\pdue\ondaq_{3}}{\at\puno+\bt\pdue},\qquad
	\qS_{4} = \frac{\aq\puno+\bq\pdue}{\at\puno+\bt\pdue}
	\ondaq_{3}\\
	 q_{1}&=\frac{\puno\qstar_{3}}{\at\puno+\bt\pdue},\qquad
	 q_{2} = \frac{\pdue\qstar_{3}}{\at\puno+\bt\pdue},
\end{align*}
with $\qstar_{3}\leq q_{3}$ such that $\rettas_{3}$ goes through the point $Q^{*}$.
By \eqref{eq:w3T2} and \eqref{eq:w4T2} we have
\begin{flalign*}
	|\wS_{3}-w_{3}|&\leq\frac{\at\bt|w_{1}-w_{2}|}{\ondaq_{3} q_{3}}|\qS_{2}(\qS_{1}- q_{1})-\qS_{1}(\qS_{2}- q_{2})|&&\\
	&\leq\frac{\at\bt|w_{1}-w_{2}|}{(\at\puno+\bt\pdue) q_{3}}|\ondaq_{3}- q_{3}|	&&\\
	|\wS_{4}-w_{4}|&\leq\frac{\aq\bq|w_{1}-w_{2}|}{(\aq\puno+\bq\pdue) q_{4}}|\ondaq_{3}- q_{3}|	&&\\
	\Rightarrow\Delta\tv_{w}(\tb)&\leq\left(\frac{\at\bt}{(\at\puno+\bt\pdue) q_{3}}+\frac{\aq\bq}{(\aq\puno+\bq\pdue) q_{4}}\right)|w_{1}-w_{2}||\ondaq_{3}- q_{3}|	.&&
\end{flalign*}

In the second case we have
\begin{equation*}
	 q_{1}\geq\qS_{1},\qquad  q_{2}=\qS_{2}, \qquad \qS_{3}=\at\qS_{1}+\bt q_{2}\leq  q_{3}, \qquad \qS_{4}=\aq\qS_{1}+\bq q_{2}\leq q_{4}.
\end{equation*}
Note that $\qS_{3}\geq\bt q_{2}$, $\qS_{4}\geq\bq q_{2}$ and that
\begin{equation*}
	\qS_{1}=\frac{\qS_{3}-\bt q_{2}}{\at},\qquad
	 q_{1}=\frac{ q_{3}-\bt q_{2}}{\at}.
\end{equation*}
We compute
\begin{itemize}[label={*},leftmargin=*]
	\Item
	\begin{flalign*}
	\Gamma(\tb+) & = \qS_{1}+\qS_{2}=\qS_{1}+ q_{2}, \qquad \Gamma(\tb-) =  q_{1}+ q_{2}&&\\
	\Rightarrow \Delta\Gamma(\tb) & = \qS_{1}- q_{1}=\frac{\qS_{3}- q_{3}}{\at}<0.&&
	\end{flalign*}
	\Item
	\begin{flalign*}
	\hh(\tb+)&=\hh(\tb-)=\frac{ q_{2}}{\pdue}&&\\
	\Rightarrow \Delta \hh(\tb)&=0.&&
	\end{flalign*}
	\Item
	\begin{flalign*}
	\tv_{Q}(\tb+) &= |\qS_{1}- q_{1}|+|\qS_{3}-\ondaq_{3}|+\aq|\qS_{1}- q_{1}| =\frac{1+\aq}{\at}|\qS_{3}- q_{3}|+|\qS_{3}-\ondaq_{3}|&&\\
	\tv_{Q}(\tb-) &= |\ondaq_{3}- q_{3}|&&\\
	\Rightarrow\Delta\tv_{Q}(\tb) &\leq \frac{1+\at+\aq}{\at}|\qS_{3}- q_{3}|=2|\Delta\Gamma(\tb)|.
	\end{flalign*}
	\item By \eqref{eq:w3T1} and \eqref{eq:w4T1} we have
	\begin{flalign*}
	|\wS_{3}-w_{3}|&\leq\frac{\at\bt q_{2}|w_{1}-w_{2}|}{\qS_{3} q_{3}}|\qS_{1}- q_{1}|	
	\leq \frac{\at|w_{1}-w_{2}|}{ q_{3}}|\Delta\Gamma(\tb)|&&\\
	|\wS_{4}-w_{4}|&\leq\frac{\aq|w_{1}-w_{2}|}{ q_{4}}|\Delta\Gamma(\tb)|	&&\\
	\Rightarrow\Delta\tv_{w}(\tb)&\leq\left(\frac{\at}{ q_{3}}+\frac{\aq}{ q_{4}}\right)|w_{1}-w_{2}||\Delta\Gamma(\tb)|.&&
	\end{flalign*}
\end{itemize}
Therefore \ref{rsp2} and \ref{rsp3} hold.
\end{enumerate}

\begin{figure}[h!]
\centering
\includegraphics[]{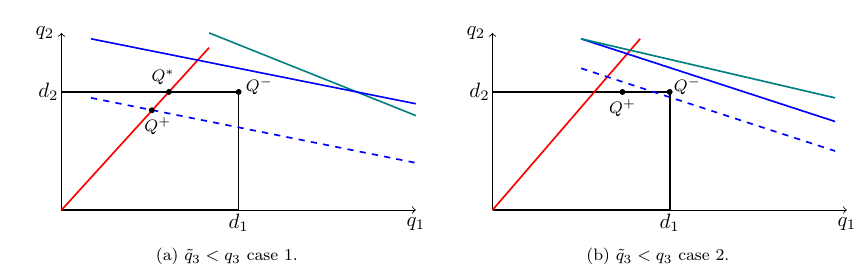}
\caption{Case A3: Wave on road 3 with $\ondaq_{3}< q_{3}$.}
\label{fig:A3}
\end{figure}

\subsection{Case B} 
This case is verified when the equilibrium is along one of the straight lines $\quno=\duno$ or $\qdue=\ddue$. Without loss of generality, we assume that the priority rule $\rettar$ first intersects the straight line $\quno=\duno$, thus the equilibrium is along the right side of the rectangle $\omegaInc$. We study the effects produced by a single wave sent on each road.
\subsubsection{Case B1: Wave on road 1}\label{sec:B1}
Let us start with a wave on road 1.
\begin{enumerate}[label=\roman*)]
\item We assume $\ondaq_{1}> q_{1}$. 
First of all we analyze the effects of a wave related only to the density $\rho$, i.e. we send a certain $\ondarho_{1}$ on road 1 keeping $w_{1}$ fixed. We have two possibilities: $\ondaduno$ is big enough to find an intersection between a straight line which maximizes the outgoing flux $\rettas_{3}$ or $\rettas_{4}$ and the priority rule $\rettar$ (see Figure \subref*{fig:B1magg1}), or the solution is along $\quno=\ondaduno$ (see Figure \subref*{fig:B1magg2}). 

In the first case we have
\begin{equation*}
	 q_{1}\leq\qS_{1}\leq\ondaq_{1},\qquad  q_{2}\geq\qS_{2}, \qquad \qS_{3}=\at\qS_{1}+\bt\qS_{2}, \qquad \qS_{4}=\aq\qS_{1}+\bq\qS_{2}= q_{4}.
\end{equation*}
We refer to the Appendix of \cite{dellemonache2018CMS} for the estimates of $\Delta\Gamma$, $\Delta\hh$ and $\Delta\tv_{Q}$ of \ref{rsp2}. We only observe that
\begin{align*}
	\Delta\hh(\tb) &=\frac{\qS_{1}- q_{1}}{\puno}\leq\frac{|\ondaq_{1}- q_{1}|}{\puno}.
\end{align*}
Moreover, since $\qS_{2}\geq\pdue q_{1}/\puno$, then $\qS_{3}\geq\bt\pdue q_{1}/\puno$ and $\qS_{4}\geq\bq\pdue q_{1}/\puno$. By \eqref{eq:w3T3} and \eqref{eq:w4T3} we have
\begin{flalign*}
	|\wS_{3}-w_{3}|&\leq\frac{\at\puno|w_{1}-w_{2}|(\ondaq_{1}+ q_{2})}{\pdue q_{1} q_{3}}(|\Delta\Gamma(\tb)|+|\Delta\hh(\tb)|)&&\\
	|\wS_{4}-w_{4}|&\leq\frac{\aq\puno|w_{1}-w_{2}|(\ondaq_{1}+ q_{2})}{\pdue q_{1} q_{4}}(|\Delta\Gamma(\tb)|+|\Delta\hh(\tb)|)&&\\
	\Rightarrow\Delta\tv_{w}(\tb)&\leq\left(\frac{\at}{ q_{3}}+\frac{\aq}{ q_{4}}\right)\frac{\puno|w_{1}-w_{2}|(\ondaq_{1}+ q_{2})}{\pdue q_{1}}(|\Delta\Gamma(\tb)|+|\Delta\hh(\tb)|)
	.&&
\end{flalign*}

Next, we analyze the effects of a wave in $\rho$ and $w$, i.e. we send a couple $(\ondarho_{1},\ondaw_{1})$ on road 1 such that we still have $\ondaq_{1}> q_{1}$. The estimates on $\Gamma$, $\hh$ and $\tv_{Q}$ do not change, while for $\tv_{w}$ by \eqref{eq:wt3T3} and \eqref{eq:wt4T3} we have 
\begin{flalign*}
	|\wS_{3}-w_{3}|&\leq\frac{\at\puno\ondaq_{1}|\ondaw_{1}-w_{1}|}{\bt\pdue q_{1}}
	+\frac{\at\puno|w_{1}-w_{2}|(\ondaq_{1}+ q_{2})}{\pdue q_{1} q_{3}}(|\Delta\Gamma(\tb)|+|\Delta\hh(\tb)|)&&\\
	|\wS_{4}-w_{3}|&\leq\frac{\aq\puno\ondaq_{1}|\ondaw_{1}-w_{1}|}{\bq\pdue q_{1}}
	+\frac{\aq\puno|w_{1}-w_{2}|(\ondaq_{1}+ q_{2})}{\pdue q_{1} q_{4}}(|\Delta\Gamma(\tb)|+|\Delta\hh(\tb)|)&&\\
	\Rightarrow \Delta\tv_{w}(\tb)&\leq\left(\frac{\at}{ q_{3}}+\frac{\aq}{ q_{4}}\right)\frac{\puno\ondaq_{1}}{\puno q_{1}}|\ondaw_{1}-w_{1}|&&\\
	&+\left(\frac{\at}{ q_{3}}+\frac{\aq}{ q_{4}}\right)\frac{\puno|w_{1}-w_{2}|(\ondaq_{1}+ q_{2})}{\pdue q_{1}}(|\Delta\Gamma(\tb)|+|\Delta\hh(\tb)|).&&
\end{flalign*}

In the second case we have
\begin{equation*}
	 q_{1}\leq\qS_{1}=\ondaq_{1},\qquad \qS_{3}=\at\qS_{1}+\bt\qS_{2}, \qquad \qS_{4}=\aq\qS_{1}+\bq\qS_{2}= q_{4},
\end{equation*}
with $\qS_{2}$ that can be both greater or lower than $ q_{2}$.
Note that, since $\qS_{2}\geq\pdue\ondaq_{1}/\puno$, then $\qS_{3}\geq\bt\pdue\ondaq_{1}/\puno$ and $\qS_{4}\geq\bq\pdue\ondaq_{1}/\puno$. We compute
\begin{itemize}[label={*},leftmargin=*]
	\Item
	\begin{flalign*}
	\Gamma(\tb+) & = \qS_{1}+\qS_{2}=\ondaq_{1}+\qS_{2}, \qquad \Gamma(\tb-) =  q_{1}+ q_{2}&&\\
	\Rightarrow \Delta\Gamma(\tb) & = (\ondaq_{1}- q_{1})+(\qS_{2}- q_{2}). &&
	\end{flalign*}
	\Item
	\begin{flalign*}
	\hh(\tb+)&=\frac{\ondaq_{1}}{\puno},\qquad \hh(\tb-)=\frac{ q_{1}}{\puno}&&\\
	\Rightarrow \Delta \hh(\tb)&=\frac{\ondaq_{1}- q_{1}}{\puno}\leq\frac{|\ondaq_{1}- q_{1}|}{\puno}.&&
	\end{flalign*}
	\Item
	\begin{flalign*}
	\tv_{Q}(\tb+) &= |\qS_{2}- q_{2}|+|\at\ondaq_{1}+\bt\qS_{2}-\at q_{1}-\bt q_{2}|&&\\
	&+|\aq\ondaq_{1}+\bq\qS_{2}-\aq q_{1}-\bq q_{2}| 
	\leq |\ondaq_{1}- q_{1}|+2|\qS_{2}- q_{2}|&&\\
	\tv_{Q}(\tb-) &= |\ondaq_{1}- q_{1}|&&\\
	\Rightarrow\Delta\tv_{Q}(\tb) &\leq 2|\qS_{2}- q_{2}| \leq 2(|\Delta\Gamma(\tb)|+|\Delta\hh(\tb)|).&&
	\end{flalign*}
	\item By \eqref{eq:w3T3} and \eqref{eq:w4T3} we have
	\begin{flalign*}
	|\wS_{3}-w_{3}|&\leq\frac{\at\puno|w_{1}-w_{2}|(\ondaq_{1}+\ddue)}{\pdue\ondaq_{1} q_{3}}(|\Delta\Gamma(\tb)|+|\Delta\hh(\tb)|)&&\\
	|\wS_{4}-w_{4}|&\leq\frac{\aq\puno|w_{1}-w_{2}|(\ondaq_{1}+\ddue)}{\pdue\ondaq_{1} q_{4}}(|\Delta\Gamma(\tb)|+|\Delta\hh(\tb)|)&&\\
	\Rightarrow\Delta\tv_{w}(\tb)&\leq\left(\frac{\at}{ q_{3}}+\frac{\aq}{ q_{4}}\right)\frac{\puno|w_{1}-w_{2}|(\ondaq_{1}+\ddue)}{\pdue\ondaq_{1}}(|\Delta\Gamma(\tb)|+|\Delta\hh(\tb)|)
	.&&
	\end{flalign*}
\end{itemize}

Next, we analyze the effects of a wave in $\rho$ and $w$, i.e. we send a couple $(\ondarho_{1},\ondaw_{1})$ on road 1 such that we still have $\ondaq_{1}> q_{1}$. The estimates on $\Gamma$, $\hh$ and $\tv_{Q}$ do not change, while for $\tv_{w}$ by \eqref{eq:wt3T3} and \eqref{eq:wt4T3} we have 
\begin{flalign*}
	|\wS_{3}-w_{3}|&\leq\frac{\at\puno|\ondaw_{1}-w_{1}|}{\bt\pdue}
	+\frac{\at\puno|w_{1}-w_{2}|(\ondaq_{1}+\ddue)}{\pdue\ondaq_{1} q_{3}}(|\Delta\Gamma(\tb)|+|\Delta\hh(\tb)|)&&\\
	|\wS_{4}-w_{3}|&\leq\frac{\aq\puno|\ondaw_{1}-w_{1}|}{\bq\pdue}
	+\frac{\aq\puno|w_{1}-w_{2}|(\ondaq_{1}+\ddue)}{\pdue\ondaq_{1} q_{4}}(|\Delta\Gamma(\tb)|+|\Delta\hh(\tb)|)&&\\
	\Rightarrow \Delta\tv_{w}(\tb)&\leq\left(\frac{\at}{\bt q_{3}}+\frac{\aq}{\bq q_{4}}\right)\frac{\puno}{\puno}|\ondaw_{1}-w_{1}|&&\\
	&+\left(\frac{\at}{ q_{3}}+\frac{\aq}{ q_{4}}\right)\frac{\puno|w_{1}-w_{2}|(\ondaq_{1}+\ddue)}{\pdue\ondaq_{1}}(|\Delta\Gamma(\tb)|+|\Delta\hh(\tb)|).&&
\end{flalign*}
Therefore \ref{rsp2} and \ref{rsp4} hold.

\begin{figure}[h!]
\centering
\includegraphics[]{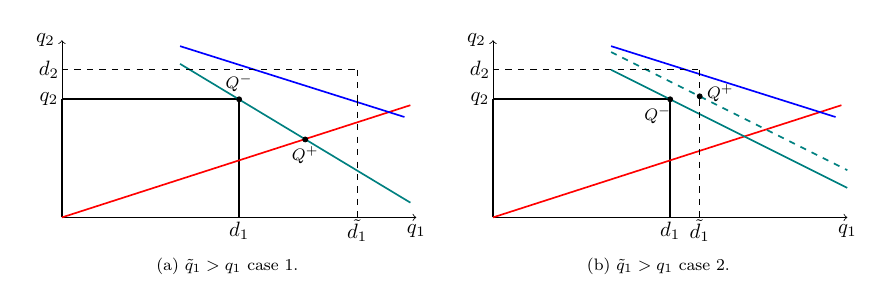}
\caption{Case B1: Wave on road 1 with $\ondaq_{1}> q_{1}$.}
\label{fig:B1}
\end{figure}

\item We assume $\ondaq_{1}< q_{1}$. 
First of all we analyze the effects of a wave related only to the density $\rho$, i.e. we send a certain $\ondarho_{1}$ on road 1 keeping $w_{1}$ fixed. In Figure \ref{fig:B1} we show a possible solution given by the algorithm. Specifically we have
\begin{equation*}
	 q_{1}>\qS_{1}=\ondaq_{1},\qquad \qS_{3}=\at\ondaq_{1}+\bt\qS_{2}, \qquad \qS_{4}=\aq\ondaq_{1}+\bq q_{2}.
\end{equation*}
Note that $\qS_{3}\geq\at\ondaq_{1}$ and $\qS_{4}\geq\aq\ondaq_{1}$. Moreover, we observe that if $\qS_{2}< q_{2}$ then $\qS_{2}- q_{2}<0$, otherwise we define the angle $\beta=\arctan(|\ddue- q_{2}|/|\ondaq_{1}- q_{1}|)$ between the segments $|Q^{-}Q^{*}|$ and $|\widehat QQ^{-}|$ and we obtain $\qS_{2}- q_{2}=|\qS_{2}- q_{2}|\leq \tan\beta|\ondaq_{1}- q_{1}|$.
We compute
\begin{itemize}[label={*},leftmargin=*]
	\Item
	\begin{flalign*}
	\Gamma(\tb+) & = \qS_{1}+\qS_{2}=\ondaq_{1}+\qS_{2}, \qquad \Gamma(\tb-) =  q_{1}+ q_{2}&&\\
	\Rightarrow \Delta\Gamma(\tb) & = (\ondaq_{1}- q_{1})+(\qS_{2}- q_{2})\leq C|\ondaq_{1}- q_{1}|.
	\end{flalign*}
	\Item
	\begin{flalign*}
	\hh(\tb+)&=\frac{\ondaq_{1}}{\puno}, \qquad \hh(\tb-) = \frac{ q_{1}}{\puno}&&\\
	\Rightarrow \Delta \hh(\tb)&=\frac{\ondaq_{1}- q_{1}}{\puno}<0.&&
	\end{flalign*}
	\Item
	\begin{flalign*}
	\tv_{Q}(\tb+) &\leq 2|\qS_{2}- q_{2}|+|\ondaq_{1}- q_{1}|\leq 2(|\Delta\Gamma(\tb)|+|\ondaq_{1}- q_{1}|)&&\\
	\tv_{Q}(\tb-) &= |\ondaq_{1}- q_{1}|&&\\
	\Rightarrow\Delta\tv_{Q}(\tb) &\leq 2(|\Delta\Gamma(\tb)|+|\Delta\hh(\tb)|).  &&
	\end{flalign*}
	\item By \eqref{eq:w3T3} and \eqref{eq:w4T3} we have
	\begin{flalign*}
	|\wS_{3}-w_{3}|&\leq\frac{\bt|w_{1}-w_{2}|}{\ondaq_{1} q_{3}}(\ondaq_{1}+\ddue)(|\Delta\Gamma(\tb)|+|\Delta\hh(\tb)|)&&\\
	|\wS_{4}-w_{4}|&\leq\frac{\bq|w_{1}-w_{2}|}{\ondaq_{1} q_{4}}(\ondaq_{1}+\ddue)(|\Delta\Gamma(\tb)|+|\Delta\hh(\tb)|)&&\\
	\Rightarrow\Delta\tv_{w}(\tb)&\leq\left(\frac{\bt}{ q_{3}}+\frac{\bq}{ q_{4}}\right)\frac{|w_{1}-w_{2}|}{\ondaq_{1}}(\ondaq_{1}+\ddue)(|\Delta\Gamma(\tb)|+|\Delta\hh(\tb)|).&&
	\end{flalign*}
\end{itemize}

Next, we analyze the effects of a wave in $\rho$ and $w$, i.e. we send a couple $(\ondarho_{1},\ondaw_{1})$ on road 1 such that we still have $\ondaq_{1}< q_{1}$. The estimates on $\Gamma$, $\hh$ and $\tv_{Q}$ do not change, while for $\tv_{w}$ by \eqref{eq:wt3T3} and \eqref{eq:wt4T3} we have
\begin{flalign*}
	|\wS_{3}-w_{3}|&\leq|\ondaw_{1}-w_{1}|+\frac{\bt|w_{1}-w_{2}|}{\ondaq_{1} q_{3}}(\ondaq_{1}+\ddue)(|\Delta\Gamma(\tb)|+|\Delta\hh(\tb)|)&&\\
	|\wS_{4}-w_{4}|&\leq|\ondaw_{1}-w_{1}|+\frac{\bq|w_{1}-w_{2}|}{\ondaq_{1} q_{4}}(\ondaq_{1}+\ddue)(|\Delta\Gamma(\tb)|+|\Delta\hh(\tb)|)&&\\
	\tv_{w}(\tb+)&\leq 2|\ondaw_{1}-w_{1}|+\left(\frac{\bt}{ q_{3}}+\frac{\bq}{ q_{4}}\right)\frac{|w_{1}-w_{2}|}{\ondaq_{1}}(\ondaq_{1}+\ddue)(|\Delta\Gamma(\tb)|+|\Delta\hh(\tb)|)
	&&\\
	\tv_{w}(\tb-)&=|\ondaw_{1}-w_{1}|&&\\
	\Rightarrow \Delta\tv_{w}(\tb)&\leq|\ondaw_{1}-w_{1}|+\left(\frac{\bt}{ q_{3}}+\frac{\bq}{ q_{4}}\right)\frac{|w_{1}-w_{2}|}{\ondaq_{1}}(\ondaq_{1}+\ddue)(|\Delta\Gamma(\tb)|+|\Delta\hh(\tb)|).&&
\end{flalign*}
Therefore \ref{rsp2}, \ref{rsp3} and \ref{rsp4} hold.

\end{enumerate}

\begin{figure}[h!]
\centering
\includegraphics[]{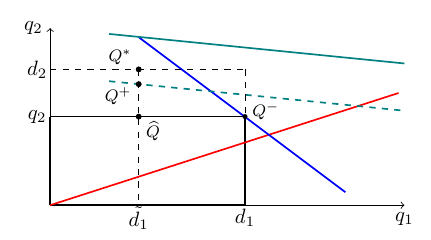}
\caption{Case B1: Wave on road 1 with $\ondaq_{1}< q_{1}$.}
\label{fig:B1-bis}
\end{figure}

\subsubsection{Case B2: Wave on road 2}\label{sec:B2}
Let us consider a wave on road 2.
\begin{enumerate}[label=\roman*)]
\item We assume $\ondaq_{2}> q_{2}$. In this case the equilibrium $Q^{-}$ coincides with the solution $Q^{+}$, see Figure \subref*{fig:B2magg}, thus nothing happens.

\item We assume $\ondaq_{2}< q_{2}$. 
First of all we analyze the effects of a wave related only to the density $\rho$, i.e. we send a certain $\ondarho_{2}$ on road 2 keeping $w_{2}$ fixed. In Figure \subref*{fig:B2min} we show a possible solution given by the algorithm. Specifically we have
\begin{equation*}
	 q_{1}=\qS_{1},\qquad  q_{2}\geq\qS_{2}=\ondaq_{2}, \qquad \qS_{3}=\at q_{1}+\bt\ondaq_{2}\leq q_{3}, \qquad \qS_{4}=\aq q_{1}+\bq\ondaq_{2}\leq q_{4}.
\end{equation*}
We refer to the Appendix of \cite{dellemonache2018CMS} for the estimates of $\Delta\Gamma$, $\Delta\hh$ and $\Delta\tv_{Q}$ of \ref{rsp2} and \ref{rsp3}. Note that $\qS_{3}\geq\at q_{1}$ and $\qS_{4}\geq\aq q_{1}$. By \eqref{eq:w3T1} and \eqref{eq:w4T1} we have
	\begin{flalign*}
	|\wS_{3}-w_{3}|&\leq\frac{\bt|w_{1}-w_{2}|}{ q_{3}}|\ondaq_{2}- q_{2}|,\qquad
	|\wS_{4}-w_{4}|\leq\frac{\bt|w_{1}-w_{2}|}{ q_{4}}|\ondaq_{2}- q_{2}|&&\\
	\Rightarrow\Delta\tv_{w}(\tb)&=\left(\frac{\bt}{ q_{3}}+\frac{\bq}{ q_{4}}\right)|w_{1}-w_{2}||\ondaq_{2}- q_{2}|.
	&&
	\end{flalign*}

Next, we analyze the effects of a wave in $\rho$ and $w$, i.e. we send a couple $(\ondarho_{2},\ondaw_{2})$ on road 2 such that we still have $\ondaq_{2}< q_{2}$. The estimates on $\Gamma$, $\hh$ and $\tv_{Q}$ do not change, while for $\tv_{w}$ by \eqref{eq:wt3T1} and \eqref{eq:wt4T1} we have
\begin{flalign*}
	|\wS_{3}-w_{3}|&\leq \frac{\bt\ondaq_{2}}{\at q_{1}}|\ondaw_{2}-w_{2}|+\frac{\bt|w_{2}-w_{1}|}{ q_{3}}|\ondaq_{2}- q_{2}|&&\\
	|\wS_{4}-w_{4}|&\leq \frac{\bt\ondaq_{2}}{\at q_{1}}|\ondaw_{2}-w_{2}|+\frac{\bt|w_{2}-w_{1}|}{ q_{4}}|\ondaq_{2}- q_{2}|&&\\
	\tv_{w}(\tb+)&\leq \left(\frac{\bt}{\at}+\frac{\bq}{\aq}\right)\frac{\ondaq_{2}}{ q_{1}}|\ondaw_{2}-w_{2}|+\left(\frac{\bt}{ q_{3}}+\frac{\bq}{ q_{4}}\right)|w_{2}-w_{1}||\ondaq_{2}- q_{2}|
	&&\\
	\tv_{w}(\tb-)&=|\ondaw_{2}-w_{2}|&&\\
	\Rightarrow \Delta\tv_{w}(\tb)&\leq\left[\left(\frac{\bt}{\at}+\frac{\bq}{\aq}\right)\frac{\ondaq_{2}}{ q_{1}}-1\right]|\ondaw_{2}-w_{2}|+\left(\frac{\bt}{ q_{3}}+\frac{\bq}{ q_{4}}\right)|w_{2}-w_{1}||\ondaq_{2}- q_{2}|.&&
\end{flalign*}
Therefore \ref{rsp2}, \ref{rsp3} and \ref{rsp4} hold.

\end{enumerate}

\begin{figure}[h!]
\centering
\includegraphics[]{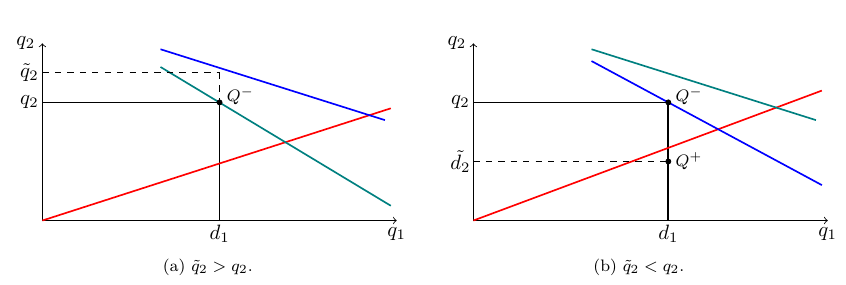}
\caption{Case B2: Wave on road 2.}
\label{fig:B2}
\end{figure}

\subsubsection{Case B3: Wave on road 3}\label{sec:B3}
We now consider a wave on road 3. The case of a wave on road 4 is analogous. Note that in this case we are not interested in what happens sending a wave in $\rho$ and $w$, since the changes in $w$ on the outgoing roads do not affect the Riemann solver. Hence, we only study waves which involve the density $\rho$ and keep $w_{3}$ fixed.

\begin{enumerate}[label=\roman*)]
\item We assume $\ondaq_{3}> q_{3}$. 
We send a certain $\ondarho_{3}$ on road 3 keeping $w_{3}$ fixed. In Figure \subref*{fig:B3magg} we show a possible solution given by the algorithm. Specifically we have
\begin{equation*}
	 q_{1}=\qS_{1},\qquad  q_{2}\leq\qS_{2}, \qquad \qS_{3}=\at q_{1}+\bt\qS_{2}\geq  q_{3}, \qquad \qS_{4}=\aq q_{1}+\bq\qS_{2}\geq q_{4}.
\end{equation*}
Note that
\begin{align*}
	\qS_{2}&=\frac{\qS_{3}-\at q_{1}}{\bt},\qquad
	 q_{2}=\frac{ q_{3}-\at q_{1}}{\bt}\\
	 q_{4}&=\aq q_{1}+\frac{\bq}{\bt}( q_{3}-\at q_{1}),\qquad
	\qS_{4}=\aq q_{1}+\frac{\bq}{\bt}(\qS_{3}-\at q_{1}).
\end{align*}
We compute
\begin{itemize}[label={*},leftmargin=*]
	\Item
	\begin{flalign*}
	\Gamma(\tb+) & = \qS_{1}+\qS_{2}= q_{1}+\qS_{2}, \qquad \Gamma(\tb-) =  q_{1}+ q_{2}&&\\
	\Rightarrow \Delta\Gamma(\tb) & = \qS_{2}- q_{2}=\frac{\qS_{3}- q_{3}}{\bt}.&&
	\end{flalign*}
	\Item
	\begin{flalign*}
	\hh(\tb+)&=\hh(\tb-)=\frac{ q_{1}}{\puno}&&\\
	\Rightarrow \Delta \hh(\tb)&=0.&&
	\end{flalign*}
	\Item
	\begin{flalign*}
	\tv_{Q}(\tb+) &= |\qS_{2}- q_{2}|+|\qS_{3}-\ondaq_{3}|+\bq|\qS_{2}- q_{2}| =\frac{1+\bq}{\bt}|\qS_{3}- q_{3}|+|\qS_{3}-\ondaq_{3}|&&\\
	\tv_{Q}(\tb-) &= |\ondaq_{3}- q_{3}|&&\\
	\Rightarrow\Delta\tv_{Q}(\tb) &=\frac{1+\bq}{\bt}|\qS_{3}- q_{3}|+|\qS_{3}-\ondaq_{3}|-|\ondaq_{3}- q_{3}|\leq\frac{2}{\bt}|\qS_{3}- q_{3}|.
	\end{flalign*}
	\item By \eqref{eq:w3T1} and \eqref{eq:w4T1}
	\begin{flalign*}
	|\wS_{3}-w_{3}|&=\frac{\at q_{1}|w_{1}-w_{2}|| q_{3}-\qS_{3}|}{\qS_{3} q_{3}}\leq\frac{\at\bt q_{1}|w_{1}-w_{2}|}{( q_{3})^{2}}|\Delta\Gamma(\tb)|
	&&\\
	|\wS_{4}-w_{4}|&=\frac{\aq q_{1}|w_{1}-w_{2}|| q_{4}-\qS_{4}|}{\qS_{4} q_{4}}\leq\frac{\aq\bq q_{1}|w_{1}-w_{2}|}{( q_{4})^{2}}|\Delta\Gamma(\tb)|
	&&\\
	\Rightarrow\Delta\tv_{w}(\tb)&\leq\left(\frac{\at\bt}{( q_{3})^{2}}+\frac{\aq\bq}{( q_{4})^{2}}\right) q_{1}|w_{1}-w_{2}||\Delta\Gamma(\tb)|.&&
	\end{flalign*}
\end{itemize}
Therefore \ref{rsp2} holds.

\item We assume $\ondaq_{3}< q_{3}$. 
We send a certain $\ondarho_{3}$ on road 3 keeping $w_{3}$ fixed. In Figure \subref*{fig:B3min} we show a possible solution given by the algorithm. Specifically we have
\begin{equation*}
	 q_{1}\geq\qS_{1},\qquad  q_{2}\geq\qS_{2}, \qquad \qS_{3}=\ondaq_{3}=\at\qS_{1}+\bt\qS_{2}\leq  q_{3}, \qquad \qS_{4}=\aq\qS_{1}+\bq\qS_{2}\leq q_{4}.
\end{equation*}
Note that 
\begin{align*}
	\qS_{1}&=\frac{\puno\ondaq_{3}}{\at\puno+\bt\pdue},\qquad
	 q_{1}=\frac{\puno\qstar_{3}}{\at\puno+\bt\pdue}\\
	\qS_{2}&=\frac{\pdue\ondaq_{3}}{\at\puno+\bt\pdue},\qquad
	 q_{2}=\frac{ q_{3}-\at q_{1}}{\bt}\\
	\qS_{4}&=\frac{\aq\puno+\bq\pdue}{\at\puno+\bt\pdue}\ondaq_{3}.
\end{align*}
We refer to the Appendix of \cite{dellemonache2018CMS} for the estimates of $\Delta\Gamma$, $\Delta\hh$ and $\Delta\tv_{Q}$ of \ref{rsp2} and \ref{rsp3}. We only observe that
\begin{align*}
	\Delta\Gamma(\tb) &=  \frac{\ondaq_{3}-\qstar_{3}}{\at\puno+\bt\pdue}+\frac{\qstar_{3}- q_{3}}{\bt}\\
	\Delta\hh(\tb) & = \frac{\qS_{1}- q_{1}}{\puno}= \frac{\ondaq_{3}-\qstar_{3}}{\at\puno+\bt\pdue}.
\end{align*}
By \eqref{eq:w3T3} and \eqref{eq:w4T3} we have
	\begin{flalign*}
	|\wS_{3}-w_{3}|& \leq\frac{\at\bt\puno|w_{1}-w_{2}|}{(\at\puno+\bt\pdue) q_{3}}(|\Delta\Gamma(\tb)|+|\Delta\hh(\tb)|)&&\\
	|\wS_{4}-w_{4}|&\leq\frac{\aq\bq\puno|w_{1}-w_{2}|}{(\aq\puno+\bq\pdue) q_{4}}(|\Delta\Gamma(\tb)|+|\Delta\hh(\tb)|)&&\\
	\Rightarrow\Delta\tv_{w}(\tb)&\leq\left(\frac{\at\bt}{ q_{3}(\puno\at+\pdue\bt)}+\frac{\aq\bq}{ q_{4}(\puno\aq+\pdue\bq)}\right)\puno|w_{1}-w_{2}|(|\Delta\Gamma(\tb)|+|\Delta\hh(\tb)|).&&
	\end{flalign*}
Therefore \ref{rsp2} and \ref{rsp3} hold.
\end{enumerate}

\begin{figure}[h!]
\centering
\includegraphics[]{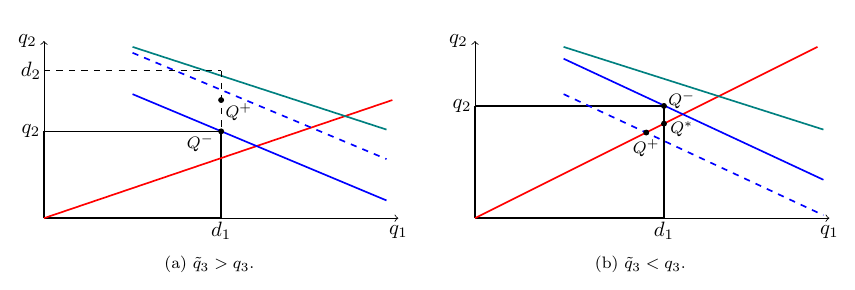}
\caption{Case B3: Wave on road 3.}
\label{fig:B3}
\end{figure}

\subsection{Case C} 
This case is verified when the equilibrium is defined by one of the straight lines which maximize the outgoing flux $\rettas_{3}$ or $\rettas_{4}$. Without loss of generality, we assume that the priority rule $\rettar$ first intersects the straight line $\rettas_{3}$. We study the effects produced by a single wave sent on each road.
\subsubsection{Case C1: Wave on road 1}\label{sec:C1}
Let us start with a wave on road 1.
\begin{enumerate}[label=\roman*)]
\item We assume $\ondaq_{1}> q_{1}$, see Figure \subref*{fig:C1magg}. The solution coincides with the equilibrium, thus nothing happens.

\item We assume $\ondaq_{1}< q_{1}$. 
First of all we analyze the effects of a wave related only to the density $\rho$, i.e. we send a certain $\ondarho_{1}$ on road 1 keeping $w_{1}$ fixed. In Figure \subref*{fig:C1min} we show a possible solution given by the algorithm. Specifically we have
\begin{equation*}
	 q_{1}>\qS_{1}=\ondaq_{1},\qquad  q_{2}\leq\qS_{2}, \qquad \qS_{3}=\at\ondaq_{1}+\bt\qS_{2}, \qquad \qS_{4}=\aq\ondaq_{1}+\bq q_{2}.
\end{equation*}
Note that $\qS_{3}\geq\at\ondaq_{1}$ and $\qS_{4}\geq\aq\ondaq_{1}$. Moreover, we observe that if $\qS_{2}< q_{2}$ then $\qS_{2}- q_{2}<0$, otherwise we define the angle $\beta=\arctan(|\ddue- q_{2}|/|\ondaq_{1}- q_{1}|)$ between the segments $|Q^{-}Q^{*}|$ and $|\widehat QQ^{-}|$ and we obtain $\qS_{2}- q_{2}=|\qS_{2}- q_{2}|\leq \tan\beta|\ondaq_{1}- q_{1}|$.
We compute
\begin{itemize}[label={*},leftmargin=*]
	\Item
	\begin{flalign*}
	\Gamma(\tb+) & = \qS_{1}+\qS_{2}=\ondaq_{1}+\qS_{2}, \qquad \Gamma(\tb-) =  q_{1}+ q_{2}&&\\
	\Rightarrow \Delta\Gamma(\tb) & = (\ondaq_{1}- q_{1})+(\qS_{2}- q_{2})<C |\ondaq_{1}- q_{1}|.
	\end{flalign*}
	\Item
	\begin{flalign*}
	\hh(\tb+)&=\frac{\ondaq_{1}}{\puno}, \qquad \hh(\tb-) = \frac{ q_{1}}{\puno}&&\\
	\Rightarrow \Delta \hh(\tb)&=\frac{\ondaq_{1}- q_{1}}{\puno}<0.&&
	\end{flalign*}
	\Item
	\begin{flalign*}
	\tv_{Q}(\tb+) &\leq 2|\qS_{2}- q_{2}|+|\ondaq_{1}- q_{1}|\leq 2(|\Delta\Gamma(\tb)|+|\ondaq_{1}- q_{1}|)&&\\
	\tv_{Q}(\tb-) &= |\ondaq_{1}- q_{1}|&&\\
	\Rightarrow\Delta\tv_{Q}(\tb) &\leq 2(|\Delta\Gamma(\tb)|+|\Delta\hh(\tb)|).  &&
	\end{flalign*}
	\item By \eqref{eq:w3T3} and \eqref{eq:w4T3} we have
	\begin{flalign*}
	|\wS_{3}-w_{3}|&\leq\frac{\bt|w_{1}-w_{2}|}{\ondaq_{1} q_{3}}(\ondaq_{1}+\ddue)(|\Delta\Gamma(\tb)|+|\Delta\hh(\tb)|)&&\\
	|\wS_{4}-w_{4}|&\leq\frac{\bq|w_{1}-w_{2}|}{\ondaq_{1} q_{4}}(\ondaq_{1}+\ddue)(|\Delta\Gamma(\tb)|+|\Delta\hh(\tb)|)&&\\
	\Rightarrow\Delta\tv_{w}(\tb)&\leq\left(\frac{\bt}{ q_{3}}+\frac{\bq}{ q_{4}}\right)\frac{|w_{1}-w_{2}|}{\ondaq_{1}}(\ondaq_{1}+\ddue)(|\Delta\Gamma(\tb)|+|\Delta\hh(\tb)|).&&
	\end{flalign*}
\end{itemize}

Next, we analyze the effects of a wave in $\rho$ and $w$, i.e. we send a couple $(\ondarho_{1},\ondaw_{1})$ on road 1 such that we still have $\ondaq_{1}< q_{1}$. The estimates on $\Gamma$, $\hh$ and $\tv_{Q}$ do not change, while for $\tv_{w}$ by \eqref{eq:wt3T3} and \eqref{eq:wt4T3} we have
\begin{flalign*}
	|\wS_{3}-w_{3}|&\leq|\ondaw_{1}-w_{1}|+\frac{\bt|w_{1}-w_{2}|}{\ondaq_{1} q_{3}}(\ondaq_{1}+\ddue)(|\Delta\Gamma(\tb)|+|\Delta\hh(\tb)|)&&\\
	|\wS_{4}-w_{4}|&\leq|\ondaw_{1}-w_{1}|+\frac{\bq|w_{1}-w_{2}|}{\ondaq_{1} q_{4}}(\ondaq_{1}+\ddue)(|\Delta\Gamma(\tb)|+|\Delta\hh(\tb)|)&&\\
	\tv_{w}(\tb+)&\leq 2|\ondaw_{1}-w_{1}|+\left(\frac{\bt}{ q_{3}}+\frac{\bq}{ q_{4}}\right)\frac{|w_{1}-w_{2}|}{\ondaq_{1}}(\ondaq_{1}+\ddue)(|\Delta\Gamma(\tb)|+|\Delta\hh(\tb)|)
	&&\\
	\tv_{w}(\tb-)&=|\ondaw_{1}-w_{1}|&&\\
	\Rightarrow \Delta\tv_{w}(\tb)&\leq|\ondaw_{1}-w_{1}|+\left(\frac{\bt}{ q_{3}}+\frac{\bq}{ q_{4}}\right)\frac{|w_{1}-w_{2}|}{\ondaq_{1}}(\ondaq_{1}+\ddue)(|\Delta\Gamma(\tb)|+|\Delta\hh(\tb)|).&&
\end{flalign*}

Therefore \ref{rsp2}, \ref{rsp3} and \ref{rsp4} hold.

\end{enumerate}

\begin{figure}[h!]
\centering
\includegraphics[]{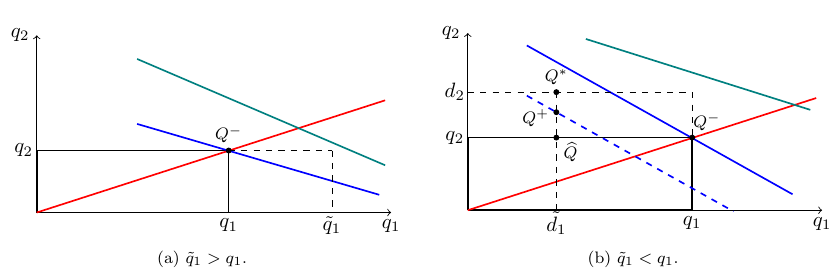}
\caption{Case C1: Wave on road 1.}
\label{fig:C1}
\end{figure}

\subsubsection{Case C2: Wave on road 2}\label{sec:C2}
Let us consider a wave on road 2.
\begin{enumerate}[label=\roman*)]
\item We assume $\ondaq_{2}> q_{2}$, see Figure \subref*{fig:C2magg}. The solution coincides with the equilibrium, thus nothing happens.

\item We assume $\ondaq_{2}< q_{2}$. 
First of all we analyze the effects of a wave related only to the density $\rho$, i.e. we send a certain $\ondarho_{2}$ on road 2 keeping $w_{2}$ fixed. In Figure \subref*{fig:C2min} we show a possible solution given by the algorithm. Specifically we have
\begin{equation*}
	 q_{1}\geq\qS_{1},\qquad  q_{2}\geq\qS_{2}=\ondaq_{2}, \qquad \qS_{3}=\at\qS_{1}+\bt\ondaq_{2}\leq q_{3}, \qquad \qS_{4}=\aq\qS_{1}+\bq\ondaq_{2}\leq q_{4}.
\end{equation*}
Note that $\qS_{3}\geq\bt\ondaq_{2}$ and $\qS_{4}\geq\bq\ondaq_{2}$.
We compute
\begin{itemize}[label={*},leftmargin=*]
	\Item
	\begin{flalign*}
	\Gamma(\tb+) & = \qS_{1}+\qS_{2}=\qS_{1}+\ondaq_{2}, \qquad \Gamma(\tb-) =  q_{1}+ q_{2}&&\\
	\Rightarrow \Delta\Gamma(\tb) & = (\qS_{1}- q_{1})+(\ondaq_{2}- q_{2}) <0.&&
	\end{flalign*}
	\Item
	\begin{flalign*}
	\hh(\tb+)&=\frac{\ondaq_{2}}{\pdue}, \qquad \hh(\tb-) = \frac{ q_{2}}{\pdue}&&\\
	\Rightarrow \Delta \hh(\tb)&=\frac{\ondaq_{2}- q_{2}}{\pdue}<0.&&
	\end{flalign*}
	\Item
	\begin{flalign*}
	\tv_{Q}(\tb+) &\leq 2|\qS_{1}- q_{1}|+|\ondaq_{2}- q_{2}| = 2|\Delta\Gamma(\tb)-(\ondaq_{2}- q_{2})|+|\ondaq_{2}- q_{2}|&&\\
	&\leq 2|\Delta\Gamma(\tb)|+3|\ondaq_{2}- q_{2}|&&\\
	\tv_{Q}(\tb-) &= |\ondaq_{2}- q_{2}|&&\\
	\Rightarrow\Delta\tv_{Q}(\tb) &\leq2(|\Delta\Gamma(\tb)|+|\Delta\hh(\tb)|).&&
	\end{flalign*}
	\item By \eqref{eq:w3T3} and \eqref{eq:w4T3} we have
	\begin{flalign*}
	|\wS_{3}-w_{3}|&\leq\frac{\at( q_{1}+\ondaq_{2})|w_{1}-w_{2}|}{\ondaq_{2} q_{3}}(|\Delta\Gamma(\tb)|+|\Delta\hh(\tb)|)&&\\
	|\wS_{4}-w_{4}|&\leq\frac{\aq( q_{1}+\ondaq_{2})|w_{1}-w_{2}|}{\ondaq_{2} q_{4}}(|\Delta\Gamma(\tb)|+|\Delta\hh(\tb)|)&&\\
	\Rightarrow\Delta\tv_{w}(\tb)&\leq\left(\frac{\at}{ q_{3}}+\frac{\aq}{ q_{4}}\right)\frac{( q_{1}+\ondaq_{2})|w_{1}-w_{2}|}{\ondaq_{2}}(|\Delta\Gamma(\tb)|+|\Delta\hh(\tb)|).
	&&
	\end{flalign*}
\end{itemize}

Next, we analyze the effects of a wave in $\rho$ and $w$, i.e. we send a couple $(\ondarho_{2},\ondaw_{2})$ on road 2 such that we still have $\ondaq_{2}< q_{2}$. The estimates on $\Gamma$, $\hh$ and $\tv_{Q}$ do not change, while for $\tv_{w}$ by \eqref{eq:wt3T3} and \eqref{eq:wt4T3} we have
\begin{flalign*}
	|\wS_{3}-w_{3}|&\leq|\ondaw_{2}-w_{2}|
	+\frac{\at( q_{1}+\ondaq_{2})|w_{1}-w_{2}|}{\ondaq_{2} q_{3}}(|\Delta\Gamma(\tb)|+|\Delta\hh(\tb)|)
	&&\\
	|\wS_{4}-w_{4}|&\leq|\ondaw_{2}-w_{2}|
	+\frac{\aq( q_{1}+\ondaq_{2})|w_{1}-w_{2}|}{\ondaq_{2} q_{4}}(|\Delta\Gamma(\tb)|+|\Delta\hh(\tb)|)
	&&\\
	\tv_{w}(\tb+)&\leq2|\ondaw_{2}-w_{2}|+\left(\frac{\at}{ q_{3}}+\frac{\aq}{ q_{4}}\right)\frac{( q_{1}+\ondaq_{2})|w_{1}-w_{2}|}{\ondaq_{2}}(|\Delta\Gamma(\tb)|+|\Delta\hh(\tb)|)
	&&\\
	\tv_{w}(\tb-)&=|\ondaw_{2}-w_{2}|&&\\
	\Rightarrow \Delta\tv_{w}(\tb)&\leq|\ondaw_{2}-w_{2}|+\left(\frac{\at}{ q_{3}}+\frac{\aq}{ q_{4}}\right)\frac{( q_{1}+\ondaq_{2})|w_{1}-w_{2}|}{\ondaq_{2}}(|\Delta\Gamma(\tb)|+|\Delta\hh(\tb)|).&&
\end{flalign*}
Therefore \ref{rsp2}, \ref{rsp3} and \ref{rsp4} hold.

\end{enumerate}

\begin{figure}[h!]
\centering
\includegraphics[]{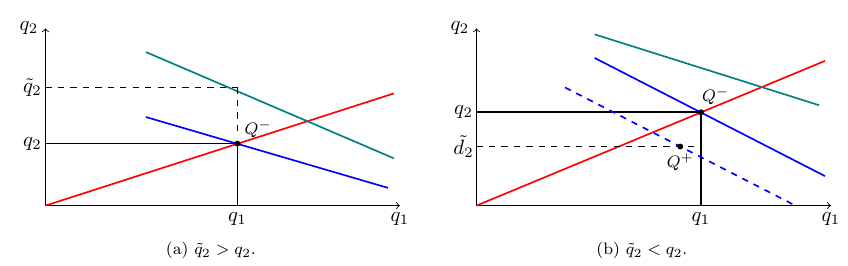}
\caption{Case C2: Wave on road 2.}
\label{fig:C2}
\end{figure}
 
\subsubsection{Case C3: Wave on road 3}\label{sec:C3}
We now consider a wave on road 3. The case of a wave on road 4 is analogous
\begin{enumerate}[label=\roman*)]
\item We assume $\ondaq_{3}> q_{3}$. We send a certain $\ondarho_{3}$ on road 3 keeping $w_{3}$ fixed. 
In Figure \subref*{fig:C3magg} we show a possible solution given by the algorithm.
Specifically we have
\begin{equation*}
	 q_{1}\leq\qS_{1},\qquad  q_{2}\leq\qS_{2}, \qquad \qS_{3}=\at\qS_{1}+\bt\qS_{2}\geq q_{3}, \qquad \qS_{4}=\aq\qS_{1}+\bq\qS_{2}\geq q_{4}.
\end{equation*}
Note that if $\ondaq_{3}\leq\pdue\duno/\puno$ then this case is similar to the case C3 with $\ondaq_{3}< q_{3}$. We have $\ondaq_{3}>\pdue\qS_{1}/\puno=\qstar_{3}$ with $\qS_{1}=\duno$. Moreover 
\begin{align*}
	 q_{1}&=\frac{\puno q_{3}}{\at\puno+\bt\pdue},\qquad
	\qS_{1}=\frac{\puno\qstar_{3}}{\at\puno+\bt\pdue}\\
	 q_{2}&=\frac{\pdue q_{3}}{\at\puno+\bt\pdue},\qquad
	\qS_{2}=\frac{\qS_{3}-\at\qS_{1}}{\bt}=\frac{\qS_{3}}{\bt}-\frac{\at\puno\qstar_{3}}{\bt(\at\puno+\bt\pdue)}.
\end{align*}
We compute
\begin{itemize}[label={*},leftmargin=*]
	\Item
	\begin{flalign*}
	\Gamma(\tb+) & = \qS_{1}+\qS_{2}, \qquad \Gamma(\tb-) =  q_{1}+ q_{2}&&\\
	\Rightarrow \Delta\Gamma(\tb) & = (\qS_{1}- q_{1})+(\qS_{2}- q_{2}).
	&&
	\end{flalign*}
	\Item
	\begin{flalign*}
	\hh(\tb+)&=\frac{\qS_{1}}{\puno}, \qquad \hh(\tb-) = \frac{ q_{1}}{\puno}&&\\
	\Rightarrow \Delta \hh(\tb)&=\frac{\qS_{1}- q_{1}}{\puno}=\frac{\puno(\qstar_{3}- q_{3})}{\at\puno+\bt\pdue}\leq\frac{\puno}{\at\puno+\bt\pdue}|\ondaq_{3}- q_{3}|.&&
	\end{flalign*}
	\Item
	\begin{flalign*}
	\tv_{Q}(\tb+) &\leq (1+\at+\aq)|\qS_{1}- q_{1}|+(1+\bt+\bq)|\qS_{2}- q_{2}|&&\\
	&=2|\qS_{1}- q_{1}|+2|\Delta\Gamma(\tb)-(\qS_{1}- q_{1})|
	\leq 4(|\Delta\Gamma(\tb)+\Delta\hh(\tb)|)
	&&\\
	\tv_{Q}(\tb-) &= |\ondaq_{3}- q_{3}|&&\\
	\Rightarrow\Delta\tv_{Q}(\tb) &\leq  4(|\Delta\Gamma(\tb)+\Delta\hh(\tb)|).&&
	\end{flalign*}
	\item By \eqref{eq:w3T3} and \eqref{eq:w4T3} we have
	\begin{flalign*}
	|\wS_{3}-w_{3}|&\leq\frac{\at\bt(\duno+\ddue)|w_{1}-w_{2}|}{( q_{3})^{2}}(|\Delta\Gamma|+|\Delta\hh(\tb)|)&&\\
	|\wS_{4}-w_{4}|&\leq\frac{\aq\bq(\duno+\ddue)|w_{1}-w_{2}|}{( q_{4})^{2}}(|\Delta\Gamma|+|\Delta\hh(\tb)|)&&\\	
	\Rightarrow\Delta\tv_{w}(\tb)&\leq\left(\frac{\at\bt}{( q_{3})^{2}}+\frac{\aq\bq}{( q_{4})^{2}}\right)(\duno+\ddue)|w_{1}-w_{2}|(|\Delta\Gamma|+|\Delta\hh(\tb)|).
	&&
	\end{flalign*}
\end{itemize}
Therefore \ref{rsp2} and \ref{rsp4} hold.

\item We assume $\ondaq_{3}< q_{3}$. 
We send a certain $\ondarho_{3}$ on road 3 keeping $\wS_{3}$ fixed. In Figure \subref*{fig:C3min} we show a possible solution given by the algorithm. Specifically we have
\begin{equation*}
	 q_{1}\geq\qS_{1},\qquad  q_{2}\geq\qS_{2}, \qquad \qS_{3}=\at\qS_{1}+\bt\qS_{2}=\ondaq_{3}\leq q_{3}, \qquad \qS_{4}=\aq\qS_{1}+\bq\qS_{2}\leq q_{4}.
\end{equation*}
We refer to the Appendix of \cite{dellemonache2018CMS} for the estimates of $\Delta\Gamma$, $\Delta\hh$ and $\Delta\tv_{Q}$ of \ref{rsp2} and \ref{rsp3}. Moreover, since both for $Q^{-}$ and $Q^{+}$ the solution is found with first step of the algorithm we have
\begin{align*}
	\wS_{3}&=w_{3} = \frac{\at\puno w_{1}+\bt\pdue w_{2}}{\at\puno+\bt\pdue}\\
	\wS_{4}&=w_{4} = \frac{\aq\puno w_{1}+\bq\pdue w_{2}}{\at\puno+\bt\pdue},	
\end{align*}
hence $\Delta\tv_{w}(\tb)=0$.
Therefore \ref{rsp2} and \ref{rsp3} hold.
\end{enumerate}

\begin{figure}[htpb!]
\centering
\includegraphics[]{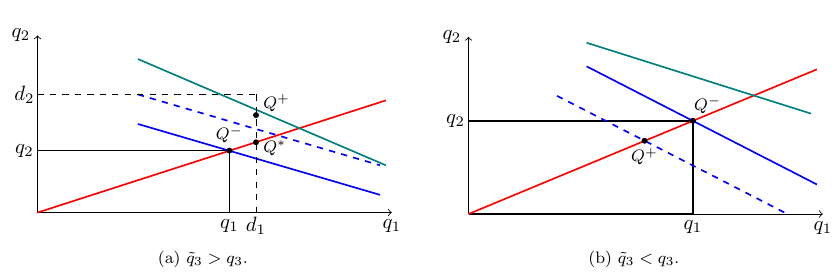}
\caption{Case C3: Wave on road 3.}
\label{fig:C3}
\end{figure}

\bibliographystyle{siam}
\bibliography{references_complete}

\begin{thebibliography}{10}

\bibitem{AwRascle2000}
{\sc A.~Aw and M.~Rascle}, {\em Resurrection of ``{S}econd {O}rder'' {M}odels
  of {T}raffic {F}low}, SIAM J. Appl. Math., 60 (2000), pp.~916--944.

\bibitem{Bressan2000}
{\sc A.~Bressan}, {\em Hyperbolic systems of conservation laws -- {T}he one
  dimensional Cauchy problem}, Oxford University Press, Oxford, 2000.

\bibitem{bressan2014flows}
{\sc A.~Bressan, S.~{\v{C}}ani{\'c}, M.~Garavello, M.~Herty, and B.~Piccoli},
  {\em Flows on networks: recent results and perspectives}, EMS Surveys in
  Mathematical Sciences, 1 (2014), pp.~47--111.

\bibitem{Dafermos2000}
{\sc C.~M. Dafermos}, {\em Hyperbolic conservation laws in continuum physics},
  Springer, Berlin, 2000.

\bibitem{dellemonache2018CMS}
{\sc M.~L. Delle~Monache, P.~Goatin, and B.~Piccoli}, {\em Priority-based
  {R}iemann solver for traffic flow on networks}, Commun. Math. Sci., 16
  (2018), pp.~185--211.

\bibitem{FanNHM14}
{\sc S.~Fan, M.~Herty, and B.~Seibold}, {\em Comparative model accuracy of a
  data-fitted generalized aw-rascle-zhang model}, Networks and Heterogeneous
  Media, 9 (2014), pp.~239--268.

\bibitem{FanSunPiccoliSeiboldWork2017}
{\sc S.~Fan, Y.~Sun, B.~Piccoli, B.~Seibold, and D.~B. Work}, {\em A
  {C}ollapsed {G}eneralized {A}w-{R}ascle-{Z}hang {M}odel and its {M}odel
  {A}ccuracy}, arXiv preprint arXiv:1702.03624,  (2017).

\bibitem{garavello2016models}
{\sc M.~Garavello, K.~Han, and B.~Piccoli}, {\em Models for vehicular traffic
  on networks}, vol.~9, American Institute of Mathematical Sciences (AIMS),
  Springfield, MO, 2016.

\bibitem{GaravelloPiccoli2006AwRascle}
{\sc M.~Garavello and B.~Piccoli}, {\em Traffic flow on a road network using
  the {A}w-{R}ascle {M}odel}, Commun. Part. Diff. Eq., 31 (2006), pp.~243--275.

\bibitem{GaravelloPiccoli2006}
{\sc M.~Garavello and B.~Piccoli}, {\em Traffic flow on networks}, American
  Institute of Mathematical Sciences, 2006.

\bibitem{garavello2006AIHP}
{\sc M.~Garavello and B.~Piccoli}, {\em Conservation laws on complex networks},
  Ann. Inst. H. Poincar\'{e} Anal. Non Lin\'{e}aire, 26 (2009), pp.~1925--1951.

\bibitem{GoettlichHertyMoutariWeissen2021}
{\sc S.~G\"{o}ttlich, M.~Herty, S.~Moutari, and J.~Weissen}, {\em Second-order
  traffic flow models on networks}, SIAM Journal on Applied Mathematics, 81
  (2021), pp.~258--281.

\bibitem{herty2006NHM}
{\sc M.~Herty, S.~Moutari, and M.~Rascle}, {\em Optimization criteria for
  modelling intersections of vehicular traffic flow}, Netw. Heterog. Media, 1
  (2006), pp.~275--294.

\bibitem{herty2006SIAM}
{\sc M.~Herty and M.~Rascle}, {\em Coupling conditions for a class of
  second-order models for traffic flow}, SIAM J. Math. Anal., 38 (2006),
  pp.~595--616.

\bibitem{HoldenRisebro1995}
{\sc H.~Holden and N.~H. Risebro}, {\em A mathematical model of traffic flow on
  a network of unidirectional roads}, SIAM J. Math. Anal., 26 (1995),
  pp.~999--1017.

\bibitem{lebacque2005first}
{\sc J.-P. Lebacque}, {\em First-order macroscopic traffic flow models:
  Intersection modeling, network modeling}, in Transportation and Traffic
  Theory. Flow, Dynamics and Human Interaction. 16th International Symposium on
  Transportation and Traffic TheoryUniversity of Maryland, College Park, 2005.

\bibitem{LebacqueMammarHajSalem2007}
{\sc J.-P. Lebacque, S.~Mammar, and H.~Haj-Salem}, {\em Generic second order
  traffic flow modelling}, in Transportation and Traffic Theory, Elsevier,
  2007, pp.~755--776.

\bibitem{LighthillWhitham1955}
{\sc M.~J. Lighthill and G.~B. Whitham}, {\em On kinematic waves {II}. {A}
  theory of traffic flow on long crowded roads}, Proc.~Roy.~Soc.~A, 229 (1955),
  pp.~317--345.

\bibitem{Richards1956}
{\sc P.~I. Richards}, {\em Shock {W}aves on the {H}ighway}, Operations
  Research, 4 (1956), pp.~42--51.

\bibitem{Zhang2002}
{\sc H.~M. Zhang}, {\em A non-equilibrium traffic model devoid of gas-like
  behavior}, Transp. Res. B, 36 (2002), pp.~275--290.

\end{thebibliography}

\end{document}